\documentclass[preprint,12pt]{elsarticle}

\usepackage{amsmath,amssymb,amsthm,a4wide,color}
\usepackage[dvipsnames]{xcolor}
\usepackage{hyperref}
\usepackage[utf8]{inputenc}
\usepackage[T1]{fontenc}
\usepackage{mathrsfs}
\usepackage{algorithm,algpseudocode}
\usepackage{graphicx,subcaption}
\usepackage{tikz}
\usetikzlibrary{shapes.misc}
\usepackage{mathtools}

\usepackage{tikz-3dplot}
\usepackage{pgfplots}
\pgfplotsset{compat=1.18}
\usepackage{amssymb}
\usepackage{amsmath}
\usepackage{comment}
\usepackage{enumitem}
\usepackage{bm}
\begin{document}
\newcommand{\nc}{\newcommand}
\nc{\RR}{\mathbb{R}}
\nc{\CC}{\mathbb{C}}

\nc{\mrm}{\mathrm}
\nc{\mA}{\mathrm{A}}
\nc{\mB}{\mathrm{B}}
\nc{\mL}{\mathrm{L}}
\nc{\mH}{\mathrm{H}}
\nc{\mJ}{\mathrm{J}}
\nc{\mV}{\mathrm{V}}
\nc{\mW}{\mathrm{W}}
\nc{\Id}{\mathrm{Id}}
\nc{\mT}{\mathrm{T}}
\nc{\mS}{\mathrm{S}}
\nc{\mQ}{\mathrm{Q}}
\nc{\mX}{\mathrm{X}}
\nc{\mbH}{\mathbb{H}}
\nc{\mbV}{\mathbb{V}}
\nc{\mbX}{\mathbb{X}}

\nc{\Th}{\mathrm{T}_{h}}
\nc{\Vh}{\mathrm{V}_{h}}
\nc{\Rh}{\mathrm{R}_{h}}
\nc{\mR}{\mathrm{R}}
\nc{\VhO}{\mathrm{V}_{h,0}}
\nc{\mbVh}{\mathbb{V}_{h}}
\nc{\calTh}{\mathcal{T}_{h}}
\nc{\Ph}{\mathrm{P}_{h}}

\nc{\mP}{\mathrm{P}}

\nc{\infsup}{\mathop{\mrm{infsup}}}
\nc{\rref}{\textsc{r}}

\nc{\bp}{\boldsymbol{p}}
\nc{\bff}{\boldsymbol{f}}
\nc{\bg}{\boldsymbol{g}}
\nc{\by}{\boldsymbol{y}}
\nc{\bx}{\boldsymbol{x}}
\nc{\bz}{\boldsymbol{z}}
\nc{\bn}{\boldsymbol{n}}
\nc{\bu}{\boldsymbol{u}}
\nc{\bv}{\boldsymbol{v}}
\nc{\bq}{\boldsymbol{q}}
\nc{\bb}{\boldsymbol{b}}
\nc{\bbeta}{\boldsymbol{\beta}}
\nc{\bxi}{\boldsymbol{\xi}}

\nc{\bfp}{\mathbf{p}}
\nc{\bfff}{\mathbf{f}}
\nc{\bfy}{\mathbf{y}}
\nc{\bfx}{\mathbf{x}}
\nc{\bfn}{\mathbf{n}}
\nc{\bfu}{\mathbf{u}}
\nc{\bfv}{\mathbf{v}}
\nc{\bfq}{\mathbf{q}}
\nc{\bfg}{\mathbf{g}}

\nc{\dir}{\textsc{d}}
\nc{\neu}{\textsc{n}}

\nc{\Ploc}{\mrm{P}_\mathrm{loc}}
\nc{\loc}{\mathrm{loc}}
\nc{\supp}{\mathrm{supp}}
\nc{\dof}{\mrm{dof}}
\nc{\DOF}{\mathbb{D}\mrm{of}}

\nc{\llangle}{\langle\!\langle}
\nc{\rrangle}{\rangle\!\rangle}
\nc{\lbr}{\lbrack}
\nc{\rbr}{\rbrack}

\nc{\ctru}{\mathfrak{u}}
\nc{\ctrv}{\mathfrak{v}}
\nc{\ctrw}{\mathfrak{w}}
\nc{\ctrp}{\mathfrak{p}}
\nc{\ctrq}{\mathfrak{q}}
\nc{\ctrf}{\mathfrak{f}}

\nc{\bs}{\boldsymbol{s}}

\renewcommand{\dim}{\mrm{dim}}
\renewcommand{\span}{\mathop{\mrm{span}}}

\nc{\bfA}{\mathbf{A}}
\nc{\bfB}{\mathbf{B}}
\nc{\bfT}{\mathbf{T}}
\nc{\bfQ}{\mathbf{Q}}

\nc{\dsp}{\displaystyle}

\newtheorem{defn}{Definition}[section]
\newtheorem{lem}[defn]{Lemma}
\newtheorem{thm}[defn]{Theorem}
\newtheorem{prop}[defn]{Proposition}
\newtheorem{cor}[defn]{Corollary}
\newtheorem{definition}[defn]{Definition}
\newtheorem{rem}[defn]{Remark}

\newtheorem{example}{Example}

\newtheorem{asum}{Assumption}
\newtheorem{cond}[asum]{Condition}

\begin{frontmatter}

\title{Hierarchical matrix approximability of inverse\\
of convection dominated finite element matrices}
  
\author[ifp]{A. Saunier}
\author[ifp]{L. Agelas}
\author[ifp]{A. Anciaux Sedrakian}
\author[ifp]{I. Ben Gharbia}
\author[poems]{X. Claeys}

\affiliation[ifp]{organization={IFP Énergies nouvelles},
                  addressline={1-4 Av. du Bois Préau}, 
                  city={Rueil-Malmaison}, 
                  postcode={92852},
                  country={France}}

\affiliation[poems]{organization={POEMS, CNRS, Inria, ENSTA, Institut Polytechnique de Paris},
                    city={Palaiseau}, 
                    postcode={91120}, 
                    country={France}}

\begin{abstract}
Several researchers have developed a rich toolbox of matrix compression techniques that exploit structure and redundancy in large matrices. Classical methods such as the block low-rank (BLR) format \cite{BLR1} and the Fast Multipole Method (FMM) \cite{Greengard} make it possible to
manipulate otherwise intractable systems by representing them in a reduced form.
Among the most sophisticated tools in this area are hierarchical matrices \cite{MR1981528, MR2767920, MBebendorf}(H-matrices), which exploit local properties of the underlying kernel or operator to approximate matrix blocks by low-rank factors, organized in a recursive hierarchy. Compared to simpler methods like BLR, H-matrices offer a more flexible and scalable framework, yielding nearly linear complexity in both storage and computation. Further extensions, such as H2-matrices \cite{MR2767920}
and Hierarchically Semi-Separable matrices (HSS) \cite{HSS1}, achieve even greater efficiency through the use of nested bases and structured factorizations.
Hierarchical matrix techniques, originally developed for boundary integral equations \cite{MR2767920,MR1981528,MR1993936}, have recently been applied to matrices stemming from the discretization of advection-dominated problems \cite{MR3422448, MR2606959}. However, their effectiveness is limited by the loss of coercivity induced by convection phenomena, where traditional methods fail. Initial work by Le Borne \cite{MR2011612} addressed this by modifying the admissibility criterion for structured grids with constant convection, but challenges remain for more general grids and advection fields.
In this work, we propose a novel partitioning strategy based on "convection tubes", clusters aligned with the convection vector field. This method does not require a structured grid or constant convection, overcoming the limitations of previous approaches. We present both theoretical analyses and numerical experiments, that demonstrate the efficiency and robustness of our method for convection-dominated PDEs on unstructured grids. The approach builds on a P\'eclet-robust Caccioppoli inequality, crucial for handling convection-dominated problems.
\end{abstract}

\end{frontmatter}
\section*{Introduction}

Hierarchical matrix acceleration techniques have proven to be powerful for the numerical treatment of dense matrices arising from discretization of boundary integral operators \cite{MR1981528,MR1993936,MR2767920}. A natural extension of this work has been the exploration of their application to inverse of sparse matrices resulting from the discretization of partial differential equations (PDEs). Early works, such as \cite{MR1694265}, have recognized the potential of hierarchical matrices as an algebraic structure suitable for the numerical treatment of elliptic problems. However, significant progress in this area was made only later, with error estimates under the assumption of strong coercivity of the underlying PDEs \cite{MR1993936,MR3422448,MR2606959}.

The introduction of hierarchical matrices for advection-dominated problems, where coercivity is typically lost, has been more challenging. In these cases, naive applications of hierarchical matrix compression often fail to yield efficient approximations \cite{MR2606959}. The issue of applying these techniques in such settings was first addressed by Le Borne in \cite{MR2236673,MR2011612}, who proposed heuristic modifications to the admissibility criterion for hierarchical matrices to handle convection-dominated problems. However, this approach was limited to structured grids with constant convection fields aligned with the grid axes.

While progress has been made in extending hierarchical matrices to more general settings \cite{MR2876448, MR3679927}, the applicability of traditional admissibility criteria remains an open challenge. Recent contributions, like \cite{TY}, have proposed advanced partitioning strategies that improve the compression efficiency for such problems. These strategies, however, still face difficulties in handling general advection fields, particularly in unstructured grid settings.

In this work, we tackle convection-dominated problems using hierarchical matrix techniques. Our approach combines theoretical contributions and numerical validation, and can be summarized as follows:

\begin{itemize}
\item We aim to follow the analytical framework developed by Börm \cite{MR2606959} and by Melenk, Faustmann, and Praetorius \cite{MR3422448}. To make this strategy applicable to convection-dominated settings, we need a \textbf{Péclet-robust Caccioppoli inequality} for local finite-dimensional approximation spaces.
\item We prove this inequality in the constant convection case by introducing clusters aligned with the flow direction (“tube clusters”), which leads to bounds independent of the Péclet number.
\item We extend the analysis to non-constant, non-vanishing convection fields by means of locally straightening diffeomorphisms.
\item These theoretical results motivate a novel clustering strategy based on tube clusters that encompass the physic of the problem.
\item Numerical experiments confirm that the resulting $\mathcal{H}$-LU factorization achieves Péclet-robust accuracy with quasi-linear complexity.
\end{itemize}

The rest of this article is organized as follows. In Section \ref{section1}, we begin by presenting the model problem. In Section \ref{section2}, we treat the case of constant advection and show that for a specific geometry of clusters ("tube cluster") one can achieve a Péclet-robust Caccioppoli estimate which translate in a local approximation result. In Section \ref{section3} we extend this result to non constant advection using a deformation argument and we propose a method to produce a method that produces a suitable cluster tree. Finally in Section \ref{section4} we provide numerical experiments that test the accuracy of the $\mathcal{H}$-LU factorization based on the tube cluster tree for different convection fields, examining the influence of the diffusion parameter. To conclude, in Section \ref{section5}, we observe that our method appears to be Péclet-robust, allowing $\mathcal{H}$-LU factorizations to be performed in quasi-linear time while maintaining a bounded error, independently of the diffusion parameter. We also discuss the limitations induced by our partitioning strategy, which inherently imposes a non-arbitrary and relatively large lower bound on the size of the cluster tree leaves.

\section{Problem setting}
\label{section1}

We consider a bounded open polyhedral domain $\Omega\subset \RR^d$ for
some dimension $d\geq 1$ (typically $d=2$ or $3$). We also consider a vector field $\bb\in C^1(\RR^{d})^d$, $\beta\in
\mL^{\infty}(\Omega)$ with $\beta(\bx)\geq \beta_0>0\;\forall \bx\in\Omega$
and a source term $f\in \mL^{2}(\Omega)$, and we
study the boundary value problem consisting in looking for $u\in
\mH^1_0(\Omega)$ such that
\begin{equation}
  \begin{aligned}
     -\epsilon\Delta u +& \bb\cdot\nabla u + \beta u = f\quad \mrm{in}\;\Omega,\\
    & u=0 \quad \mrm{on}\;\partial\Omega.
  \end{aligned}
\end{equation}
  \label{StudiedPb1}

 In this problem, the parameter $\epsilon>0$ measures the importance of
the diffusion term in the equation. We are interested in the regime
where the so-called Péclet parameter $1/\epsilon$ is large.
As $\epsilon\to 0$, this equation turns into a transport equation. The
problem above admits the variational formulation: find $u\in \mH^{1}_0(\Omega)$
such that $a(u,v) = \ell(v)\;\forall v\in \mH^{1}_0(\Omega)$ with
the right-hand side $\ell(v):= \int_{\Omega} fv\,d\bx$ and
the sesquilinear form 
\begin{equation}
  a(u,v) := \int_{\Omega}\epsilon \nabla u \cdot\nabla
  v + v \bb\cdot\nabla u
  + \beta\,uv\,d\bx .
\label{prob_var}
\end{equation}
Let us suppose that 
\begin{equation}
    \mrm{div}(\bb(\boldsymbol{x}))/2-\beta(\boldsymbol{x}) < 0 \text{ for all } \boldsymbol{x}\in \Omega .
\label{conditionbeta}
\end{equation}
Then it is a clear consequence of Poincaré's inequality that
the sesquilinear form $a(\cdot,\cdot)$ is strongly coercive i.e.
\begin{equation}\label{Coercivity}
   a(v,v)\geq \epsilon \Vert \nabla v\Vert_{\mL^{2}(\Omega)}^2\quad \forall v\in \mH^{1}_0(\Omega).
\end{equation}

Let $\mathbb{T}_h(\Omega)$ refer to a quasi-uniform simplicial mesh of $\Omega$.
We consider the discretization of \eqref{StudiedPb1} by means of a
standard $\mathbb{P}_1-$Lagrange finite element scheme constructed
over $\mathbb{T}_h(\Omega)$. Define $\Vh(\Omega):=\{
v\in C^0(\overline{\Omega}):
v\vert_{\tau}\in\mathbb{P}_1(\tau)\;\forall \tau\in
\mathbb{T}_h(\Omega)\}$. The bilinear form $a(\cdot,\cdot)$ induces a linear
operator $\mA_h:\Vh(\Omega)\to \Vh(\Omega)^*$ defined by
\begin{equation}
  \langle \mA_h(u_h),v_h\rangle :=a(u_h,v_h)\quad \forall u_h,v_h\in\Vh(\Omega).
\end{equation}
It is an obvious consequence of the coercivity of $a(\cdot,\cdot)$
based on Lax-Milgram's lemma that $\mA_h$ is invertible. While the
operator $\mA_h$ is sparse in the sense that $\langle
\mA_h(u_h),v_h\rangle = 0$ as soon as
$\mrm{supp}(u_h)\cap\mrm{supp}(v_h)$ has vanishing Lebesgue measure,
its inverse $\mA_h^{-1}$ does not enjoy such a locality property.

The purpose of the present contribution is to investigate the
approximability of $\mA^{-1}_h$ by means of hierarchical
compression. While such a question was already investigated in
\cite{MR2606959,MR3422448}, we are particularly interested
in establishing an approximability result that is robust in the high
P\'eclet number regime $\epsilon\to 0$.

As was detailed in \cite{MR2011612,MR2236673}, a straightforward naive
application of standard hierarchical matrix approximation strategy
as presented in \cite{MR1981528} will fail to properly compress the
operator $\mA^{-1}_h$. This is related to the transport phenomenon
taking place as $\epsilon\to 0$. To circumvent this issue, we shall
modify the partitioning strategy, taking inspiration from the approaches
that were developed for the Helmholtz equation in \cite{MR2876448,MR3679927}.

\section{The case of constant advection}
\label{section2}
To gain a clear understanding of the objects we need to manipulate, we will begin by considering the constant case. Multiplying the equation by a constant factor and rotating the coordinate axes if necessary, we can assume without loss of generality that $\bb = (1, 0, \dots, 0)$. The proof strategy in \cite{MR1993936, MR2606959, MR3422448} consistently employs a Caccioppoli-type inequality as a key ingredient in establishing the low-rank approximability of admissible interactions. Our aim here is to explore the appropriate definition of admissibility for distant interactions, such that the Caccioppoli inequality holds independently of the Péclet number.
\subsection{P\'eclet-uniform Caccioppoli inequality}
\quad\\
In the following we define on $\Omega$ the $\mL^2$ scalar product as $\langle f , g\rangle_{\mL^2(\Omega)} = \int_\Omega fg d\bx, \; \forall f,g\in \mL^2(\Omega)$, the $\mL^2$-norm as $\Vert f\Vert ^2_{\mL^2(\Omega)} = \langle f, f\rangle_{\mL^2(\Omega)}$ and the $\mL^\infty$-norm as $\Vert f\Vert_{\mL^\infty(\Omega)} =  \sup_{x\in\Omega}\vert f(x)\vert$. For any domain $D\subset\mathbb{R}^d$ we take $\mathbf{n}_D$ to be the unit vector field normal to $\partial D$ pointing toward the exterior of $D$ . We take a bounded Lipschitz open set $\omega\subset \Omega$ and $\eta\in
\mrm{W}^{1,\infty}(\omega) = \{ v\in \mL^\infty(\omega) | \; \partial^\alpha_x  v \in \mL^1(\omega) \; \forall\alpha\in \mathbb{N}^d,\; \vert \alpha \vert \leq 1, \forall x \in \omega , \Vert x \Vert = 1\}$ satisfying $\mrm{supp}(\eta)\subset \omega$
and $\eta(\bx)\geq 0$ for all $\bx\in \RR^d$.  Then for any $u\in
\mH^{1}(\omega)$ we have 
\[\begin{aligned}
    \Vert \nabla (\eta u) \Vert_{\mL^2(\omega)}^2 &= \langle \nabla u , \nabla (\eta^2u)\rangle_{\mL^2(\omega)}-\langle \nabla u , \eta u  \nabla \eta \rangle_{\mL^2(\omega)} + \langle u\nabla \eta , \nabla( \eta u) \rangle_{\mL^2(\omega)} \\
    &= \langle \nabla u , \nabla (\eta^2 u) \rangle_{\mL^2(\omega)} + \langle u\nabla \eta , u\nabla \eta\rangle_{\mL^2(\omega)} \\
    &= \varepsilon^{-1}a(u , \eta^2  u) - \varepsilon^{-1}\int_{\omega} (\bb\cdot (\nabla u) \eta^2 u +\beta \eta^2 u^2) d\bx + \Vert u\nabla \eta \Vert_{\mL^2(\omega)}^2
\end{aligned} \] 
Because we assumed that there exists $\beta_0 > 0$ such that $\beta(x)\geq \beta_0$ $\forall x\in \Omega$, using $\mrm{div}(\bb/2)-\beta <0$ we get
\begin{equation}
     \varepsilon \Vert \nabla(  \eta u )\Vert_{\mL^2(\omega)}^2 \leq   a(u, \eta^2 u ) + \int_{\omega} \eta u ^2 \bb \cdot \nabla \eta d\bx + \varepsilon \Vert u\nabla \eta \Vert_{\mL^2(\omega)}^2
     \label{ineq0}
    \end{equation}
From this inequality we would like to deduce a Caccioppoli inequality
that is robust at large P\'eclet numbers i.e. for $\epsilon\to 0$.
Define 
\begin{equation}
    \mathcal{H}(\omega):= \{ v\vert_{\omega}, \; v\in \mH^1_0(\Omega) \; \mrm{ and } \;a(v,\varphi) = 0\;\forall \varphi\in \mH^{1}_0(\omega)\}
\end{equation}
which is the space of functions that are harmonic over $\omega$ with respect to the
sesquilinear form $a(\cdot,\cdot)$. A P\'eclet-robust Caccioppoli
inequality can be obtained with a clever choice of the cut-off
function $\eta$. Indeed, using \eqref{ineq0}, we have \eqref{CaccioppoliInequality11}

\begin{equation}\label{CaccioppoliInequality11}
  \begin{aligned}
    & \Vert \nabla (\eta u)\Vert_{\mL^{2}(\omega)}^2\leq 
    \Vert \nabla\eta \Vert_{\mL^{\infty}(\omega)}^2\Vert u\Vert_{\mL^{2}(\omega)}^2\\
    & \mrm{if}\;\; \bb\cdot \nabla \eta\leq 0\;\; \mrm{and}\;\; u\in \mathcal{H}(\omega).
  \end{aligned}
\end{equation}

Let us decompose the coordinate system like $\bx = (x_1,\tilde{x})$ and assume that
$\eta$ only depends on $\tilde{x}$ so that $\bb\cdot\nabla \eta = 0$ since we assumed
that $\bb = (1,0,\dots,0)$ rotating the coordinate axis if necessary. In order to satisfy the condition $\bb \cdot \nabla \eta = 0$, the function $\eta$ must be independent of the variable \( x_1 \). As a consequence, its support must extend across the entire domain \( \Omega \) in the direction of \( \bb \), reaching both the inflow and outflow boundaries.

\quad\\
Consider a bounded Lipschitz open subset $\tilde{\omega}\subset
\RR^{d-1}$ and $\tilde{\omega}_{\delta}:=\{\tilde{\bx}\in \RR^{d-1},
\mrm{dist}(\tilde{\bx},\tilde{\omega})\leq \delta\}$, and choose $\omega_\delta =
(\RR\times\tilde{\omega}_{\delta})\cap \Omega$ and $\omega =
(\RR\times\tilde{\omega})\cap \Omega$. Set $\chi(t) = 1$ if $t\leq
0$, $\chi(t) = 1-t$ for $0\leq t\leq 1$ and $\chi(t) = 0$ for $t\geq
1$, and take $\eta(\bx) =
\chi(\mrm{dist}(\tilde{\bx},\tilde{\omega})/\delta)$ for $\bx =
(x_1,\tilde{\bx})$. Then $\eta\in \mrm{W}^{1,\infty}(\RR^d)$ and $\Vert
\nabla \eta\Vert_{\mL^\infty(\RR^d)} \leq  \delta^{-1}$ and $\mrm{supp}(\eta)\subset
\RR\times \tilde{\omega}_{\delta}$ and $\eta(\bx) = 1$ for $\bx\in
\RR\times \tilde{\omega}$ as illustrated in Figure \ref{fig:nestdomain}.
\begin{figure}
    \centering
    \begin{tikzpicture}
    \draw(0,0)[purple] rectangle (4,2) ; 
    \draw [thick, purple](0,0)  --(0,2) ;
    \draw[purple] (0,1) node[left]{$\tilde{\omega}$};
    \draw[  blue] (0,0) --(0, -0.7) ; 
    \draw[  blue](0,2)--(0, 2.7) ;
    \draw[blue] (0,2.5) node [left]{$\tilde{\omega}_\delta$} ; 
    \draw [blue](0, 2.7)--(4,2.7)--(4, -0.7)--(0, -0.7) ;
    \draw[blue] (3,2.4) node{$\omega_\delta$} ;
    \draw[purple] (2,1) node{$\omega$} ; 
    \draw[->](1,0.5)--(1.4, 0.5) ; 
    \draw(1.2,0.3) node{$\bb$} ;

    \draw [thick, purple](6,0)  --(6,2) ;
    \draw[  blue] (6,0) --(6, -0.7) ; 
    \draw[  blue](6,2)--(6, 2.7) ;
    \draw[green](7,0)--(7,2) ; 
    \draw [green](6,2.7)--(7,2) ; 
    \draw[green](6,-0.70)--(7,0) ;
    \draw(6,-0.7)--(7.2,-0.7) ;
    \draw[green] (6.7, 2.5) node{$\eta$}; 
    \draw( 6,-0.7) node[below]{$0$} ; 
    \draw(7,-0.7) node[below]{$1$} ; 
    \draw (7.7,1.2) node{$\bigotimes\bb$};
    \draw [purple](6, 1) node[left]{$\tilde{\omega}$} ; 
    \draw[blue]( 6,2.3) node[left]{$\tilde{\omega}_\delta$} ; 
    
    \end{tikzpicture}
    \caption{Cutoff function $\eta$ on nested domains}
    \label{fig:nestdomain}
\end{figure}
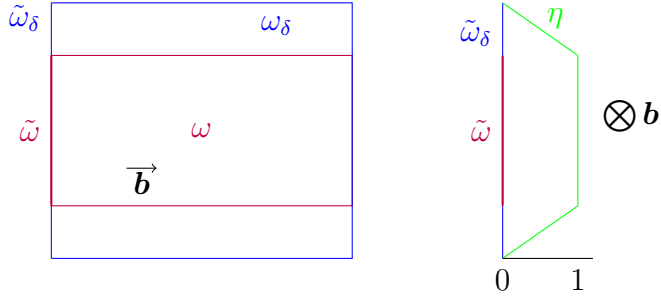
Applying \eqref{CaccioppoliInequality11} for
this particular setting yields

\begin{equation}\label{CaccioppoliEstimate1}
  \Vert \nabla u\Vert_{\mL^{2}(\omega)}^2\leq \frac{1}{\delta^2}\Vert u\Vert_{\mL^{2}(\omega_\delta)}^2\quad\forall u\in \mathcal{H}(\omega_\delta).
\end{equation}
It is remarkable that this inequality does not involve the parameter
$\epsilon$ which quantifies the P\'eclet number.  Of course, this was
obtained at the price of considering subsets $\omega_\delta$ and $\omega$ that are
elongated in the direction of the convection field $\bb$.

\begin{definition}\quad\\
We say that a domain $\tau\subset\Omega\subset \mathbb{R}^d$ is a tube cluster (implicitly a $\bb$-tube cluster) if there exists a bounded Lipschitz domain $\tau_\perp\in\mathbb{R}^{d-1}$ such that $\tau = (\mathbb{R}\times\tau_\perp)\cap\Omega$. In the following, for any tube cluster $\tau$, $\tau_\perp$ refers to the hyperplane of normal $\bb$. 

For a tube cluster $\tau=(\mathbb{R}\times \tau_\perp)\cap \Omega$ and $\delta>0$ we define the enlarged tube $\tau_\delta:=(\mathbb{R}\times \sigma)\cap \Omega$ where $\mrm{dist}(\tau_\perp,\partial \sigma)=\delta$ and $\tau_\perp\subset \sigma$.
\label{tube-cluster}
\end{definition}

\begin{rem}
    One should be careful in interpreting the geometry of $\tau_\delta$: it is an \emph{anisotropic inflation} of $\tau$. More precisely, we have $
\mathrm{dist}(\tau, \partial \tau_\delta) = 0,$ since both domains share the same extent along the $\bb$ direction. However, when projected onto the space orthogonal to $\bb$, the enlargement is visible: $\mathrm{dist}(\tau_\perp, (\tau_\delta)_\perp) = \delta > 0$.

This reflects the fact that $\tau_\delta$ is obtained by thickening $\tau$ only in directions perpendicular to the vector field.
\end{rem}

\begin{lem}\label{PoincareApprox}\quad\\
 Let $\omega\subset \Omega$ convex and $\mathcal{Z}$ referring to a closed subspace of
  $\mathrm{L}^{2}(\omega)$. For any $\ell\geq 1$ there
  exists a subspace $\mathcal{V}\subset \mathcal{Z}$ with
  $\mathrm{dim}(\mathcal{V})\leq
  \ell^{d}$ such that
  \begin{equation*}
    \inf_{v\in \mathcal{V}}\Vert u-v\Vert_{\mathrm{L}^{2}(\omega)}\leq \frac{\sqrt{d}}{\pi}
    \Big(\frac{\mathrm{diam}(\omega)}{\ell}\Big)
    \Vert \nabla u\Vert_{\mathrm{L}^{2}(\omega)}\quad \forall u\in \mathcal{Z}\cap\mathrm{H}^{1}(\omega). 
  \end{equation*}
  
\end{lem}
\begin{proof}
\label{proof:PW}
Let $\omega$ be a tube cluster, it is a bounded domain. We take $\boldsymbol{x}_0\in\mathbb{R}^d$ such that $Q:=\{\boldsymbol{x}\in \mathbb{R}^d | \ \Vert \boldsymbol{x}-\boldsymbol{x}_0\Vert \leq \frac{1}{2}\mathrm{diam}(\omega) \}$ contains $\omega$, it is called a bounding box. We fix $\ell\in\mathbb{N}$ and define a regular cartesian subdivision $(Q_i)_{i\in\mathcal{J}}$ of $Q$ by dividing each coordinate direction into $\ell$ equal interval of length $\text{diam}(\omega)/\ell$ such that $\# \mathcal{J} :=k= \ell^d$.
Each subdomain is defined as \( \omega_i := Q_i \cap \omega \).  
By construction, each \( \omega_i \) is contained in a cube of side length \( \frac{1}{\ell} \mathrm{diam}(\omega) \), and therefore satisfies the estimate \[ \mathrm{diam}(\omega_i) \leq  \frac{\sqrt{d}}{\ell}\mathrm{diam}(\omega).\]On each $\omega_i$ we use the Poincar\'e Wirtinger's inequality with a suitable constant approximation $u_i = \frac{1}{|\omega_i|}\int_{\omega_i}ud\boldsymbol{x}$ and we get
\[ \int_{\omega_i}|u-u_i|^2d\boldsymbol{x} \leq \frac{\mathrm{diam}(\omega_i)^2}{\pi^2} \int_{\omega_i}|\nabla u|^2 d\boldsymbol{x}.\] If we define $\mathcal{W}_k=\{ v\in\mathrm{L}^2(\omega) | \;v|_{\omega_i}\in \mathbb{R}\}$ and take $\Pi(u)\in\mathcal{W}_k$ where $\Pi(u)|_{\omega_i} = u_i \quad \forall i \in \mathcal{I}$, we have 
\begin{equation*}
    \begin{aligned}
        \Vert u-\Pi(u)\Vert^2 _{\mathrm{L}^2(\omega)}&= \sum_{i=1}^k \Vert u-u_i \Vert^2_{\mathrm{L}^2(\omega_i)}
        \leq \frac{d}{\pi^2 } \sum_{i=1}^k \frac{\mathrm{diam}(\omega)^2}{\ell^2}\Vert \nabla u\Vert^2_{\mathrm{L}^2(\omega_i)}\\
        &\leq \frac{d}{\pi^2 }\frac{\mathrm{diam}(\omega)^2}{\ell^2}\Vert \nabla u \Vert^2_{\mathrm{L}^2(\omega)}.
    \end{aligned}
\end{equation*} 
Finally, consider the $\mathrm{L}^{2}(\omega)-$orthogonal projection $\mathrm{P}:\mathrm{L}^{2}(\omega)\to \mathcal{Z}$
and set $\mathcal{V}:=\mathrm{P}(\mathcal{W}_k)$ so that $\mathrm{dim}(\mathcal{V})\leq \mathrm{dim}(\mathcal{W}_k)$.
Then for any $u\in \mathcal{Z}\cap \mathrm{H}^{1}(\omega)$ we have
$\Vert u-\mathrm{P}\cdot\Pi(u)\Vert_{\mathrm{L}^{2}(\omega)} = \Vert \mathrm{P}(u-\Pi(u))\Vert_{\mathrm{L}^{2}(\omega)}\leq
\Vert u-\Pi(u)\Vert_{\mathrm{L}^{2}(\omega)}\leq (\sqrt{d}/\pi)(\mathrm{diam}(\omega)/\ell)
\Vert \nabla u\Vert_{\mathrm{L}^{2}(\omega)}$. \hfill 

\end{proof}

This result is given in
\cite[Lemma 3]{MR2606959} and \cite[Lemma 2.1]{MR1993936} but might be improved. Indeed we will see in the following that the $\omega_i$ are elongated in one direction and a thorough study of this inequality might allow us to highlight both $\mrm{diam}(\omega)$ and $\mrm{diam}(\omega_\perp)$.

In this framework we can prove some results on the local approximability of the solution in low dimension.

\begin{prop}\quad\\
  Let $\eta>0$, $q\in (0,1)$, $p\in \mathbb{N}$ with $p>2$ and a tube cluster $\tau = \Omega\cap (\RR\times \tau_{\perp})$. For all $\sigma\subset\Omega$ satisfying $\mrm{diam}(\tau) \leq 2\eta \mrm{dist}(\tau , \sigma)$, we can find a space $\mV\subset\mL^{2}(\tau)$
  where for some constant $C_{\mrm{dim}}$
  \begin{equation}\label{DimensionBound1}
    \mrm{dim}(\mV)\leq C_{\mrm{dim}}p^{d+1}
  \end{equation}
  and such that for all
  right hand sides $f\in  \mL^2(\Omega)$ with $\mrm{supp}(f)\subset \sigma$, the unique function
  $u\in\mH^{1}_{0}(\Omega)$ satisfying $a(u,\varphi) = \langle f,\varphi\rangle_{\mL^2(\Omega)}\;\forall\varphi\in \mH^{1}_0(\Omega)$
  can be approximated on $\tau$ by $v\in \mV$ satisfying the estimates for some constant $C$
  \begin{equation}
    \begin{aligned}
      & \Vert \nabla u-\nabla v\Vert_{\mL^{2}(\tau)}\leq C \frac{p}{\mrm{dist}(\tau , \sigma)} q^{p-1}\Vert f\Vert_{\mL^2(\Omega)}\\
      & \Vert u-v\Vert_{\mL^2(\tau)}\leq C q^{p}\Vert f\Vert_{\mL^{2}(\Omega)}
    \end{aligned}
  \end{equation}
\label{theoreme}
\end{prop}

\begin{proof}
We consider a tube cluster $\tau=(\mathbb{R}\times\tau_\perp)\cap\Omega$. Let $\eta>0$, $f\in\mL^2(\Omega)$ with $\mrm{supp}(f)\subset \sigma$ such that $\mrm{diam}(\tau)\leq 2\eta\mrm{dist}(\tau, \sigma)$ and $u\in \mH^1_0(\Omega)$ the solution of the variational problem \eqref{prob_var}.
Set $\delta  =
\mrm{dist}(\tau_{\perp},\sigma)$, we consider
\(\omega_{\perp,j}:=\{\bx'\in
\RR^{d-1}\vert\;\mrm{dist}(\bx',\tau_{\perp})\leq (1-j/p)\delta \},\)
and define $\omega_j:= \Omega\cap (\RR\times\omega_{\perp,j})$ for
$j=1,\dots,p$. These are nested tubes clusters $\tau = \omega_{p}\subset
\omega_{p-1}\subset \dots\subset \omega_1\subsetneq\tau_\delta$. The idea of the proof is to iterate lemma \ref{PoincareApprox} and our Caccioppoli estimate \eqref{CaccioppoliEstimate1} on narrower sets.
\quad\\
 By eventually extending $u$ by $0$ outside of $\Omega$ we build $u_1\in \mH^1_0(\Omega \cup \omega_1)$ , $u_1|_\Omega := u$ and $u_1|_{\omega_1\backslash\Omega}:=0$. For any domains $D\subset \omega_1$ such that $\mrm{dist}(D, \partial\omega_1)>0$ we have $u_1|_D\in \mH^1(D)$. Moreover, for any $v\in \mH^1_0(\Omega)$ with $\mrm{supp}(v)\subset D$, because $\mrm{supp}(v)\cap\mrm{supp}(f)=\emptyset$, we have \[ a(u_1|_\Omega, v) = a(u,v)=0,\]
hence $u_1\in \mathcal{H}(\omega_1)$.
As a consequence, we can apply lemma \ref{PoincareApprox} with $\mathcal{Z} =
\mathcal{H}(\omega_1)$ to find
a subspace $\mV_1\subset \mathcal{H}(\omega_1)$ such that $\mrm{dim}(\mV_1)\leq
\ell^d$ and some $v_1\in \mV_1$ satisfying with our admissibility condition
\begin{equation*}
  \begin{aligned}
     \Vert u_1-v_1\Vert_{\mL^{2}(\omega_1)}&\leq \frac{\sqrt{d}}{\pi} \frac{\mrm{diam}(\omega_1)}{\ell}\Vert \nabla u_1\Vert_{\mL^{2}(\omega_1)}\\
     &\leq \frac{\sqrt{d}}{\pi}
    \frac{2\delta (\eta+1)}{\ell}\Vert \nabla u_1\Vert_{\mL^{2}(\omega_1)}
  \end{aligned}
\end{equation*}
By taking $u_2 := (u_1-v_1)|_{\omega_2}$, as $\omega_1$ and $\omega_2$ are nested tube clusters we can apply our Caccioppoli \eqref{CaccioppoliEstimate1} estimate and we get 
\begin{equation*}
  \begin{aligned} \Vert \nabla u_2\Vert_{\mL^{2}(\omega_2)}&=\Vert \nabla (u_1-v_1)\Vert_{\mL^{2}(\omega_2)}  \leq \frac{1}{\mrm{dist}(\partial \omega_1, \omega_2)}  \Vert u_1-v_1 \Vert_{\mL^{2}(\omega_1)} \\
  &\leq \frac{p}{\delta} \frac{\sqrt{d}}{\pi\ell}2(\eta +1)\delta \Vert \nabla u_1\Vert_{\mL^{2}(\omega_1)}\\
  &\leq c \frac{p}{l} \Vert\nabla u_1\Vert_{\mL^{2}(\omega_1)}
    \end{aligned}
\end{equation*}
with $c=2\sqrt{d}(\eta+1)/\pi $, this leads to $u_1\in \mathcal{H}(\omega_2)\cap \mH^1(\omega_2)$ and we can again apply lemma \ref{PoincareApprox}. This way, we recursively build for any $i\leq p$, $u_i:= (u_{i-1}-v_{i-1})|_{\omega_i}$ and we have $u_i\in\mathcal{H}(\omega_i)\cap\mH^1(\omega_i)$. More precisely for any $i\leq p$ we can find $v_0\in\mV_0, \dots ,v_{i-1}\in V_{i-1}$ such that $\mathrm{dim}(V_k)\leq l^d$ for all $k\leq i-1$ and  
\[
\begin{aligned}
    \Vert\nabla u_i \Vert_{\mL^2(\omega_i)} &= \Vert \nabla (u_{i-1}|_{\omega_i}-v_{i-1}|_{\omega_i} )\Vert_{\mL^2(\omega_i)} \leq \frac{cp}{\ell} \Vert \nabla u_{i-1}\Vert_{\mL^2(\omega_{i-1})} \\
    &\leq \dots \leq (\frac{cp}{\ell})^{i-1} \Vert \nabla u_1\Vert_{\mL^2(\omega_1)}.
\end{aligned} \]
Because $u_i\in\mathcal{H}(\omega_i)\cap\mH^1(\omega_i)$ we can use the lemma \ref{PoincareApprox} to find $v_i\in\mV_i$ such that $\mathrm{dim}(V_i)\leq \ell^d$ and 
\begin{equation*}
  \begin{aligned}
\Vert u_i -v_i\Vert_{\mL^2(\omega_i)} &\leq \frac{\sqrt{d}}{\pi}\frac{\mrm{diam}(\omega_i)}{\ell}\Vert \nabla u_i \Vert_{\mL^2(\omega_{i})} \\
&\leq \frac{\sqrt{d}}{\pi} \frac{2\delta (\eta+1)}{\ell} (\frac{cp}{\ell})^{i-1} \Vert \nabla u_1\Vert_{\mL^2(\omega_1)}\\
&\leq \frac{\delta }{p} (\frac{cp}{ \ell})^{i} \Vert \nabla u_1\Vert_{\mL^2(\omega_1)}
     \end{aligned}
\end{equation*}

At last, using $\tau\subset \omega_p$, we get $v= (v_1+\dots+v_p)|_\tau \in \mV=(\mV_1+\dots+\mV_p)$ with $\mrm{dim}(\mV)\leq p\ell^d$ and the error estimate 
\[ \Vert \nabla (u|_\tau - v) \Vert_{\mL^2(\tau)} \leq(2\frac{(\eta +1)\sqrt{d}}{\pi} \frac{p}{\ell})^{p-1} \Vert \nabla u_1\Vert_{\mL^2(\omega_1)} .\]
Using one last times the lemma \ref{PoincareApprox} on $\omega_p$ we get using the condition on $\mrm{diam}(\tau)$ 
\begin{equation*}
  \begin{aligned}  \Vert u|_\tau -v \Vert_{\mL^2(\tau)} &\leq \frac{\sqrt{d}}{\pi}\frac{\mrm{diam}(\tau)}{\ell} \Big((2\frac{(\eta +1)\sqrt{d}}{\pi} \frac{p}{\ell})^{p-1}||\nabla u_1||_{L^2(\omega_1)} \Big)\\
  &\leq \frac{\delta}{p} (\frac{ 2(\eta +1)\sqrt{d}}{\pi} \frac{p}{\ell})^{p} \Vert \nabla u_1\Vert_{\mL^2(\Omega)} .\end{aligned}
\end{equation*}
In order to finish and get a bound independent of the solution we will use the separation between $\omega_1$ and $\sigma$ ($\mrm{dist}(\omega_{\perp, 1}, \sigma)=\delta/p> 0$) in a final Caccioppoli inequality
\[ ||\nabla u_1||_{\mL^2(\omega_1)}\leq \frac{p}{\delta}\Vert u\Vert_{\mL^2(\omega_\delta)} .\]
Finally, because of our assumptions \eqref{conditionbeta} $\inf_{\Omega}(\beta)\geq\beta_0 >0$ we get that 
\begin{equation*}
  \begin{aligned}
  \Vert u\Vert_{\mL^2(\Omega)} \Vert f \Vert_{\mL^2(\Omega)} &\geq a(u,u) = \varepsilon\langle \nabla u, \nabla u \rangle_{\mL^2(\Omega )} +\int_\Omega \beta u^2 d\bx \\
  &\geq \beta_0||u||_{\mL^2(\Omega)}^2.
  \end{aligned}
\end{equation*}
Hence
\begin{equation}
    \Vert \nabla (u|_\tau - v) \Vert_{\mL^2(\Omega)} \leq\frac{1}{\beta_0}\frac{p}{\delta}(2\frac{(\eta +1)\sqrt{d}}{\pi} \frac{p}{\ell})^{p-1} \Vert f\Vert_{\mL^2(\Omega)}
\label{ineq:nablau}
\end{equation}
\begin{equation}
     \Vert u|_\tau - v \Vert_{\mL^2(\Omega)} \leq \frac{1}{\beta_0} (\frac{ 2(\eta +1)\sqrt{d}}{\pi} \frac{p}{\ell})^{p} \Vert f\Vert_{\mL^2(\Omega)}
\label{ineq:u}
\end{equation}
Finally, with $c:= 2\frac{(\eta+1)\sqrt{d}}{\pi}$, if we take $\ell = \lceil \frac{ cp^2}{q(p-1)} \rceil  $
one has \[ c\frac{p}{\ell} \leq  cp \frac{q(p-1)}{cp^2} = q(1-\frac{1}{p})\]
and so \[(c\frac{p}{\ell})^p\leq q^p(1-\frac{1}{p})^p \leq q^p e^{-1}.\]
One one hand, because $p>2$ we have
\begin{equation*}
  \begin{aligned}
    \ell &\leq \frac{cp^2}{q(p-1)}+1 \leq \frac{2c}{q}p +\frac{p}{2} \leq p(2\frac{c}{q}+\frac{1}{2}).
    \end{aligned}
\end{equation*}
However, using $\mrm{dim}(\mV)\leq p\ell^d$ we get
\[\mrm{dim}(\mV)\leq (\frac{2c}{q} +\frac{1}{2})^d p^{d+1}  \]
which concludes the proof with constants $C = \frac{1}{e\beta_0}$ and $C_{\mrm{dim}} = (\frac{2c}{q} + \frac{1}{2})^d$.\hfill
\end{proof}

The result differs from \cite[Thm.1]{MR2606959} due to the scale of the Péclet number, translating in a $\varepsilon^{-1}$ term in the error bound if we only rely on the coercivity. Instead, one must rely on the norm $\Vert f\Vert_{\mL^2(\Omega)}$ rather than $\Vert f\Vert_{\mH^{-1}(\Omega)}$ as we can see in \eqref{ineq:nablau} and \eqref{ineq:u}. This change of reference frame compelled us to consider a maximal approximation space $\omega_1\subsetneq\omega_\delta$, which enabled us to establish a Caccioppoli inequality. Moreover, this result holds only for the specific domains described above: tubes aligned with the advection stream.
\subsection{Partitioning}
Hierarchical compression exploit local regularity properties of the solution such as proposition \ref{theoreme} in order to infer a low rank block structure. Because of this locality, we are incline to work on subdomains $\tau$ and $\sigma$ of $\Omega$. However, as we cannot test every partitioning we narrow the search using a cluster tree $\mathcal{T}(\Omega)$. 
Let us briefly recall this concept. Given a set of points $X=\{\mathbf{x}_i\}_{i\in\mathcal{I}}$ in $\Omega$ where $\mathcal{I}\subset \mathbb{N}$ is finite, a cluster tree $\mathcal{T}(\mathcal{I})$ is a tree satisfying 
\begin{itemize}
    \item root$(\mathcal{T}(\mathcal{I}))=\mathcal{I}$.
    \item For $t\in \mathcal{T}(\mathcal{I})$, if $\mrm{sons}(t)\neq \emptyset$ then $t=\{t' |\;t'\in\mrm{sons}(t)\}$.
    \item For $t\in \mathcal{T}(\mathcal{I})$, if $\mrm{sons}(t)= \emptyset$ then $\#t\leq \bm{n}_{min}$.
\end{itemize}
The nodes of the tree are called clusters and are subsets of $\mathcal{I}$. The parameter $\bm{n}_{min}$ determines the maximal size of the leaves. This way we can split the index set $\mathcal{I}\times\mathcal{I}$ of a matrix $A\in \mathbb{R}^{\mathcal{I}\times\mathcal{I}} $ into a partition \[ P = \{t\times s , \; t,s\in \mathcal{T}(\mathcal{I})\}.\] For simplicity, using the correspondence between a domain \( \tau \) and the discretization points $t$ it contains, we will write \( \mathcal{T}(\Omega) \) instead of \( \mathcal{T}(\mathcal{I}) \). This allows us to focus more on the geometric partitioning rather than the discrete indexing.
In our setting, the children of a node in the cluster tree are obtained through a splitting strategy. This strategy defines how a given domain, or the associated index set, is recursively subdivided into smaller subdomains.
More precisely, we focus on a geometric perspective, where the splitting is defined on the spatial domain $\Omega$ rather than directly on the index set. We consider a splitting function
\[
\psi : \mathcal{P}(\Omega) \rightarrow \mathcal{P}(\Omega)\times \mathcal{P}(\Omega)
\]
that maps a domain $\tau \subset \Omega$ to a pair of non-overlapping subdomains $(\tau_1, \tau_2)$ such that
\[
\tau_1 \cap \tau_2 = \emptyset \quad \mrm{and} \quad \tau_1 \cup \tau_2 = \tau.
\]
This function determines how each domain is divided into two children, and recursively applied it defines the entire cluster tree structure.

We've seen from our analysis that in order to obtain a Caccioppoli inequality we had to consider tube clusters introduced in definition \ref{tube-cluster}. Hence, in order to apply our results, each node of the cluster tree must be a tube cluster, meaning that the splitting strategy must be chosen accordingly. For instance, in the context of elliptic problems, the most commonly used method is the Principal Component Analysis (PCA). However, this approach is not suitable in our case, as we aim to construct clusters that align with the field lines.
For a constant advection it is actually rather easy to produce such a cluster tree, in the following we take $\bb=(1, 0, \dots , 0)$ (by eventually rotating  and rescaling the domain). Let us consider a tube cluster $\tau =(\mathbb{R}\times\tau_{\perp})\cap\Omega$. It appears that any splitting $\psi:\tau_\perp\mapsto \{ \tau_{ \perp}^1, \dots, \tau_\perp^n \}$, with $\tau_{ \perp}^i\cap \tau_\perp^j = \emptyset$ if $i\neq j$ and $\bigcup_i \tau_\perp^i = \tau_\perp$, produces subdomains $\tau^i=(\mathbb{R}\times\tau_{\perp}^i)\cap\Omega$ which are tube clusters as illustrated in Figure \ref{fig:split_orth}.\tdplotsetmaincoords{60}{20}
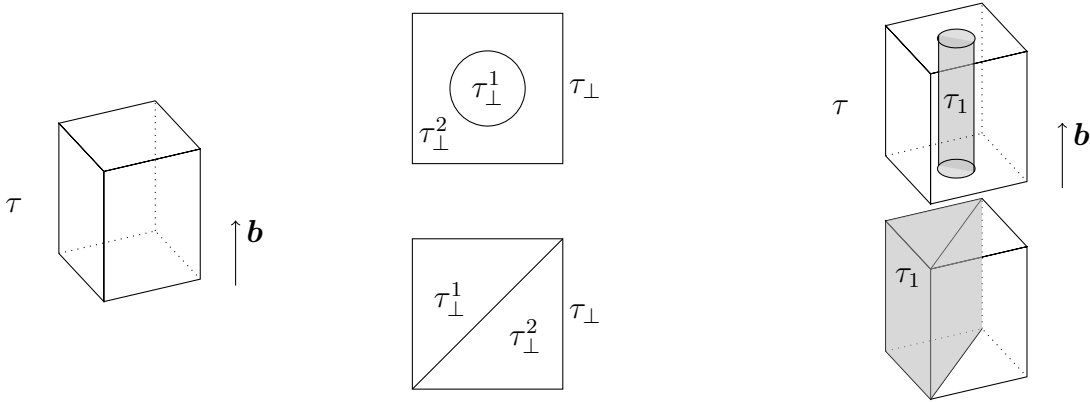
\begin{figure}
\begin{minipage}{0.3\textwidth}
\begin{tikzpicture}[tdplot_main_coords]
\draw (1,0,2) -- (0,1,2) -- (-1,0,2) -- (0,-1,2) -- cycle
    (-1,0,0) -- (0,-1,0) -- (0,-1,2) -- (-1,0,2) -- cycle
    (1,0,2) -- (0,-1,2) -- (0,-1,0) -- (1,0,0) -- cycle;

\draw (-2,1,0) node{$\tau$} ; 
\draw[dotted](-1, 0,0)--(0, 1, 0)--(1,0,0); 
\draw[dotted](0,1,0)--(0,1,2) ;
\draw[->](1.5, 0, 0)--(1.5,0,0.5)node[above right]{$\bb$}--(1.5,0,1) ; 
\end{tikzpicture}
\end{minipage}%
\hspace{0.5cm}
\begin{minipage}{0.3\textwidth}
\begin{tikzpicture}
    \draw(0,0) rectangle (2, 2) ; 
    \draw (2.3, 1) node{$\tau_\perp$} ; 
    \draw (1,1) circle(0.5) ;
    \draw (1,1)node{$\tau_\perp^1$} ; 
    \draw(0.3,0.4) node{$\tau_\perp^2$} ;

    \draw(0,-3) rectangle (2, -1) ; 
    \draw (2.3, -2) node{$\tau_\perp$} ; 
    \draw(0, -3) -- (2,-1) ; 
    \draw (0.5,-1.8) node {$\tau_\perp^1$} ; 
    \draw (1.5,-2.3) node {$\tau_\perp^2$} ; 
\end{tikzpicture}
\end{minipage}%
\hspace{0.5cm}
\begin{minipage}{0.3\textwidth}
\begin{tikzpicture}[tdplot_main_coords]
\draw (1,0,2) -- (0,1,2) -- (-1,0,2) -- (0,-1,2) -- cycle
    (-1,0,0) -- (0,-1,0) -- (0,-1,2) -- (-1,0,2) -- cycle
    (1,0,2) -- (0,-1,2) -- (0,-1,0) -- (1,0,0) -- cycle;
\fill[gray!30, opacity=0.8] (0,0,0) circle (0.25);

\begin{scope}[canvas is xy plane at z=2]
    \fill[gray!30, opacity=0.8] (0,0) circle (0.25);
\end{scope}

\fill[gray!50, opacity=0.6] 
    (0.25, 0, 0) -- (0.25, 0, 2) -- (-0.25, 0, 2) -- (-0.25, 0, 0) -- cycle;

\draw (0,0,0) circle (0.25); 
\draw (0,0,2) circle (0.25); 
\draw (0.25, 0, 0) -- (0.25, 0, 2); 
\draw (-0.25, 0, 0) -- (-0.25, 0, 2); 
\draw(0,0,1) node{$\tau_1$} ;

\draw (-2,1,0) node{$\tau$} ; 
\draw[dotted](-1, 0,0)--(0, 1, 0)--(1,0,0); 
\draw[dotted](0,1,0)--(0,1,2) ;
\draw[->](1.5, 0, 0)--(1.5,0,0.5)node[above right]{$\bb$}--(1.5,0,1) ; 

\draw (1,0,-1) -- (0,1,-1) -- (-1,0,-1) -- (0,-1,-1) -- cycle
    (-1,0,-3) -- (0,-1,-3) -- (0,-1,-1) -- (-1,0,-1) -- cycle
    (1,0,-1) -- (0,-1,-1) -- (0,-1,-3) -- (1,0,-3) -- cycle;
\draw(0,1,-1)--(0,-1,-1) ;
\draw(0,1,-3)--(0,-1,-3) ;
\draw[dotted](-1, 0,-3)--(0, 1, -3)--(1,0,-3); 
\draw[dotted](0,1,-3)--(0,1,-1) ;
\fill[gray!50, opacity=0.6] (0,1,-1)--(0,-1,-1)--(0,-1,-3)--(0,1,-3)--cycle ; 
\fill[gray!50, opacity=0.6](0,1,-1)--(0,-1,-1)-- (-1,0,-1)-- cycle ; 
\fill[gray!50, opacity=0.6](-1,0,-1)--(0,-1,-1)--(0,-1,-3)--(-1,0,-3)--cycle ; 
\draw (-0.5, -0.5, -1.5) node{$\tau_1$} ; 
\end{tikzpicture}
\end{minipage}%
\caption{Two splitting producing tube clusters}
\label{fig:split_orth}
\end{figure}Because $\Omega$ is bounded, possibly after extending it in the direction $b^\perp$, it is a tube cluster. More precisely, $\Omega = (\mathbb{R}\times\Omega_\perp )\cap \Omega$ where $\Omega_\perp$ does not contains $\bb$, for simplicity we will assume that $\Omega_\perp$ is normal to $\bb$. By splitting $\Omega_\perp$ we produce tube clusters on $\Omega$ and by iterating this procedure, we build the cluster tree $\mathcal{T}(\Omega)$. The cluster tree we build this way is actually a tube cluster tree: for all nodes $\tau$ of $\mathcal{T}(\Omega)$, $\tau$ is a tube cluster. One way to proceed is to sort the points according to their projection on a vector of $\tau_\perp$ as we can see in the pseudo code 
\begin{algorithm}
\caption{Split the points $t$ of $\tau$ into two parts using projection on $b^\perp$}
\begin{algorithmic}[1]
\State Compute a vector $b^\perp$ orthogonal to $b$;
\State Sort the points $p\in t$ by increasing order of $\langle p, b^\perp \rangle$, using
\[
\texttt{sort}(t.\texttt{begin()}, t.\texttt{end()}, (p_1, p_2) \mapsto \langle p_1, b^\perp \rangle < \langle p_2, b^\perp \rangle);
\]
\State Let $t_1$ be the first half of $t$, and $t_2$ the second half;
\State \Return $(t_1, t_2)$
\end{algorithmic}
\end{algorithm}
Equipped with this tube cluster tree, we obtain a block representation of the finite element matrix and its inverse. We will denote $A|_{\tau\times\sigma}$ the block corresponding to the interaction of the clusters $\tau$ and $\sigma$. It is crucial to understand that partitioning induces a renumbering of the points, thereby determining the structure of the matrix as illustrated In Figure \ref{fig:partition1}. This highlights how our partitioning strategy significantly impacts the block structure of the matrix.

\begin{figure}
    \centering
    \begin{minipage}[c]{0.33\textwidth}
        \centering
        \begin{tikzpicture}[scale=0.8]
            \draw(0,0) rectangle (2, 2); 
            \draw(1, 0)--(1, 2); 
            \draw (0.5, 1) node{$\tau_1$}; 
            \draw (1.5, 1) node{$\tau_2$}; 
            \draw(1.6, 1.5) node{\tiny$\times$} ; 
            \draw(1.6, 1.5) node[above]{$x$} ; 
            \draw(0.6, 1.5) node{\tiny$\times$} ; 
            \draw(0.6, 1.5) node[above]{$y$} ; 
            \draw[thin, dashed, ->](1.6, 1.5)--(2.6, 1)node[color= blue, right]{$x_j$} ; 
            \draw[thin, dashed, ->](0.6, 1.5)--(-0.6, 1)node[color= blue, left]{$y_k$} ; 
            
            \draw(0,-3) rectangle (2, -1); 
            \draw(1.6, 1.5-3) node{\tiny$\times$} ; 
            \draw(1.6, 1.5-3) node[above]{$x$} ; 
            \draw(0.6, 1.5-3) node{\tiny$\times$} ; 
            \draw(0.6, 1.5-3) node[above]{$y$} ; 
            \draw[thin, dashed, ->](1.6, 1.5-3)--(2.6, -1)node[color= red, right]{$x_i$} ; 
                        \draw[thin, dashed, ->](0.6, 1.5-3)--(-0.6, -1)node[color= red, left]{$y_l$} ; 

            \draw(0, -2)--(2, -2); 
            \draw (1, 1.5-3) node{$\tau_1$}; 
            \draw (1, 0.5-3) node{$\tau_2$}; 
        \end{tikzpicture}
    \end{minipage}
    \hspace{0.1cm}
    \begin{minipage}[c]{0.33\textwidth}
        \centering
        \begin{tikzpicture}[scale = 0.8]
            \draw(0,0) rectangle (2,2) ; 
            \draw(0,1)--(2,1) ; 
            \draw(1,0)--(1,2) ; 
            \draw(1.8,0)node{\tiny$\times$} ; 
            \draw(1.8,0)node[below,blue]{$x_j$} ; 
            \draw (2, 0.2)node{\tiny$\times$} ; 
            \draw (2, 0.2)node[right, blue]{$x_j$} ; 
            \draw (0, 1.6)node{\tiny$\times$} ; 
            \draw (-0.3, 1.6)node[ blue]{$y_k$} ;
            \draw (0.4, 2.0)node{\tiny$\times$} ; 
            \draw (0.4, 2.0)node[ blue, above]{$y_k$} ; 
            \draw(0,0-3) rectangle (2,2-3) ; 
            \draw(0,1-3)--(2,1-3) ; 
            \draw(1,0-3)--(1,2-3) ; 
            \draw(0,-1.8)node{\tiny$\times$} ; 
            \draw(0,-1.8)node[left,red]{$x_i$} ; 
            \draw(0.8, -1)node{\tiny$\times$} ; 
            \draw(0.8, -1)node[above,red]{$x_i$} ; 
            \draw(0,-1.2)node{\tiny$\times$} ; 
            \draw(0,-1.2)node[left,red]{$y_l$} ; 
            \draw(0.2, -1)node{\tiny$\times$} ; 
            \draw(0.2, -1)node[above,red]{$y_l$} ;

        \end{tikzpicture}
    \end{minipage}
    \caption{Impact of the numbering, on the left two splitting and on the right the block matrix induced}
    \label{fig:partition1}
\end{figure}
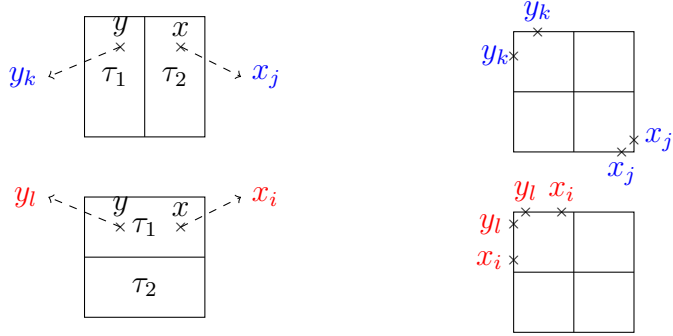
By construction, nodes $\tau$ and $\sigma$ of $\mathcal{T}(\Omega)$ are tubes aligned with the convection. In particular, if $\mathrm{dist}(\tau,\sigma) >0$ we can use proposition \ref{theoreme}. Following the approach of Steffen B\"orm in \cite{MR2606959} one can prove via Clement type operator the following result
\begin{prop}
    Let $\eta>0, q \in(0,1)$,$C>0$ a constant depending only on $\eta, q $ and $\Omega$, $\tau$ and $\sigma$ tubes clusters of the form of definition \ref{tube-cluster} satisfying $\mrm{diam}(\tau)\leq 2\eta \mrm{dist}(\tau, \sigma)$. Then, for all $p\in\mathbb{N}$ with $p>2$ one can find some matrices $U_{t, k}\in\mathbb{R}^{|\tau|\times k}$ and $V_{s, k}\in\mathbb{R}^{|\sigma|\times k}$ with $k<C_{dim}p^{d+1}$ such that 
    \begin{equation}
        \Vert (A^{-1}|_{\tau\times\sigma} - U_{t, k}V_{s, k}^T)x\Vert \leq Cq^p\Vert x\Vert \ \ \ \forall x \in \mathbb{R}^{|\sigma|}.
    \end{equation}
\label{gros_th}
\end{prop}
The main difference with \cite{MR2606959} lies in the fact that our cluster are of the form of definition \ref{tube-cluster} and that the dependence of our constants $C$ and $C_{dim}$ on $\eta$ is slightly different. 
\begin{rem}
The tube cluster hypothesis inherently constrains the minimum leaf size in the cluster tree. Since our analysis is restricted to tube clusters, and our definition requires them to intersect both the inflow and outflow boundaries of the domain (i.e., regions where $\bb\cdot \bm{n}_{\partial\Omega}\neq 0$), the characteristic length of the flow imposes a lower bound on the leaf size. For instance with a cartesian grid of $n\times n$ points of $[0,1]^2$ points, $\bb = (1,0)$ leads to a minimal size of $n$ while $\bb=(1,1)$ doesn't have a constant minimal size.

As a consequence, refining the mesh leads to larger dense blocks in the hierarchical representation. While our theoretical framework does not justify taking leaf sizes smaller than the characteristic length, our numerical experiments suggest that smaller values, down to one-fifth of the characteristic length, can still yield meaningful approximations, though always in relation to the flow structure.
    \label{restriction}
\end{rem} 
\section{Study of the non constant case}
\label{section3}

In the following we are going to extend this result to the more general case $\bb\in C^1(\mathbb{R}^d)$ and not vanishing, we still assume that $\mathrm{div}(\bb/2)-\beta \leq -\beta_0$ with $\beta_0>0$. Namely we would like to extend our definition of tube clusters such that for two tube clusters $\tau\subsetneq\tau_\delta$ we can prove the existence of a cutoff $\eta$ verifying \[ \int_{\tau_\delta} \eta u^2\bb\cdot \nabla \eta d\bx = 0.\]
\subsection{P\'eclet-uniform Caccioppoli inequality}
We consider a vector field $\bb\in C^1(\mathbb{R}^d)^d$ such that $\bb(x)\neq 0 \quad \forall x\in\mathbb{R}^d$ is not vanishing. We need to properly define the field line as we can no longer take $\bb$ as a basis vector.
For any $x\in\mathbb{R}^d$, there exists an interval $I\subset \mathbb{R}$, $0\in I$, and $\varphi : I\rightarrow \mathbb{R}$ a $C^1$ map such that $\partial_t \varphi(t) = \bb(\varphi(t))$ $\forall t\in I$ and $\varphi(0)=x$. The couple $(I, \varphi)$ is called a solution of the dynamical system induced by $\bb$, it is said to be a maximal solution if for any other solution $(J, \psi)$ such that $I\subset J$ and $\psi|_I=\varphi$, then $I=J$. For all $x\in\mathbb{R}^d$ we take $\mathcal{I}(x)\subset \mathbb{R}$ and $t\mapsto \mathbf{\phi} (x, t)\in\mathbb{R}^d$ such that $(\mathcal{I}(x), \mathbf{\phi}(x, \cdot))$ is the maximal solution verifying
\begin{equation}
    \begin{aligned}
    \partial_t \mathbf{\phi}(x, t)&= \bb(\mathbf{\phi}(x, t))\\
    \mathbf{\phi}(x, 0)& = x
\end{aligned}
\label{fieldline}
\end{equation}
In Figure \ref{maxdomain} we illustrate the difference between one solution and the maximal solution for some $x\in \Omega$.
\begin{figure}
    \centering
  
\begin{tikzpicture}
    \draw[thick] (0,0) rectangle (5,3);
    
    \draw[domain=0:1.5,smooth,variable=\x,red,dotted, thick] 
        plot ({\x}, {1.5 + sin(deg(2*pi*\x/5))});
    \draw[domain=3.5:5,smooth,variable=\x,red,dotted, thick] 
        plot ({\x}, {1.5 + sin(deg(2*pi*\x/5))});
    \draw[domain=1.5:3.5,smooth,variable=\x,blue,thick] 
        plot ({\x}, {1.5 + sin(deg(2*pi*\x/5))});
    \draw (4.5, 2) node[thick]{$\Omega$} ; 

    \node[rotate=45] at (2.5,1.5) {$\times$};
     \draw (2.5, 1.5) node[above right]{$x$} ; 

    \draw[red, dotted, thick](7, 1.5)--(8.5, 1.5) ; 
        \draw[red, dotted, thick](12, 1.5)--(10.5, 1.5) ; 

    \draw[blue, thick](8.5, 1.5)--(10.5, 1.5) ; 
    \draw(9.5, 1.5) node{$\times$} ;  
    \draw(9.5, 1.5) node[above]{$0$} ;
    \draw(9.5, 1.5)[blue] node[below]{$\mathcal{I}$} ;
    \draw(7.5, 1.5)[red] node[below]{$\mathcal{I}(x)$} ;
    
\end{tikzpicture}
    \caption{A \textcolor{blue}{solution} and the \textcolor{red}{maximal solution} of \eqref{fieldline}}
    \label{maxdomain}
\end{figure}
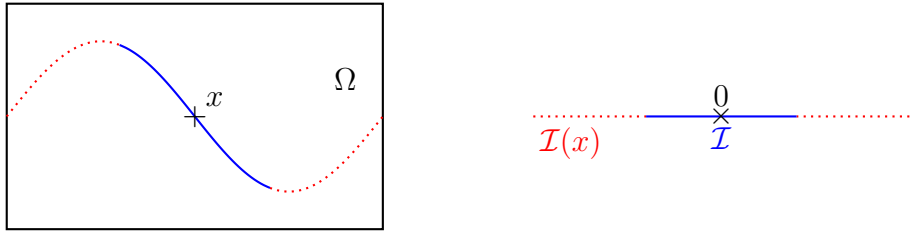
With this notion we can properly define a trajectory of the field $\bb$ as a curve $\mathbf{\phi}(x, \mathcal{I}(x))$ (if there are no ambiguity we will only write $\mathbf{\phi}(x, \mathcal{I})$). From our assumptions on $\bb$ we get that the trajectories are $C^1$ curves. Because of Cauchy-Lipschitz Theorem we have a local unicity of the solution leading for any $x\in \Omega$ to the unicity of the maximal solution $(\mathcal{I}(x), \phi)$. In other words, for two maximal solutions such that $\mathbf{\phi}(x, \mathcal{I}) \cap \mathbf{\phi}(y, \mathcal{I})\neq \emptyset \Rightarrow\mathbf{\phi}(x, \mathcal{I})=\mathbf{\phi}(y, \mathcal{I})$. From there we get that $\{ \mathbf{\phi}(x, \mathcal{I}) \}_{x\in \mathbb{R}^d}$ is a partition of $\mathbb{R}^d$.
With this notion of trajectory, because they partition the domain we can consider the equivalence relation 
\[ x\sim y \Leftrightarrow\mathbf{\phi}(x, \mathcal{I})=\mathbf{\phi}(y, \mathcal{I}). \]
This naturally defines a partition of the space into equivalence classes, where each class corresponds to a distinct trajectory of the flow. The quotient space of the field lines plays a role analogous to $\Omega_\perp$ for $\bb=(1, 0,\dots, 0)$ when $\Omega=(\mathbb{R}\times\Omega_\perp)\cap\Omega$, as it represents the space transversal to the flow trajectories.

In this way, we want to define a tube cluster $\tau$ as a domain satisfying the following property: for all $x \in \tau$, if $y = \phi(x, t_y)$ and $y \in \tau$, then the streamline portion $\phi(x, [0, t_y])$ is entirely contained in $\tau$.

Note that the domain $\Omega$ may not be aligned with the field, typically some field lines may exit $\Omega$ and later re-enter it. This leads to tube clusters that may not be connected in $\Omega$ as depicted In Figure \ref{fig:Omegabar}. However, we can extend $\Omega$ to a larger domain $\bar{\Omega}$ defined as
\[
\bar{\Omega} =: \{ \phi(x, \mathcal{I}_\Omega(x)) \mid x \in \Omega \}
\]
where for each $x \in \Omega$, the interval $\mathcal{I}_\Omega(x) = [t_0, t_1]\subset \mathcal{I}(x)$ is such that
\[
\forall t \in \mathcal{I}(x)\setminus \mathcal{I}_\Omega(x), \quad \phi(x, t) \notin \Omega
\] and \[ \forall x\in\bar{\Omega}, \quad \forall  \mathcal{J}\subset\mathbb{N} \text{ such that } \bar{\Omega}\cap\phi(x, \mathcal{J})\subset \phi(x, \mathcal{J}) \Rightarrow \mathcal{I}_\Omega(x)\subset \mathcal{J} .\]
With this restriction on $\mathcal{I}_\Omega$ we have unicity of $\bar{\Omega}$. This construction is possible as the field line are continuous and $\Omega$ is bounded.

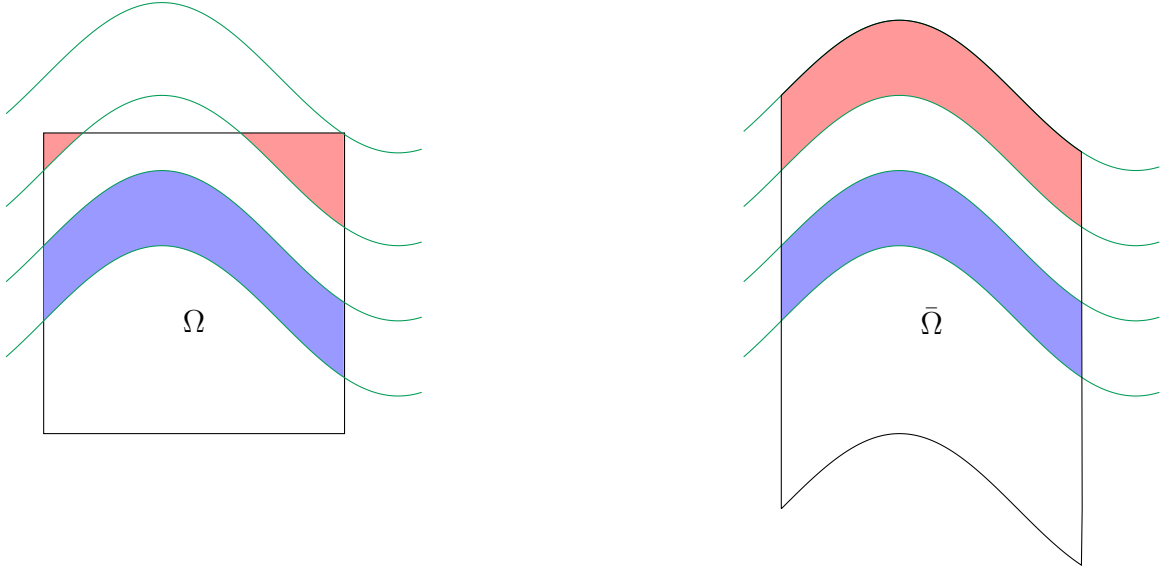
\begin{figure}
\centering
\begin{minipage}{0.48\textwidth}
    \begin{tikzpicture}
    \useasboundingbox (-0.5,-1.5) rectangle (4.5,6);
        \fill[red!40] 
            plot[domain=0:1.27*pi, samples=100] (\x, {sin(\x r)+3.5}) -- 
            plot[domain=1.27*pi:0, samples=100] (\x, {sin(\x r)+4.735}) -- 
            cycle;
        
        \fill[white] (-0.1,4) rectangle (4.1,6); 
        \draw (0,0) rectangle (4,4); 
        
        \draw[ForestGreen] plot[domain=-0.5:1.6*pi, samples=100] (\x, {sin(\x r)+3.5});
        \draw[ForestGreen] plot[domain=-0.5:1.6*pi, samples=100] (\x, {sin(\x r)+4.735});

        \fill[blue!40] 
            plot[domain=0:1.27*pi, samples=100] (\x, {sin(\x r)+1.5}) -- 
            plot[domain=1.27*pi:0, samples=100] (\x, {sin(\x r)+2.5}) -- 
            cycle;
        
        \draw[ForestGreen] plot[domain=-0.5:1.6*pi, samples=100] (\x, {sin(\x r)+2.5});
        \draw[ForestGreen] plot[domain=-0.5:1.6*pi, samples=100] (\x, {sin(\x r)+1.5});
        \draw (2,1.5) node{$\Omega$}; 
    \end{tikzpicture}
\end{minipage}
\hfill
\begin{minipage}{0.48\textwidth}
\centering
    \begin{tikzpicture}
    \fill[red!40] 
            plot[domain=0:1.27*pi, samples=100] (\x, {sin(\x r)+3.5}) -- 
            plot[domain=1.27*pi:0, samples=100] (\x, {sin(\x r)+4.5}) -- 
            cycle;
    \fill[blue!40] 
            plot[domain=0:1.27*pi, samples=100] (\x, {sin(\x r)+1.5}) -- 
            plot[domain=1.27*pi:0, samples=100] (\x, {sin(\x r)+2.5}) -- 
            cycle;
    \useasboundingbox (-0.5,-1.5) rectangle (4.5,6);
      
        \draw[ForestGreen] plot[domain=-0.5:1.6*pi, samples=100] (\x, {sin(\x r)+3.5});
        \draw[ForestGreen] plot[domain=-0.5:1.6*pi, samples=100] (\x, {sin(\x r)+4.5});

        \draw[ForestGreen] plot[domain=-0.5:1.6*pi, samples=100] (\x, {sin(\x r)+2.5});
        \draw[ForestGreen] plot[domain=-0.5:1.6*pi, samples=100] (\x, {sin(\x r)+1.5});
        \draw( 0,-1)--(0,4.5)--plot[domain=0:1.27*pi, samples=100] (\x, {sin(\x r)+4.5})--(4,-1)--plot[domain=1.27*pi:0, samples=100] (\x, {sin(\x r)-1});
                \draw (2,1.5) node{$\bar{\Omega}$}; 

    \end{tikzpicture}
\end{minipage}
    \caption{Tube clusters in $\Omega$ and $\bar{\Omega}$}
    \label{fig:Omegabar}
\end{figure}
Since the solution $u$ vanishes outside $\Omega$, considering the extension $\bar{\Omega}$ instead of $\Omega$ does not affect the analysis. In the following, we shall abuse notation and write $\Omega = \bar{\Omega}$ which allows us to give a simple definition of tube clusters.
\begin{definition}\label{truetube}
    We say that a domain $\tau \subset \Omega $ is a tube cluster (implicitly a $\bb$-tube cluster) if there exists an hyperplane $H\subset \mathbb{R}^d$, with $\boldsymbol{n}_H\cdot \bb(x) \neq 0 \ \forall x \in H$, and $\Gamma_0$ a connected domain of $H$, with  such that \[ \tau = \{ \boldsymbol{\phi}(x_0, s) \ | \ s\in \mathcal{I}(x_0),x_0\in \Gamma_0\}\cap \Omega,\]
    or simply  \[ \tau = \{ \boldsymbol{\phi}(x_0, s) \ | \ s\in \mathcal{I}_\Omega(x_0),x_0\in \Gamma_0\}.\]
        \label{def:tube_cluster_noncst}

\end{definition}

    This definition does transfer to the constant case as $H$ plays the same role as $\tau_\perp$.
Now that we have a notion of tube clusters we would like to find a proper bounding box which would allow us to establish a Caccioppoli inequality. Namely, for a tube cluster $\tau$ we want to characterize $\tau_\delta$ with $\mathrm{dist}(\tau, \partial \tau_\delta \backslash \partial \Omega) >0$, such that we can define a cut-off function $\eta$ which would satisfies \[\int_{\tau_\delta}\eta u^2 (\bb\cdot \nabla \eta) \, d\boldsymbol{x}=0 .\]

We recall that for a domain \( D \subset \mathbb{R}^d \), we define the inflow and outflow boundaries of \( D \) with respect to a vector field \( \bb \) as
\[
\partial D^- = \left\{ x \in \partial D \;\middle|\; \bb(x) \cdot \mathbf{n}_D(x) < 0 \right\}, \quad
\partial D^+ = \left\{ x \in \partial D \;\middle|\; \bb(x) \cdot \mathbf{n}_D(x) > 0 \right\},
\]
where \( \mathbf{n}_D(x) \) denotes the outward unit normal to \( D \) at the point \( x \in \partial D \).

For $\delta >0$, $H$ an hyperplane $H\subset \mathbb{R}^d$ and $\Gamma_0$ a connected part of $H$ we consider the tube cluster $ \tau = \{ \mathbf{\phi}(x_0, s) \ | \ s\in \mathcal{I}_\Omega(x_0),x_0\in \Gamma_0\}$. This way we can consider $\Gamma_{0, \delta} = \{ x\in H \ | \ \text{dist}(x, \Gamma_0)\leq \delta \}$ and define the tube cluster $\tau_\delta := \{ \mathbf{\phi}(x_0, s) \ | \ s\in \mathcal{I}_\Omega(x_0),x_0\in \Gamma_{0, \delta}\}$. This construction is illustrated In Figure \ref{fig:tube_non_const}.

\begin{figure}
    \centering
    \begin{tikzpicture}[xscale=1.5, yscale=0.8]

        \fill[blue!10] 
        plot[domain=0:pi, samples=100] (\x, {-1 + 0.5*sin(3*\x r)}) -- 
        plot[domain=pi:0, samples=100] (\x, {1 + 0.5*sin(3*\x r)}) -- 
        cycle;
    \draw[blue] plot[domain=0:pi, samples=100] (\x, {0.5*sin(3*\x r)}) node[right] {$\mathbf{\phi}(x_0, \mathcal{I})$};
    
    \draw[red, dotted] plot[domain=0:pi, samples=100] (\x, {-1 + 0.5*sin(3*\x r)}) ;
    
    \draw[red, dotted] plot[domain=0:pi, samples=100] (\x, {1 +0.5* sin(3*\x r)});
    \draw[ dashed] (0.5,-1)--(1, 2)node[right]{$H$} ; 
    \draw[purple, very thick] (0.57,-0.52)--(0.865, 1.25) node [left]{$\Gamma_{0}$} ;
        \fill[red!10] 
        plot[domain=0:pi, samples=100] (\x, {-1.3 + 0.5*sin(3*\x r)}) -- 
        plot[domain=pi:0, samples=100] (\x, {-1 + 0.5*sin(3*\x r)}) -- 
        cycle;
        \fill[red!10] 
        plot[domain=0:pi, samples=100] (\x, {1 + 0.5*sin(3*\x r)}) -- 
        plot[domain=pi:0, samples=100] (\x, {1.3 + 0.5*sin(3*\x r)}) -- 
        cycle;
    \draw [green, very thick](0.57,-0.52)--(0.52, -0.8) ; 
    \draw[green, very thick] (0.865, 1.25)--(0.915, 1.55)node[right]{$\Gamma_{0, \delta}\backslash\Gamma_0$} ;
    
    \draw (0.73, 0.4) node[very thick] {$\times$} ; 
        \draw (0.73, 0.4) node[blue, right ] {$x_0$} ;

    \draw[red, dotted] plot[domain=0:pi, samples=100] (\x, {-1.3 + 0.5*sin(3*\x r)}) ;
    
    \draw[red, dotted] plot[domain=0:pi, samples=100] (\x, {1.3 +0.5* sin(3*\x r)});

\end{tikzpicture}
    \caption{\textcolor{blue}{$\tau$} and \textcolor{red}{$\tau_\delta\backslash\tau$} .}
    \label{fig:tube_non_const}
\end{figure}
In the following, we will suppose that there exists an hyperplane $H$ of normal $\mathbf{n}$ crossed by every field line only once. Namely, that for all $x\in H$, $\mathbf{\phi}(x, \mathcal{I})\cap H = \{x\}$, $\bb(x)\cdot \mathbf{n} \neq 0$ and $\Omega\subset\{\mathbf{\phi}(x, \mathcal{I})\}_{x\in H}$. In the following, we will assume that $\Omega$ is already a tube cluster (i.e $\Omega =\hat{\Omega}$) which simplify the notation by writing $\mathcal{I}(x)$ instead of $\mathcal{I}_\Omega(x)$.

\noindent Let us build a proper cut-off function. Given a function $\eta_0\in C^1(H, \mathbb{R})$, with $\eta_0|_{\Gamma_0} = 1$ and $\eta_0|_{H\backslash \Gamma_{0, \delta} }= 0$, we define the function $\eta:\mathbb{R}^d\rightarrow \mathbb{R}$ as \[ \eta(\mathbf{\phi}(x, t)) = \eta_0(x)\  , \forall x \in H\,\, \forall t\in \mathcal{I}(x).   \ \ \ \ (\mathrm{see \ Figure \ }\ref{fig:construction_eta})\]

\tdplotsetmaincoords{50}{120} 
\begin{figure}
    \centering
        \begin{minipage}{0.48\textwidth}
        \centering

    \begin{tikzpicture}[tdplot_main_coords, scale=0.9]
  \draw[red] plot[domain=0:3*pi, samples=100] (\x, {0.5*sin(\x r)}, 0) ;
  \draw[blue] plot[domain=0:3*pi, samples=100] (\x, {0.5*sin(\x r)+1}, 0) ;
  \draw[blue] plot[domain=0:3*pi, samples=100] (\x, {0.5*sin(\x r)+2}, 0) ;
    \fill[blue, opacity=0.2]
    plot[domain=0:3*pi, samples=100]
      (\x, {0.5*sin(\x r)+1}, 0)
    --
    plot[domain=3*pi:0, samples=100]
      (\x, {0.5*sin(\x r)+2}, 0)
    -- cycle;
  \draw[red] plot[domain=0:3*pi, samples=100] (\x, {0.5*sin(\x r)+3}, 0) ;
  \draw [very thick , purple] (5,-1,0)--(5,3,0)node[right]{$\Gamma_{0, \delta}$} ;
  \draw (5,-1,0.5) node[right]{$\eta_0$}; 
  \fill[red, opacity = 0.3](5,-1,0)--(5,3,0) -- (5,3,1)--(5,-1,1)--(5,-1,0) ; 
  \draw[very thick](5,-0.5,0)--(5,0.5, 1)--(5,1.5, 1)--(5,2.5,0) ; 
  \draw[ForestGreen,very thick] plot[domain=0:3*pi, samples=100] (\x, {0.5*sin(\x r)+1}, 1) ;
  \draw[ForestGreen, very thick] plot[domain=0:3*pi, samples=100] (\x, {0.5*sin(\x r)+2}, 1) ;
\draw[dotted, ForestGreen, very thick] (8, {0.5*sin(8 r)} ,0) -- (8, {0.5*sin(8 r) + 1} ,1)--(8, {0.5*sin(8 r) + 2} ,1) --(8, {0.5*sin(8 r) + 3} ,0)node[above, ForestGreen]{$\eta$};
\draw[dotted, ForestGreen, very thick] (3, {0.5*sin(3 r)} ,0) -- (3, {0.5*sin(3 r) + 1} ,1)--(3, {0.5*sin(3 r) + 2} ,1) --(3, {0.5*sin(3 r) + 3} ,0);
\draw (7.8, {0.5*sin(7.8 r) + 1.5} ,0) node[blue]{$\mathbf{\tau}$} ; 
\draw[->](7.8,4,0) --(7.8,4,1) ;
\draw[->](7.8,4,0) --(7.8,5,0)node[right]{$y$} ; 
\draw[->](7.8,4,0) --(8.8, 4, 0)node[left]{$x$} ; 

\end{tikzpicture}
\end{minipage}
    \begin{minipage}{0.48\textwidth}
        \centering

    \begin{tikzpicture}[scale=0.9]
            \draw[red] plot[domain=0:2*pi, samples=50] ( {0.5*sin(\x r)+3}, \x) ; 
                        \draw[blue] plot[domain=0:2*pi, samples=50] ( {0.5*sin(\x r)+2}, \x) ; 
            \draw[blue] plot[domain=0:2*pi, samples=50] ( {0.5*sin(\x r)+1}, \x) ; 

            \draw[red] plot[domain=0:2*pi, samples=50] ( {0.5*sin(\x r)}, \x) ; 

            \fill[ForestGreen, opacity=0.2]
    plot[domain=0:2*pi, samples=50]
      ( {0.5*sin(\x r)+1}, \x)
    --
    plot[domain=2*pi:0, samples=50]
      ( {0.5*sin(\x r)}, \x)
    -- cycle;
               \fill[ForestGreen, opacity=0.2]
    plot[domain=0:2*pi, samples=50]
      ( {0.5*sin(\x r)+1}, \x)
    --
    plot[domain=2*pi:0, samples=50]
      ( {0.5*sin(\x r)+0.2}, \x)
    -- cycle;
                \fill[ForestGreen, opacity=0.2]
    plot[domain=0:2*pi, samples=50]
      ( {0.5*sin(\x r)+1}, \x)
    --
    plot[domain=2*pi:0, samples=50]
      ( {0.5*sin(\x r)+0.4}, \x)
    -- cycle;
                \fill[ForestGreen, opacity=0.2]
    plot[domain=0:2*pi, samples=50]
      ( {0.5*sin(\x r)+1}, \x)
    --
    plot[domain=2*pi:0, samples=50]
      ( {0.5*sin(\x r)+0.6}, \x)
    -- cycle;
                \fill[ForestGreen, opacity=0.2]
    plot[domain=0:2*pi, samples=50]
      ( {0.5*sin(\x r)+1}, \x)
    --
    plot[domain=2*pi:0, samples=50]
      ( {0.5*sin(\x r)+0.8}, \x)
    -- cycle;

          \fill[ForestGreen, opacity=0.2]
    plot[domain=0:2*pi, samples=50]
      ( {0.5*sin(\x r)+2}, \x)
    --
    plot[domain=2*pi:0, samples=50]
      ( {0.5*sin(\x r)+3}, \x)
    -- cycle;
              \fill[ForestGreen, opacity=0.2]
    plot[domain=0:2*pi, samples=50]
      ( {0.5*sin(\x r)+2}, \x)
    --
    plot[domain=2*pi:0, samples=50]
      ( {0.5*sin(\x r)+2.8}, \x)
    -- cycle;
              \fill[ForestGreen, opacity=0.2]
    plot[domain=0:2*pi, samples=50]
      ( {0.5*sin(\x r)+2}, \x)
    --
    plot[domain=2*pi:0, samples=50]
      ( {0.5*sin(\x r)+2.6}, \x)
    -- cycle;
              \fill[ForestGreen, opacity=0.2]
    plot[domain=0:2*pi, samples=50]
      ( {0.5*sin(\x r)+2}, \x)
    --
    plot[domain=2*pi:0, samples=50]
      ( {0.5*sin(\x r)+2.4}, \x)
    -- cycle;
              \fill[ForestGreen, opacity=0.2]
    plot[domain=0:2*pi, samples=50]
      ( {0.5*sin(\x r)+2}, \x)
    --
    plot[domain=2*pi:0, samples=50]
      ( {0.5*sin(\x r)+2.2}, \x)
    -- cycle;

                \fill[ForestGreen]
    plot[domain=0:2*pi, samples=50]
      ( {0.5*sin(\x r)+1}, \x)
    --
    plot[domain=2*pi:0, samples=50]
      ( {0.5*sin(\x r)+2}, \x)
    -- cycle;
    \draw  ( {0.5*sin(5 r)+1.5}, 5) node{$\eta = 1$} ; 
    \draw[purple , thick]  ( {0.5*sin(2.2 r)-0.2}, 2.2)--( {0.5*sin(2.2 r)+3.2}, 2.2) node[right]{$\Gamma_{0, \delta}$}; 
    \draw [->](4,4)-- (4, 3)node[left]{$x$}; 
    \draw[->](4,4)--(5, 4)node[right]{$y$} ; 
    \draw(4,4) node{$\bigotimes$} ;

    \draw (4, 0.4)node[below]{$0$}--(6, 0.4) node[below]{$1$} ; 
    \fill[ForestGreen, opacity = 0.2] (4, 0.4) rectangle (6, 0.9) ; 
    \fill[ForestGreen, opacity = 0.2] (4.4, 0.4) rectangle (6, 0.9) ; 
    \fill[ForestGreen, opacity = 0.2] (4.8, 0.4) rectangle (6, 0.9) ; 
    \fill[ForestGreen, opacity = 0.2] (5.2, 0.4) rectangle (6, 0.9) ; 
    \fill[ForestGreen, opacity = 0.2] (5.6, 0.4) rectangle (6, 0.9) ; 

        \end{tikzpicture}
    \end{minipage}
   \caption{Construction of \textcolor{ForestGreen}{$\eta$} using $\eta_0$}
    \label{fig:construction_eta}
\end{figure}
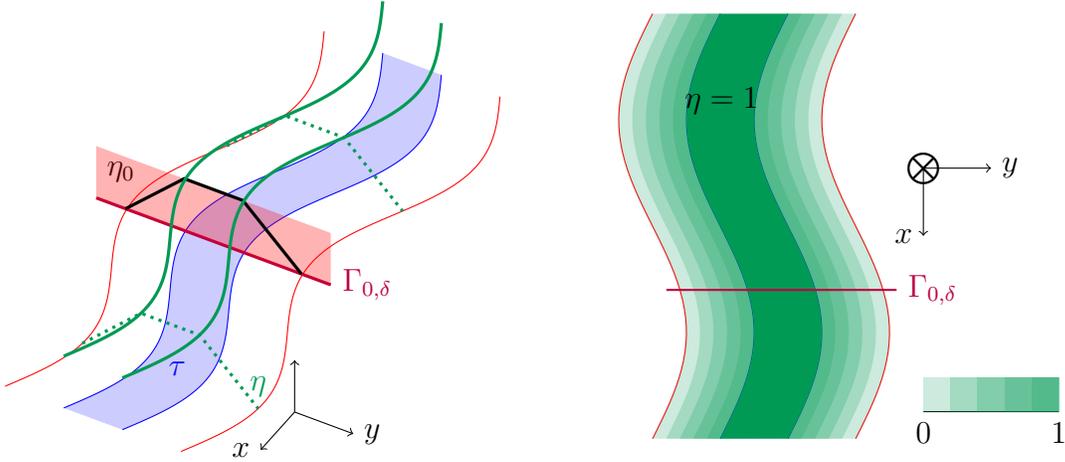

We would like to prove that $\eta$ is of class $C^1$ and constant along $\bb$.
Let $x_0\in H$, we take $F:x\mapsto (x-x_0)\cdot \mathbf{n}_H$, we have that $\partial_t F(\mathbf{\phi}(x, t)) = \partial_t \mathbf{\phi}(x, t) \cdot (\nabla F) (\mathbf{\phi}(x,t)) = \bb(\mathbf{\phi}(x, t))\cdot \mathbf{n}_H$. In particular, as $\bb(\mathbf{\phi}(x_0, 0))\cdot \mathbf{n}_H\neq 0$, using the implicit function theorem in $(x, t)=(x_0, 0)$ we get that there exists a ball $B_0$ centered in $x_0$ and $T\in C^1(B_0, \mathbb{R})$ such that 
\[ \mathbf{\phi}(x, t)\in H \Leftrightarrow F(\mathbf{\phi}(x, t))=0 \Leftrightarrow t=T(x).\] Then, for all $x\in B_0$ one has \[ \eta(x) = \eta(\mathbf{\phi}(x, 0)) = \eta(\mathbf{\phi}(x, T(x))) =\eta_0(\mathbf{\phi}(x, T(x))) . \]
Because $\eta_0, \mathbf{\phi}$ and $\tau$ are $C^1$ functions, we get that $\eta$ is $C^1$ on a neighborhood of $H$.
We now need to extend this locality property on $\Omega$. To do so we consider the flow $\Phi$ of the vector field, defined as the maximal solution of the ODE \[ \frac{d}{dt}\Phi^t(x) = \bb(\Phi^t(x)),   \; \Phi^0(x) = x , \;\; \forall x\in \Omega , \forall t\in \mathcal{I}(x). \]
For any $x\in\Omega$ there exists $t_x>0$ such that the function $t\mapsto \Phi^t(x)$ is $C^1$ on an interval $\mathcal{V}(x) = (-t_x, t_x)$, in particular for any compact subset $K_0\subset \Omega$, by taking \[T_0= \underset{x\in K_0}{\inf}( t_x),\] the mapping $(t,x)\mapsto \Phi^t(x)$ is well defined and is $C^1$ on $[-T_0, T_0]\times K_0$. 
 
 \noindent Let us show that $\Phi^t(x)$ is a local $C^1$ diffeomorphism. Let $x\in\Omega$ , $t,s\in \mathcal{V}(x)$ such that $t+s\in \mathcal{V}(x)$, on one side we have 
\[ \frac{d}{dt}(\Phi^{t+s}(x)) = \bb(\Phi^{t+s}(x)) ,\, \, \Phi^{t+s}(x)|_{t=0} = \Phi^{s}(x).\]
And on the other hand
\[ \frac{d}{dt} ( \Phi^{t}(\Phi^s(x))) = \bb ( \Phi^{t}(\Phi^s(x))), \, \,\Phi^{t}(\Phi^s(x))|_{t=0} = \Phi^s(x).  \]
From the unicity of Cauchy-Lipschitz it comes that 
\[ \Phi^{t+s}(x) = \Phi^t(\Phi^s(x)).\]
Finally, because $\Phi^0$ is the identity we get that $\Phi^t$ is a local $C^1$-diffeomorphism of inverse $\Phi^{-t}$. 

Taking some $x_* \in\Omega$, since we assumed that $H$ is crossed by every field line exactly once, the Cauchy--Lipschitz theorem ensures the uniqueness of a pair $(x_0, t_0)\in H\times \mathbb{R}$ such that $x_*=\mathbf{\phi}(x_0, t_0)$.
We get from the above that there exists a neighborhood $U_0$ of $x_0$ such that $\Phi^{t_0}:x\mapsto \mathbf{\phi}(x, t_0)$ is a $C^1$-diffeomorphism from $U_0$ to $U_*=\Phi^{t_0}(U_0)$, neighborhood of $x_*$. Then, $\eta(x) = \eta(\Phi^{-t_0}(x)) \;  \forall x\in U_*$, however, because $\eta$ is $C^1$ on a neighborhood $U_0$ of $x_0$, we deduce that $\eta$ is also $C^1$ on $U^\star$ neighborhood of $x^*$, as illustrated In Figure \ref{fig:C1diffeo}. 

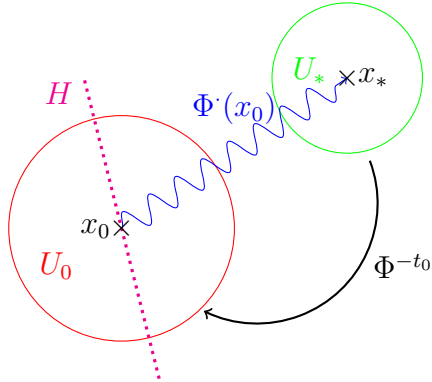
\begin{figure}
    \centering
    \begin{tikzpicture}
        \draw[red](0,0) circle (1.5) ;
        \draw[green](3,2) circle(1); 
        \draw[decorate,->,blue, decoration={snake, amplitude=5pt, segment length=12pt}] (0,0) -- (3,2);
        \draw(1.5 , 1.6) node[color=blue]{$\Phi^{\cdot}(x_0)$} ; 
        \draw[dotted, very thick, magenta] (0.5,-2)--(-0.5, 2); 
        \draw(0,0) node{$\times$} ; 
        \draw(0,0) node[left]{$x_0$} ; 
        \draw(3,2) node{$\times$} ; 
        \draw(3,2) node[right]{$x_*$} ; 

        \draw[->, thick] (3.3,0.9) arc[start angle=20.7, end angle=-116.3, radius=1.6cm];
        \draw (3.2,-0.5) node[right]{$\Phi^{-t_0}$} ; 
        \draw (-0.5 , -0.5) node[red, left]{$U_0$} ; 
        \draw(2.5, 2.1) node[green]{$U_*$} ;
        \draw ( -0.5, 1.8) node[magenta, left] {$H$} ; 
        
    \end{tikzpicture}
    \caption{$C^1$-diffeomorphism between $U_0$ and $U_*$}
    \label{fig:C1diffeo}
\end{figure}
Hence, we have proven that $\eta$ is $C^1$ on $\tau_\delta$.
Moreover, $\eta$ is constant along the field lines, indeed $\forall x_0 \in \Gamma_{0, \delta} , \forall s\in \mathcal{I}(x_0)$ we have $\frac{\partial}{\partial s}\eta(\mathbf{\phi}(x_0, s))= 0$.
Hence, we have
\[0 = \frac{\partial}{\partial s} \eta (\mathbf{\phi}(x_0, s)) = \frac{\partial \mathbf{\phi}}{\partial s }(x_0, s)\cdot \nabla \eta(\mathbf{\phi}(x_0, s)) =\bb(\mathbf{\phi}(x_0, s))\cdot  \nabla \eta(\mathbf{\phi}(x_0, s)),\] 
which leads to \[\forall x \in \tau_\delta, \ \bb(x)\cdot \nabla \eta(x) = 0. \]
A downside of our analysis is that there might be several suitable $H$ for the definition of the tube cluster and, as illustrated In Figure \ref{fig:influenceorientation}, this leads to different $\tau_\delta$.

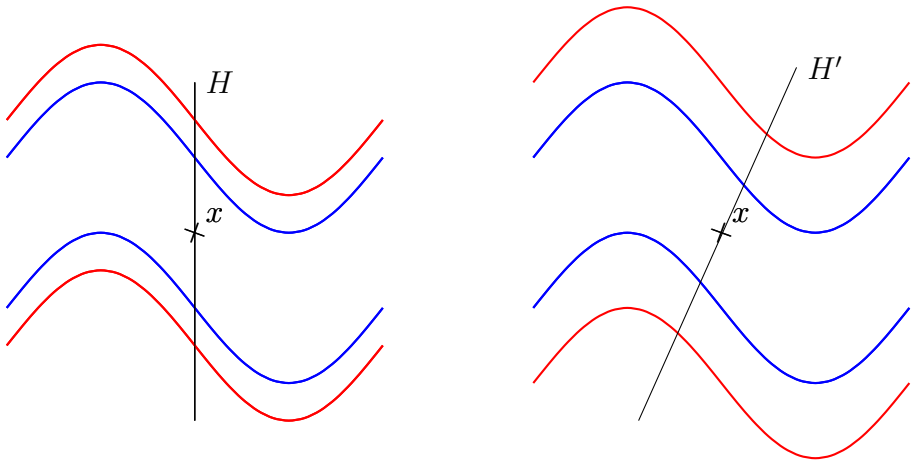
\begin{figure}
    \centering
    \begin{tikzpicture}
    \draw[domain=0:5,smooth,variable=\x,blue, thick] 
        plot ({\x}, {0.5 + sin(deg(2*pi*\x/5))});

    \draw[domain=0:5,smooth,variable=\x,blue, thick] 
        plot ({\x}, {2.5 + sin(deg(2*pi*\x/5))});

    \draw (2.5, -1)--(2.5, 3.5) ; 

    \node[rotate=25] at (2.5,1.5) {$\times$};
     \draw (2.5, 1.5) node[above right]{$x$} ; 

      \draw[domain=0:5,smooth,variable=\x,red, thick] 
        plot ({\x}, {0 + sin(deg(2*pi*\x/5))});

    \draw[domain=0:5,smooth,variable=\x,red, thick] 
        plot ({\x}, {3 + sin(deg(2*pi*\x/5))});

     \draw[domain=0:5,smooth,variable=\x,blue, thick] 
        plot ({\x}, {0.5 + sin(deg(2*pi*\x/5))});

    \draw[domain=0:5,smooth,variable=\x,blue, thick] 
        plot ({\x}, {2.5 + sin(deg(2*pi*\x/5))});

    \draw (2.5, -1)--(2.5, 3.5)node[right]{$H$} ; 

    \node[rotate=25] at (2.5,1.5) {$\times$};
     \draw (2.5, 1.5) node[above right]{$x$} ; 

      \draw[domain=0:5,smooth,variable=\x,red, thick] 
        plot ({\x}, {0 + sin(deg(2*pi*\x/5))});

    \draw[domain=0:5,smooth,variable=\x,red, thick] 
        plot ({\x}, {3 + sin(deg(2*pi*\x/5))});

           \draw[domain=0+7:5+7,smooth,variable=\x,blue, thick] 
        plot ({\x}, {0.5 + sin(deg(2*pi*(\x-7)/5))});

    \draw[domain=0+7:5+7,smooth,variable=\x,blue, thick] 
        plot ({\x}, {2.5 + sin(deg(2*pi*(\x-7)/5))});

    \draw (2.5, -1)--(2.5, 3.5) ; 

    \node[rotate=25] at (2.5+7,1.5) {$\times$};
     \draw (2.5+7, 1.5) node[above right]{$x$} ; 

      \draw[domain=0+7:5+7,smooth,variable=\x,red, thick] 
        plot ({\x}, {-0.5 + sin(deg(2*pi*(\x-7)/5))});

    \draw[domain=0+7:5+7,smooth,variable=\x,red, thick] 
        plot ({\x}, {3.5 + sin(deg(2*pi*(\x-7)/5))});

     \draw[domain=0+7:5+7,smooth,variable=\x,blue, thick] 
        plot ({\x}, {0.5 + sin(deg(2*pi*(\x-7)/5))});

    \draw[domain=0+7:5+7,smooth,variable=\x,blue, thick] 
        plot ({\x}, {2.5 + sin(deg(2*pi*(\x-7)/5))});

    \draw (2.4+6, -1)--(2.5+8, 3.7)node[right]{$H'$} ; 

    \node[rotate=25] at (2.5+7,1.5) {$\times$};
     \draw (2.5+7, 1.5) node[above right]{$x$} ; 
    \end{tikzpicture}
    \caption{Influence of $H$ on the definition of $\tau_\delta$.}
    \label{fig:influenceorientation}
\end{figure}
This is not an issue at all however it does change the distance between the borders of $\tau$ and $\tau_\delta$, and therefore the coefficient in the Caccioppoli inequality. One way to bypass this is to take $\tau_\delta$ the tube cluster such that  \[\inf_{\tau_\delta\supsetneq \tau} \text{dist}(\partial \tau, \partial\tau_\delta ) = \delta.\]
In this framework, the minimal distance between \( \tau \) and \( \partial\tau_\delta \) is \( \delta \). Consequently, for any hyperplane \( H \) defining \( \tau_\delta \), the cutoff function \( \eta_0 \) constructed on \( \Gamma_{0, \delta} \subset H \) satisfies, for some constant \( C \geq 1 \),
\[
\|\nabla \eta_0\|_{\mathrm{L}^\infty(\Gamma_{0, \delta})} \leq \frac{C}{\delta}.
\]
Assuming $C=1$ directly implies
\[
\|\nabla \eta\|_{\mathrm{L}^\infty(\Omega)} \leq \frac{1}{\delta}.
\]From there, for any tube cluster $\tau$ and $\delta>0$ there exists $\eta$ defined on $\tau_\delta$ such that

\[\begin{aligned}
    \Vert \nabla (\eta u) \Vert_{\mathrm{L}^2(\tau)}^2 &= \langle \nabla u , \nabla (\eta^2u)\rangle_{\mathrm{L}^2(\tau_\delta)}-\langle \nabla u , \eta u  \nabla \eta \rangle_{\mathrm{L}^2(\tau_\delta)} + \langle u\nabla \eta , (\nabla \eta u) \rangle_{\mathrm{L}^2(\tau_\delta)} \\
    &= \langle \nabla u , \nabla (\eta^2 u )\rangle_{\mathrm{L}^2(\tau_\delta)} + \langle u\nabla \eta , u\nabla \eta\rangle_{\mathrm{L}^2(\tau_\delta)} \\
    &= \varepsilon^{-1}a(u , \eta^2  u) - \varepsilon^{-1}\int_{\tau_\delta} (\bb\cdot (\nabla u) \eta^2 u +\beta \eta^2 u^2) d\boldsymbol{x} + \Vert u\nabla \eta \Vert_{\mathrm{L}^2(\tau_\delta)}^2 \\
     &= - \varepsilon^{-1}\left(\int_{\partial\tau_\delta}\frac{\bb}{2}\eta^2u^2\cdot \mathbf{n}_{\tau_\delta} d\boldsymbol{s}  -\int_{\tau_\delta}(\text{div}(\frac{\bb}{2})-\beta) \eta^2u^2 d\boldsymbol{x}- \int_{\tau_\delta} (\bb\cdot \nabla \eta)\eta u^ 2d\boldsymbol{x}\right) \\ 
     & \quad+\Vert u\nabla \eta \Vert _{\mathrm{L}^2(\tau_\delta)}.
\end{aligned}\]

Hence, with our hypothesis $u\in \mathrm{H}^1 _0(\Omega) $, $\text{div}(\bb/2)-\beta<0$, $C=1$ and because $b\cdot \nabla \eta =0$ we get
\begin{equation*}
    \begin{aligned}
        \Vert\nabla u \Vert_{\mathrm{L}^2(\tau)}^2 &\leq \varepsilon^{-1} \int_{\tau_\delta}\eta u ^2 \bb \cdot \nabla \eta d\boldsymbol{x} +\Vert u\nabla \eta \Vert_{\mathrm{L}^2(\tau_\delta)}^2 \\
    &\leq \frac{1}{\delta^2}\Vert u \Vert_{\mathrm{L}^2(\tau_\delta)}^2
\end{aligned}.
\end{equation*}

\begin{rem}
Another possible approach would have been to consider the Riemannian manifold $M$ orthogonal to the vector field $\bb$ (i.e., for all $x \in M$, $\vert \bb(x) \cdot \boldsymbol{n}_{M}(x)\vert = \vert \bb(x)\vert $), and to define a function $\eta_0$ on $M$. By transporting $M$ along the flow of $\bb$, one can cover the entire space and thereby extend $\eta_0$, yielding a function $\eta$ such that $\bb \cdot \nabla \eta = 0$.
However, this approach involves several difficulties, notably due to the fact that the flow speed is not necessarily constant. In particular, if the vector field $\bb$ does not derive from a potential, there is no clear meaning to an elementary displacement of the manifold along the field, which significantly complicates the analysis.
\end{rem}

Thus, we have succeeded in establishing a Caccioppoli-type inequality in the case of tubes clusters, and we would therefore like to apply the same proof scheme as in the constant vector field case. However, unlike the constant convection case where the cluster tubes were convex domains, we do not have an explicit Poincaré constant for non convex domains. In particular, it is not known whether this constant depends on the diameter of the domain.
Fortunately, under our assumptions on the vector field, we can derive conditions ensuring that the diffeomorphism which locally straightens the field is in fact a global diffeomorphism. In this way, we can establish a nontrivial bound involving the diameter in the Poincaré constant of a non-convex domain. However, it is important to emphasize that this constant does not apply to arbitrary non-convex domains, but only to the particular case of domains that are cluster tubes for a sufficiently regular vector field~$\bb$.

Let us prove in dimension $d=2,3$ that there exists $C^1$-diffeomorphism that straightens the field lines of $\bb$. We consider the hyperplane $H$  equipped with an orthonormal basis ${\mathcal B}_H := (e_1,..., e_{d-1})$. Let us fix $p \in \Gamma_{0, \delta}$ and let $f$ be the affine transformation defined from $\mathbb{R}^{d-1}$ to $H$ by: for all $y\in\mathbb{R}^{d-1}$,
\begin{eqnarray*}
f(y):=p+\displaystyle\sum_{i=1}^{d-1} y_i e_i.
\end{eqnarray*}
Since ${\mathcal B}_H$ is a basis of $H$, then we deduce that $f$ is an affine isomorphism. We thus set $\widetilde{\Gamma}_{0,\delta} := f^{-1}({\Gamma}_{0,\delta})$.
Then, we introduce $\widetilde{\tau}_\delta$ the set of $\mathbb{R}^d$ defined by $\widetilde{\tau}_\delta:=\{(y,s): ~~y\in \widetilde{\Gamma}_{0,\delta}, s\in \mathcal{I}(x) \hbox{ with } x=f(y)\}$.
We also introduce the function $\widetilde{\phi}$ defined from $\widetilde{\tau}_{\delta}$ to $\tau_\delta$ by: for all $(y,s) \in \widetilde{\tau}_{\delta}$ 
\[ \widetilde{\phi}(y, s) := \phi\left(f(y), s\right).\]
By this way for any $y\in \widetilde{\Gamma}_{0,\delta}$, $\displaystyle \frac{\partial}{\partial s} \widetilde{\phi}(y, s) = {\bb}(\widetilde{\phi}(y, s))$. We consider $J$ the determinant of the Jacobian of $\widetilde{\phi}$ defined by $J :=\mathrm{det}(\nabla \widetilde{\phi})$. The use of differentiation rules for the determinant leads to \[ \frac{\partial J}{\partial s}  = \mathrm{div}({\bb}(\widetilde{\phi}(\cdot, s))J,\]
so that \[ J(\cdot, s) = J(\cdot, 0) e^{\int_0^s \mathrm{div}({\bb}(\widetilde{\phi}(\cdot, \sigma)))d\boldsymbol{\sigma} }.\]

Since for any $y\in \widetilde{\Gamma}_{0,\delta}$, we get $\widetilde{\phi}(y, 0) = \phi(f(y),0)=f(y)$, then we deduce that for any $y\in \widetilde{\Gamma}_{0,\delta}$  \begin{eqnarray*}
J(y, 0) &=& \mathrm{det}(e_1,\ldots,e_{d-1}, {\bb}(\widetilde{\phi}(y, 0))) \\
&=& \pm {\bf n}_H \cdot {\bb}(f(y)).
\end{eqnarray*}

For any $y_0\in \widetilde{\Gamma}_{0,\delta}$, we have $f(y_0)\in  {\Gamma}_{0,\delta}$ and then ${\bf n}_H \cdot {\bb}(f(y_0)) \neq 0$ which implies that $J(y_0,0)\neq 0$. Hence we deduce that $J(y_0, s) \neq 0 $, $\forall (y_0,s) \in \widetilde{\tau}_\delta$. Moreover, we recall that the hyperplane $H$  intersects each field line exactly once, and consequently so does $\Gamma_{0, \delta}$ since $\Gamma_{0,\delta}\subset H$.
This latter combined with the Cauchy Lipschitz Theorem implies that $\widetilde{\phi}: \widetilde{\tau}_\delta\mapsto \tau_\delta$ is injective. This way, using the global inversion theorem (\cite{Lee00}[Corollary~C.36]) we get that $\widetilde{\phi}$ is a $C^1$-diffeomorphism and we can conclude that $\phi$ is also a $C^1$-diffeomophism.

In this setting, let us show that it is possible to establish a Poincaré Wirtinger inequality involving the diameter of the domain.
\begin{lem}
Let \( \omega \subset \Omega \) be a tube cluster, and let \( \mathcal{Z} \) be a closed subspace of \( \mathrm{L}^{2}(\omega) \). Assume that there exists a diffeomorphism \( \Phi \) mapping the extension \( \omega \) onto $\hat{\omega}$ a convex domain in \( \mathbb{R}^d \). Then, for any integer \( \ell \geq 1 \), there exists a subspace \( \mathcal{V} \subset \mathcal{Z} \) with $\dim(\mathcal{V}) \leq \ell^d $
such that
\[
\inf_{v \in \mathcal{V}} \|u - v\|_{\mathrm{L}^{2}(\omega)} \leq C_\Phi \frac{\sqrt{d}}{\pi}\frac{\mathrm{diam}(\omega)}{\ell}  \|\nabla u\|_{\mathrm{L}^{2}(\omega)} \quad \forall u \in \mathcal{Z} \cap \mathrm{H}^{1}(\omega),
\]
where the constant \( C_\Phi > 0 \) depends only on the regularity and geometry of the diffeomorphism \( \Phi \).
\end{lem}

\begin{proof}
Let $\omega\subset\Omega$ be a tube cluster, it is a convex domain. We consider $\hat{\omega}:= \Phi(\omega)$, it is a parallelepiped and in particular it is convex. 
Let \( \ell \in \mathbb{N} \). We perform a regular Cartesian partition of the reference domain \( \hat{\omega} \) into \( \ell^d \) subdomains \( (\hat{\omega}_i)_{1 \leq i \leq \ell^d} \) by dividing each coordinate direction into \( \ell \) intervals of equal length. As illustrated In Figure \ref{fig:partition_bis}, this induces a corresponding partition of the physical domain \( \omega \) into subdomains \( (\omega_i)_{1 \leq i \leq \ell^d} \), defined by
\[
\omega_i := \Phi^{-1}(\hat{\omega}_i), \quad \forall i \in \{1, \dots, \ell^d\}.
\]

Let $u\in \mathrm{H}^1(\omega)$, we consider $\hat{u}:= u\circ \Phi^{-1}$ and we have $\hat{u}\in\mathrm{H}^1(\hat{\Omega})$. For any $\lambda \in\mathbb{R}$ and $i\in[1, \ell^d]$, denoting $D\psi$ the jacobian matrix of a diffeomorphism $\psi$, one has 
\[ \Vert u-\lambda \Vert_{\mathrm{L}^2(\omega_i)}^2 = \int_{\hat{\omega_i}} |\hat u(y)-\lambda|^2 |\mathrm{det}(D\Phi^{-1}(y))| d\boldsymbol{y}. \]
However, $\hat{\omega_i}$ is a convex domain, which means that we can use the Poincaré Wirtinger inequality and we have 
\[ \Vert \hat{u}- \hat{u}_i \Vert_{\mathrm{L}^2(\hat{\omega_i})}^2 \leq \frac{\mathrm{diam}(\hat{\omega}_i)^2}{\pi^2} \Vert \nabla \hat{u} \Vert_{\mathrm{L}^2(\hat{\omega_i})}, \]
where $\hat{u}_i:= \frac{1}{|\hat{\omega}_i|}\int_{\hat{\omega}_i} \hat{u}(y) d\boldsymbol{y}$ is the mean value of $\hat{u}$ on $\hat{\omega}_i$.
We have 

\[\begin{aligned}
    \Vert \nabla \hat{u} \Vert_{\mathrm{L}^2(\hat{\omega}_i)}^2 &= \int_{\hat{\omega}_i}\Vert D\Phi^{-1}(y) \nabla u(\Phi^{-1}(y)) \Vert ^2 d\boldsymbol{y}\leq\int_{\hat{\omega}_i} \Vert D\Phi^{-1}(y) \Vert^2 \Vert \nabla u(\Phi^{-1}(y))  \Vert ^2 d\boldsymbol{y}\\
    &\leq \int_{\omega_i} \Vert D\Phi^{-1}(\Phi(x))\Vert^2 \Vert \nabla u(x)\Vert ^2 |\mathrm{det}(D\Phi(x))|d\boldsymbol{x} \\
    &\leq \underset{y\in\hat{\omega}_i}{\mathrm{sup}}(\Vert D\Phi^{-1}(y)\Vert^2) \underset{x\in\omega_i}{\mathrm{sup}}(|\mathrm{det}(D\Phi(x))|) \Vert \nabla u \Vert^2_{\mathrm{L}^2(\omega_i)}.
\end{aligned}\]

Where $\forall y\in\hat{\omega}$ we take $\Vert D\Phi^{-1}(y)\Vert := \underset{v\in\mathbb{R}^d , v\neq 0}{\mathrm{sup}}(\frac{\Vert D\Phi^{-1}(y) v\Vert}{\Vert v\Vert})$.
This way we get
\[ \Vert u -\hat{u}_i\Vert_{\mathrm{L}^2(\omega_i)}^2 \leq C_{\Phi, i}^2 \frac{\mathrm{diam}(\hat{\omega}_i)^2}{\pi^2}\Vert \nabla u \Vert^2_{\mathrm{L}^2(\omega_i)}, \]
with $C_{\Phi, i} = \underset{y\in\hat{\omega}_i}{\mathrm{sup}}(\Vert D\Phi^{-1}(y)\Vert) \underset{x\in\omega_i}{\mathrm{sup}}(|\mathrm{det}(D\Phi(x))|^{1/2})\underset{y\in\hat{\omega_i}}{\mathrm{sup}}(|\mathrm{det}(D\Phi^{-1}(y))|^{1/2}).$
By construction we have $\mathrm{diam}(\hat{\omega_i}) \leq \frac{\sqrt{d}}{\ell}\mathrm{diam}({\hat{\omega}})$ and $\mathrm{diam}(\hat{\omega}) \leq \underset{x\in\omega}{\mathrm{sup}}(\Vert D\Phi(x)\Vert)\mathrm{diam}(\omega)$. From there the proof is identical to Lemma \ref{PoincareApprox}, taking $k= \ell^d$ we define $\mathcal{W}_k=\{ v\in\mathrm{L}^2(\omega) | \;v|_{\omega_i}\in \mathbb{R}\}$ and take $\Pi(u)\in\mathcal{W}_k$ where $\Pi(u)|_{\omega_i} = \hat{u}_i$ we have 
\begin{equation*}
    \begin{aligned}
        \Vert u-\Pi(u)\Vert^2 _{\mathrm{L}^2(\omega)}&= \sum_{i=1}^k \Vert u-u_i \Vert^2_{\mathrm{L}^2(\omega_i)}
        \leq \frac{d}{\pi^2 } C_{\Phi}^2\sum_{i=1}^k \frac{\mathrm{diam}(\omega)^2}{\ell^2}\Vert \nabla u\Vert^2_{\mathrm{L}^2(\omega_i)}\\
        &\leq \frac{d}{\pi^2 }C_{\Phi}^2\frac{\mathrm{diam}(\omega)^2}{\ell^2}\Vert \nabla u \Vert^2_{\mathrm{L}^2(\omega)}, 
    \end{aligned}
\end{equation*} 
where $C_\Phi =\underset{y\in\hat{\omega}}{\mathrm{sup}}(\Vert D\Phi^{-1}(y)\Vert)\underset{x\in\omega}{\mathrm{sup}}(\Vert D\Phi(x)\Vert) \underset{x\in\omega}{\mathrm{sup}}(|\mathrm{det}(D\Phi(x))|^{1/2})\underset{y\in\hat{\omega}}{\mathrm{sup}}(|\mathrm{det}(D\Phi^{-1}(y))|^{1/2}).$

Finally, consider the $\mathrm{L}^{2}(\omega)-$orthogonal projection $\mathrm{P}:\mathrm{L}^{2}(\omega)\to \mathcal{Z}$
and set $\mathcal{V}:=\mathrm{P}(\mathcal{W}_k)$ so that $\mathrm{dim}(\mathcal{V})\leq \mathrm{dim}(\mathcal{W}_k)$.
Then for any $u\in \mathcal{Z}\cap \mathrm{H}^{1}(\omega)$ we have
$\Vert u-\mathrm{P}\cdot\Pi(u)\Vert_{\mathrm{L}^{2}(\omega)} = \Vert \mathrm{P}(u-\Pi(u))\Vert_{\mathrm{L}^{2}(\omega)}\leq
\Vert u-\Pi(u)\Vert_{\mathrm{L}^{2}(\omega)}\leq C_\Phi(\sqrt{d}/\pi)(\mathrm{diam}(\omega)/\ell)
\Vert \nabla u\Vert_{\mathrm{L}^{2}(\omega)}$. \hfill 

\end{proof}

\begin{figure}
    \centering
    \begin{tikzpicture}
            \draw[domain=0:5,smooth,variable=\x,blue, thick] 
        plot ({\x}, {0.5 + sin(deg(2*pi*\x/5))});

    \draw[domain=0:5,smooth,variable=\x,blue, thick] 
        plot ({\x}, {2.5 + sin(deg(2*pi*\x/5))});
    \draw[blue, thick] (5, {2.5 + sin(deg(2*pi))}) -- (5, {0.5 + sin(deg(2*pi))});
    \draw [blue, thick](0, {2.5 )}) -- (0, {0.5 });

    \draw [->](5.5, 1.25)--(6.5,1.25) ; 
    \draw (6,1.25) node[above]{$\Phi$} ; 
    \draw (7, 0.5) rectangle (12, 2.5) ; 
    \draw [->](10.5, 0)--(10.5, -0.5);
    \draw (10.5, -0.25) node[right] {\text{Regular subdivision}};

    \draw (7, -1) rectangle (12, -3) ; 
    \draw [dashed](8.67, -1)--(8.67, -3) ; 
    \draw [dashed](10.34, -1)--(10.34, -3) ;
    \draw[dashed](7, -1.66)--(12, -1.66) ; 
    \draw[dashed](7,-2.33)--(12,-2.33) ; 

     \draw [->](6.5 , -2)--(5.5, -2) ;
     \draw(6 , -2)node[above] {$\Phi^{-1}$} ; 

    \draw[domain=0:5,smooth,variable=\x,blue, thick] 
        plot ({\x}, {0.5 + sin(deg(2*pi*\x/5))-3});

    \draw[domain=0:5,smooth,variable=\x,blue, thick] 
        plot ({\x}, {2.5 + sin(deg(2*pi*\x/5))-3});

    \draw[domain=0:5,smooth,variable=\x,dashed] 
        plot ({\x}, {0.5 + sin(deg(2*pi*\x/5))-2.34});
    \draw[domain=0:5,smooth,variable=\x,dashed] 
        plot ({\x}, {0.5 + sin(deg(2*pi*\x/5))-1.66});

    \draw[blue, thick](0,-0.5)--(0, -2.5) ; 
        \draw[blue, thick](5,-0.5)--(5, -2.5) ;

    \draw[dashed](1.66, {0.5 + sin(deg(2*pi*1.66/5))-3})--(1.66, {2.5+ sin(deg(2*pi*1.66/5))-3}) ; 
    \draw[dashed](3.32, {0.5 + sin(deg(2*pi*3.32/5))-3})--(3.32, {2.5+ sin(deg(2*pi*3.32/5))-3}) ; 
    \draw (2.5, 1.5) node{\huge$\omega$} ; 
    \draw(9.5, 1.5) node{\huge$\hat{\omega}$} ; 
    \end{tikzpicture}
    \caption{Partionning of $\omega$}
    \label{fig:partition_bis}
\end{figure}
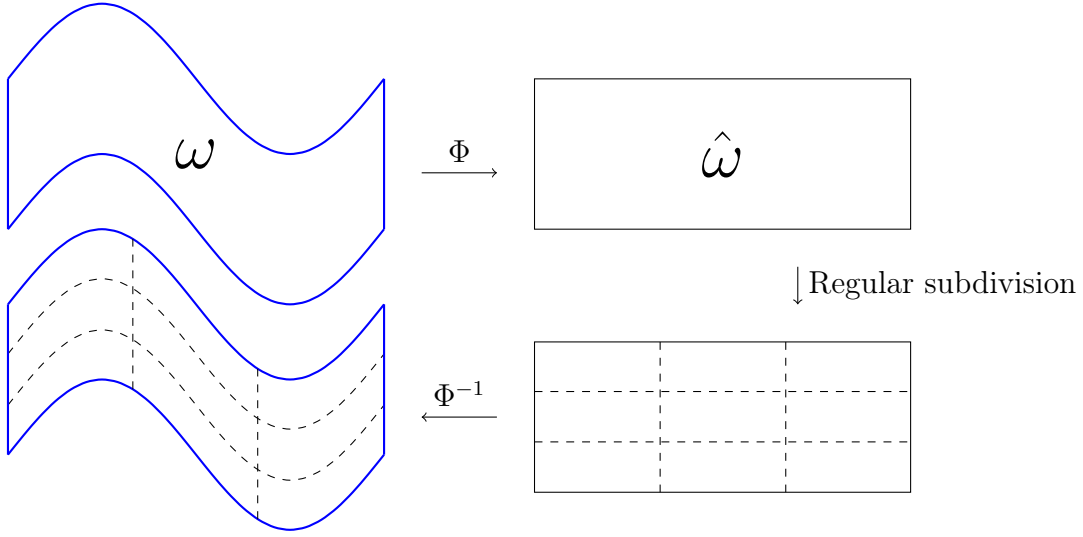
\begin{rem}
    It is important to see that the hypothesis "$\Phi (\omega)$ is convex" is not trivial at all. We have considered the case where $\Phi(\omega)$ is a parallelepiped, that is, $
\mathcal{I}(x_0) = \mathcal{I} \quad \text{for all } x_0 \in H,$
but this assumption is not necessary, since convexity alone is sufficient. 
However, for the sake of simplicity, by eventually extending $\omega$ to a larger domain $\omega'$ in the $\bb$-direction beyond $\partial \omega^+$ and $\partial \omega^-$, as illustrated In Figure \ref{fig:inflatetube}, 
we shall assume that $\Phi(\omega)$ is a parallelepiped. 
Such an extension does not affect the analysis, since the solution vanishes on $\omega' \setminus \omega$.

    \begin{figure}
        \centering
        \begin{tikzpicture}
               \draw[domain=0:5,smooth,variable=\x,blue, thick] 
        plot ({\x}, { sin(deg(2*pi*\x/5))});
         \draw[domain=0:5,smooth,variable=\x,blue, thick] 
        plot ({\x}, { 3+sin(deg(2*pi*\x/5))});
         \draw[thick, blue](0, 0)--(1, 1.5)--(0, 3) ; 
         \draw[thick, blue](5,0)--(4,1.5)--(5, 3) ; 

         \draw [very thick, ->](5.5, 1.5)--(6.5, 1.5) ; 

                      \draw[domain=7:12,smooth,variable=\x,blue, thick] 
        plot ({\x}, { sin(deg(2*pi*(\x-7)/5))});
         \draw[domain=7:12,smooth,variable=\x,blue, thick] 
        plot ({\x}, { 3+sin(deg(2*pi*(\x-7)/5))});
         \draw[blue, dashed](7, 0)--(8, 1.5)--(7, 3) ; 
         \draw[ blue, dashed](12,0)--(11,1.5)--(12, 3) ;

         \draw[blue, thick] (7, 0)--(7, 3) ; 
         \draw[blue, thick] (12, 0)--(12, 3) ; 

         \draw (2.5, 1.5) node{$\omega$} ; 
         \draw(9.5, 1.5) node{$\omega'$} ;

        \end{tikzpicture}
        \caption{Inflation of $\omega$ into $\omega'$}
        \label{fig:inflatetube}
    \end{figure}
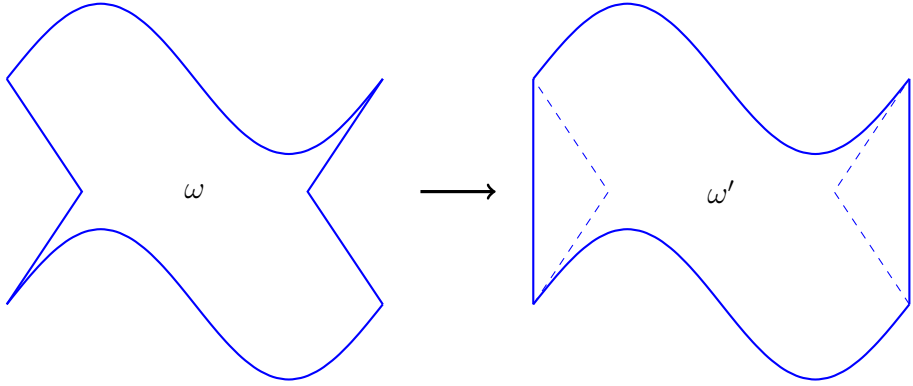
\end{rem}
We couldn't find an optimal bound for the Poincaré constant but for the special case of tube clusters, in the case were $\Phi:\Omega \rightarrow \hat{\Omega}$ is a diffeomorphism,  by working in the deformed space $\hat{\Omega}$ we were able to give a non trivial superior bound which introduces a constant $C_\Phi$. This way, we can proceed exactly as in the constant case of proposition \ref{theoreme} by considering nested tube clusters of the form of definition \ref{truetube} to prove proposition \ref{theoremenonconstant}.

\begin{prop}\quad\\
  Let $\eta>0$, $q\in (0,1)$, $p\in \mathbb{N}$ with $p>2$, $\Omega$ and $\tau\subset \Omega$ two tube clusters of the form of Definition \ref{truetube} and $\Phi$ a diffeomorphism such that $\Phi(\Omega)$ is a parallelepiped. For all tube clusters $\sigma\subset\Omega$ satisfying $\mathrm{diam}(\tau) \leq 2\eta \mathrm{dist}(\tau , \sigma)$, we can find a space $\mathrm{V}\subset\mathrm{L}^{2}(\tau)$
  where for some constant $C_{\mathrm{dim}}>0$
  \begin{equation}\label{DimensionBound2}
    \mathrm{dim}(\mathrm{V})\leq C_{\mathrm{dim}}p^{d+1}
  \end{equation}
  and such that for all right-hand sides $f\in  \mathrm{L}^2(\Omega)$ with $\mathrm{supp}(f)\subset \sigma$, the unique function
  $u\in\mathrm{H}^{1}_{0}(\Omega)$ satisfying $a(u,\varphi) = \langle f,\varphi\rangle_{\mathrm{L}^2(\Omega)}\;\forall\varphi\in \mathrm{H}^{1}_0(\Omega)$
  can be approximated on $\tau$ by $v\in \mathrm{V}$ satisfying the estimates for some constant $C>0$
  \begin{equation}
    \begin{aligned}
      & \Vert \nabla u-\nabla v\Vert_{\mathrm{L}^{2}(\tau)}\leq C \frac{p}{\mathrm{dist}(\tau, \sigma )} q^{p-1}\Vert f\Vert_{\mathrm{L}^2(\Omega)}\\
      & \Vert u-v\Vert_{\mathrm{L}^2(\tau)}\leq C q^{p}\Vert f\Vert_{\mathrm{L}^{2}(\Omega)}
    \end{aligned}.
  \end{equation}
\label{theoremenonconstant}
\end{prop}

 Our analysis relies on the fact that one can find $\delta >0$ such that $\tau_\delta$ is a tube cluster. However, this might not be the case depending on $\bb$. For instance on $(-1,1)^2$ with $\bb(x,y)=(x,y)$ the point $(0,0)$ belongs to every tube cluster (Figure \ref{fig:b=0}), which means that one can not find any $\tau_{\delta>0}$ and thus we cannot achieve our Caccioppoli estimate. This is one of the reason why we assumed $\bb$ not vanishing.

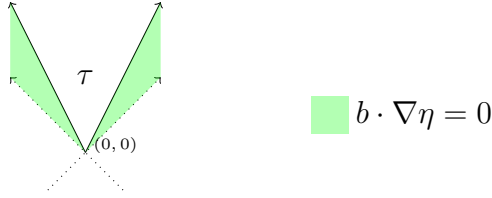
\begin{figure}
\centering
    \begin{tikzpicture}
        \draw[->](0,0) --(-1, 2) ; 
        \draw[->](0,0)-- (1, 2) ;  
        \draw[dotted, ->](-0.5, -0.5)--(1, 1); 
        \draw[dotted, ->](0.5, -0.5)--(-1, 1);
        \fill[green, opacity = 0.3] (0,0)--(-1, 1)--(-1,2)--(0,0) ; 
        \fill[green, opacity = 0.3] (0,0)--(1, 1)--(1,2)--(0,0) ; 
        \fill[green , opacity = 0.3] (3, 0.25) rectangle (3.5, 0.75) ; 
        \draw (4.5, 0.5) node{$b\cdot\nabla \eta = 0$} ;
        \draw(0.4,0.1) node {\tiny$(0,0)$} ; 
        \draw(0,1)node{$\tau$}  ; 
    \end{tikzpicture}
    \caption{tube cluster cannot be nested: $(0,0)\in\tau\cap\tau_\delta \  \forall \delta>0$}
    \label{fig:b=0}
\end{figure}

In this analysis, we attempted to transpose the independence observed in the constant case between the flow direction and a basis of $\tau_\perp$. To do so, it was necessary to generalize the role of a basis vector when the field $\bb$ is no longer constant. In doing so, we effectively introduced, without explicitly naming it, the notion of foliations. Future work could investigate whether approaching the problem directly through the framework of foliations yields further insights, especially regarding the minimal regularity that $\bb$ must satisfy in order for our analysis to apply. For now, however, we limit ourselves to our current method.

\subsection{Partitioning}
In the case where $\bb$ isn't constant it is not clear how to split into tube clusters. Here we are describing one strategy susceptible to work but we wont get in the theoretical requirements on $\bb$ to determine the scope of suitable vector fields. We consider $\tau$ some tube cluster of $\Omega$.

The idea is to inspire ourselves of the constant case as the role of $\tau_\perp$ is equivalent to the one of the hyperplane $H$. Indeed, from the definition of tube clusters, we know that there exists some connected domain $\Gamma_0$ of $\mathbb{R}^{d-1}$ such that any field line of $\tau$ go through $\Gamma_0$. We can introduce $\pi_{\bb}$ the projection along $\bb$ on $\Gamma_0$ defined as $\pi_{\bb}(x) = x_0$ where $x\in \bm{\phi}(x_0, \mathcal{I})$ and $x_0\in\Gamma_0$. This way ,eventually omitting a zero measure space as In Figure \ref{fig:explication1} and \ref{fig:explication2}, all the points of a field line get mapped on the same point exactly as in the constant case.
\begin{figure}
    \centering
    \begin{minipage}{0.45\textwidth}
        \centering
        \begin{tikzpicture}[scale=0.9]
            
            \foreach \angle in {0, 60, 120} {
                \draw[thick, ->, >=stealth, blue] (0,0) ++(\angle:2) arc[start angle=\angle, end angle=\angle+60, radius=2];
            }
            \foreach \angle in {180, 240, 300} {
                \draw[thick, ->, >=stealth, green] (0,0) ++(\angle:2) arc[start angle=\angle, end angle=\angle+60, radius=2];
            }
            \draw (2,0) node{$\times$};
            \draw (2,0.25) node[right]{$x$};
            \draw (-2,0) node{$\times$};
            \draw (-2,0.25) node[left]{$x'$};
            \draw[thick, purple] (-2.2, 0)--(2.2, 0);
            \draw(0, 0.3) node[purple]{$\Gamma_0$};

            \draw[purple, thick](4,-1.5)--(4, 1.5);
            \draw[thick , blue](4,-0.7)--(6, -0.7); 
            \draw[thick , green](4,0.7)--(6, 0.7); 
            \draw (4,0.7) node[left]{$\pi_{\bb}(x)$}; 
            \draw (4,-0.7) node[left]{$\pi_{\bb}(x')$}; 
            \draw (4,0.7) node{$\times$}; 
            \draw (4,-0.7) node{$\times$}; 
            \node[anchor=south west,inner sep=0,outer sep=0, text=blue] at (1.5,1.5) {$\tau$};
            \clip (1.5,1.5) rectangle (2,1.6);
            \node[anchor=south west,inner sep=0,outer sep=0, text=green] at (1.5,1.5) {$\tau$};
        \end{tikzpicture}
        \caption{Field lines unsuitable for our strategy}
        \label{fig:explication1}
    \end{minipage}
    \hspace{1cm}
    \begin{minipage}{0.45\textwidth}
        \centering
        \begin{tikzpicture}[scale=0.9]
            \draw[dashed, very thick] (-2.5, 0)node[above]{$\epsilon$}--(0,0);
            \foreach \angle in {0, 60, 120} {
                \draw[thick, ->, >=stealth, blue] (0,0) ++(\angle:2) arc[start angle=\angle, end angle=\angle+60, radius=2];
            }
            \foreach \angle in {180, 240, 300} {
                \draw[thick, ->, >=stealth, green] (0,0) ++(\angle:2) arc[start angle=\angle, end angle=\angle+60, radius=2];
            }
            \draw (2,0) node{$\times$};
            \draw (2,0.25) node[right]{$x$};
            \draw[thick, purple] (0, 0)--(2.2, 0);
            \draw(0, 0.3) node[purple]{$\Gamma_0$};

            \draw[purple, thick](4,-1.5)--(4, 1.5);
            \draw[thick , blue, dashed](4,0.7)--(6, 0.7); 
            \draw[thick , green, dashed](4.1,0.7)--(6, 0.7); 
            \draw (4,0.7) node[left]{$\pi_{\bb}(x)$}; 
            \draw (4,0.7) node{$\times$}; 
            \node[anchor=south west,inner sep=0,outer sep=0, text=blue] at (1.5,1.5) {$\tau$};
            \clip (1.5,1.5) rectangle (2,1.6);
            \node[anchor=south west,inner sep=0,outer sep=0, text=green] at (1.5,1.5) {$\tau$};
        \end{tikzpicture}
        \caption{Suitable $\Gamma_0$ on $\Omega\backslash\epsilon$}
        \label{fig:explication2}
    \end{minipage}
    \label{fig:explication}
\end{figure}

\noindent From there it appears that any splitting  $\psi:\Gamma_0\mapsto\{\gamma_1, \dots, \gamma_n\}$, with $\gamma_i\cap\gamma_j=\emptyset$ if $i\neq j$ and $\bigcup\gamma_i=\Gamma_0$, leads to tube clusters $\tau_i=\{ \bm{\phi}(x_0, s) \ | \ s\in \mathcal{I}(x_0),x_0\in \gamma_i\}$ partitioning $\tau$ (see Figure \ref{fig:partition2D}).

\begin{figure}
\begin{tikzpicture}[xscale=2 , yscale=1]
\centering
  \begin{scope}
    \draw[thick] (0, 0) rectangle (2, 2)node [right, scale=2]{$\tau$}; 
    \draw[purple, very thick](0, -0.2)--(0, 2.2)node[left]{$\Gamma_0$} ; 
    \draw[domain=0:2, smooth, variable=\x, blue] 
      plot ({\x}, {0.3+ 0.3 * sin(100 * pi * \x)});
      \draw[domain=0:2, smooth, variable=\x, cyan] 
      plot ({\x}, {0.8+ 0.3 * sin(100 * pi * \x)});
    \draw[domain=0:2, smooth, variable=\x, orange] 
      plot ({\x}, {1.5+ 0.3 * sin(100 * pi * \x)});
        \draw[orange] (0,1.5)node[left]{$x_2$} ; 
    \draw(0,1.5)node{$\times$} ; 
    \draw[cyan] (0,0.8)node[left]{$x_1$} ; 
    \draw(0,0.8)node{$\times$} ; 
    \draw[blue] (0,0.3)node[left]{$x_0$} ; 
    \draw(0,0.3)node{$\times$} ; 
  \end{scope}

  \begin{scope}[xshift=3.5cm]
    \draw[very thick, purple] (0, -0.2) -- (0, 2.2)node[left]{$\Gamma_0$}; 
    \draw[orange] (0,1.5)node[left]{$x_2$} ; 
    \draw(0,1.5)node{$\times$} ; 
    \draw[cyan] (0,0.8)node[left]{$x_1$} ; 
    \draw(0,0.8)node{$\times$} ; 
    \draw[blue] (0,0.3)node[left]{$x_0$} ; 
    \draw(0,0.3)node{$\times$} ; 
    \draw[dotted](-0.3, 1)--(0.3, 1) ; 
    \draw[very thick , green] (0.3, 1)--(0.3,1.5) node[right]{$\gamma_1$}--(0.3,2) ;
    \draw[very thick , red] (0.3, 0)--(0.3,0.5) node[right]{$\gamma_2$}--(0.3,1) ;
    \draw (1, 1) node {\Huge$\Leftrightarrow$} ;

    \end{scope}

  \begin{scope}[xshift=5.5cm]
  \fill[green!30] 
    (2, {1 + 0.3 * sin(100 * pi * 2)}) 
    -- plot[domain=2:0, smooth, variable=\x] ({\x}, {1 + 0.3 * sin(100 * pi * \x)}) 
    -- (0, 2) 
    -- (2, 2) 
    -- cycle;
     \fill[red!30] 
    (2, {1 + 0.3 * sin(100 * pi * 2)}) 
    -- plot[domain=2:0, smooth, variable=\x] ({\x}, {1 + 0.3 * sin(100 * pi * \x)}) 
    -- (0, 0) 
    -- (2, 0) 
    -- cycle;
    \draw[thick] (0, 0) rectangle (2, 2)node[right, scale= 2] {$\tau$}; 
    \draw[purple, very thick](0, -0.2)--(0, 2.2) ; 
    \draw[domain=0:2, smooth, variable=\x, blue] 
      plot ({\x}, {0.3+ 0.3 * sin(100 * pi * \x)});
      \draw[domain=0:2, smooth, variable=\x, cyan] 
      plot ({\x}, {0.8+ 0.3 * sin(100 * pi * \x)});
    \draw[domain=0:2, smooth, variable=\x, orange] 
      plot ({\x}, {1.5+ 0.3 * sin(100 * pi * \x)});
        \draw[orange] (0,1.5)node[left]{$x_2$} ; 
    \draw(0,1.5)node{$\times$} ; 
    \draw[cyan] (0,0.8)node[left]{$x_1$} ; 
    \draw(0,0.8)node{$\times$} ; 
    \draw[blue] (0,0.3)node[left]{$x_0$} ; 
    \draw(0,0.3)node{$\times$} ;
    \draw[domain=0:2, smooth, variable=\x, thick, dotted] 
      plot ({\x}, {1+ 0.3 * sin(100 * pi * \x)});
    \draw(2.2, 1.5) node[green]{$\tau_1$} ; 
    \draw(2.2, 0.5) node[red]{$\tau_2$} ; 
  \end{scope}

\end{tikzpicture}
    \caption{Different steps of the splitting}
    \label{fig:partition2D}
\end{figure}
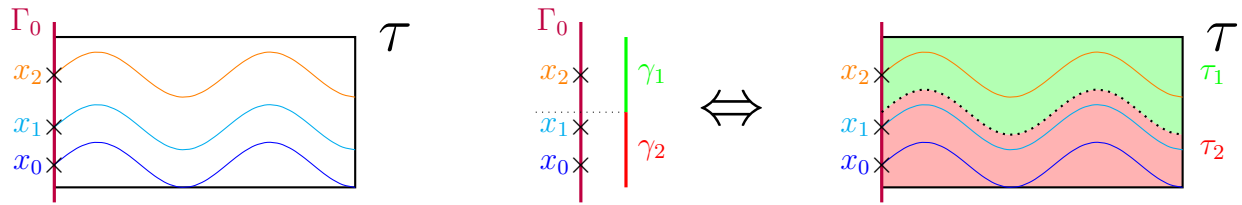
This way one can build a cluster tree on $\mathcal{T}(\Gamma_0)$ with any splitting and for any nodes $\tau'\in\mathcal{T}(\Gamma_0)$, the cluster $\tau=\pi_{\bb}^{-1}(\tau')$ is a tube cluster. Hence, building the cluster tree $\mathcal{T}(\Gamma_0)$ induces a tube cluster tree $\mathcal{T}(\tau) = \pi_{\bb}^{-1}(\mathcal{T}(\Gamma_0))$.

\noindent Our procedure is done in two steps, first we project along the field line and then we split. Unless one needs some advanced splitting, the algorithm complexity mainly comes from the projection along the field line as one as to find the intersection point of every field line with the hyperplane. Indeed, in the constant case a simple test on the scalar product of the points against $\bb^\perp$ was in off to split them in tube clusters. However, in the non constant case, one has to find the intersection of every points along $\bb$ with $\Gamma_0$ by iterating $p_{n+1} = p_n +\delta b(p_n)$ in the good direction until it reaches $\Gamma_0$. the splitting can be done as in Algorithm \ref{split_nonconstant}.
\begin{algorithm}
\caption{Split the points $t$ of $\tau$ into two parts in a deformed space}
\begin{algorithmic}[1]
\State $t'=\{\}$
\For{$p \in t$}
    \State Compute $p' \in \Gamma_0\subset \mathbb{R}^{d-1}$ by iterating $p_{n+1} = p_n + \delta\, \bb(p_n)$
    \State Add $p'$ to $t'$
\EndFor
\State Compute a vector $\bm{v}$ supported in $\Gamma_0$
\State Sort the points in $t'$ by increasing value of $\langle p', \bm{v} \rangle$, using:
\[
\texttt{sort}(t'.\texttt{begin()}, t'.\texttt{end()}, (p_1, p_2) \mapsto \langle p_1, \bm{v} \rangle < \langle p_2,\bm{v} \rangle)
\]
\State Let $t_1$ be the first half of $t'$, and $t_2$ the second half
\State \Return $(t_1, t_2)$
\end{algorithmic}
\label{split_nonconstant}

\end{algorithm}

The main computational cost of the algorithm lies in determining the intersection points between the field lines and the reference manifold $\Gamma_0$. This step involves iterating along the vector field $\bb$ for each point, which can become expensive for large datasets. In principle, the complexity could be significantly reduced by computing only one representative point per field line, or by leveraging a diffeomorphism that straightens the field lines. However, in the absence of a general method to construct such a transformation, we adopt a pragmatic approach: we select a sufficiently large step size $\delta$ to balance accuracy and computational efficiency when determining the intersection points.

Equipped with such a cluster tree, we get from the subsection 3.1 that we fall in the scope of proposition \ref{gros_th}.

\section{Numerical experiments}
\label{section4}

In order to test the performances of our new approach we are going to look at the resolution of the system \[Ax=y\] by the means of hierarchical matrices whose cluster tree are tube cluster tree.
Namely we will compute the $\mathcal{H}$-LU factorization associated with tube cluster tree performing $A\approx L_\mathcal{H}U_\mathcal{H}$ where $L_\mathcal{H}$ and $U_\mathcal{H}$ are triangular hierarchical matrices. With this factorization, we can use forward and backward substitution to find $x$ such that $Ax= y$ and we look at the error \[ err=\frac{\Vert x- \mrm{solve}(L_\mathcal{H}U_\mathcal{H}, Ax)\Vert}{\Vert x\Vert} .\] For our tests we use \texttt{HTOOL}, a \texttt{C++} library developed by Pierre Marchand, to perform $\mathcal{H}$-matrix arithmetic on matrices obtained using \texttt{FreeFem++}. We will highlight how with our clustering, the hierarchical approach leads to very promising results.

We consider the model problem: find $u \in H_0^1(\Omega)$ on the square domain $\Omega = [-1,1]\times[-1,1]$ such that
\begin{equation}
\begin{cases}
-\varepsilon \Delta u + \mathbf{b} \cdot \nabla u + 2u = (1 - |x|)(1 - |y|), & \text{in } \Omega, \\[5pt]
u = 0, & \text{on } \partial \Omega.
\end{cases}
\label{Pb:Dirichlet}
\end{equation}
and we are particularly interested in how the behavior depends on the Péclet number and the number of degrees of freedom.
The mesh (Delaunay's triangulation) and the stiffness matrix are produced by \texttt{FreeFem}. We study the influence of the diffusion parameter $\varepsilon$ by increasing the mesh size for different vector fields $\bb$. We will denote the hierarchical approximation of matrix $A$ by $A_\mathcal{H}$. We consider $\mathcal{\ell}^+$ the set the low rank blocks (admissible leaves) and $\mathcal{\ell}^-$ the one of the dense blocks (non-admissible leaves) and for a matrix $M$ we take $\#\mathrm{row}(M)$ the numbers of rows of $M$ (respectively $\#\mathrm{col}(M)$ its number of columns). We call compression of an $\mathcal{H}$-matrix the proportion of low ranks leaves, it is given by \[compr(A_\mathcal{H})= 1- \frac{\sum_{M\in\mathcal{\ell}^+}\mrm{rank}(M)(\#\mrm{row}(M)+\#\mrm{col}(M)) +\sum_{N\in\mathcal{\ell}^-} \#\mrm{row}(N)\#\mrm{col}(N) }{\#\mrm{row}(A) \#\mrm{col}(A)}\]
Because of the triangular nature, $L_\mathcal{H}$ and $U_\mathcal{H}$ can be written in one matrix that we will note $L_\mathcal{H}\backslash U_\mathcal{H}$. We are going to plot the compression of $L_\mathcal{H}\backslash U_\mathcal{H}$, the error on the resolution of the system using this factorization, and the time of the factorization.  In our tests we take $\bb=(1,0), \ \bb=(1, (0.5-y)cos(4\pi x))$ and $\bb=(0.5(y+0.8)\exp(x))$ (which all meet the condition \eqref{conditionbeta} on the computational domain), and the diffusion parameter $\varepsilon \in\{1, 10^{-2}, 10^{-4}, 10^{-6} \}$. The low rank approximations are done using adaptive cross approximation with a tolerance of $10^{-6}$ and we take the admissibility parameter $\eta=1$.

\noindent By straightening the fields line if needed, we partition the domain with the suited tube clusters as illustrated In Figure \ref{fig:clusteringtest}. As mentioned in remark \ref{restriction}, the minimal size of the cluster $minsize$ cannot be arbitrarily small with respect to the characteristic length of the advection stream $length$. In our tests we take $minsize= lenght/5$.

\begin{figure}
    \centering
    \begin{subfigure}[b]{0.45\linewidth}
        \centering
        \includegraphics[width=\linewidth, height=0.55\linewidth, trim = 0 0 0 60, clip]{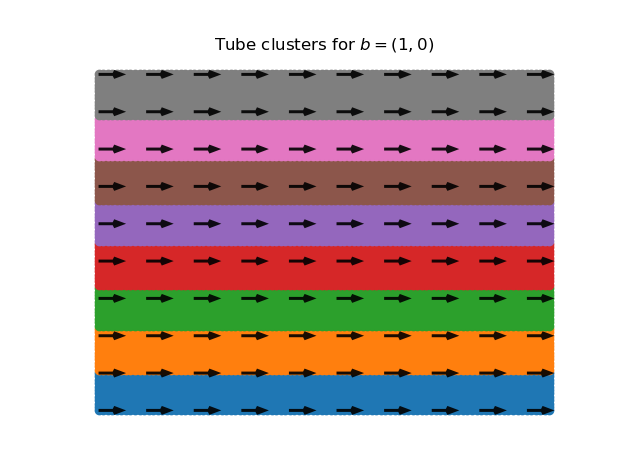}
        \caption{$\bb = (1, 0)$}
        \label{fig:err(1,0)}
    \end{subfigure}
    \begin{subfigure}[b]{0.45\linewidth}
        \centering
        \includegraphics[width=\linewidth,height=0.55\linewidth, trim = 0 0 0 60, clip]{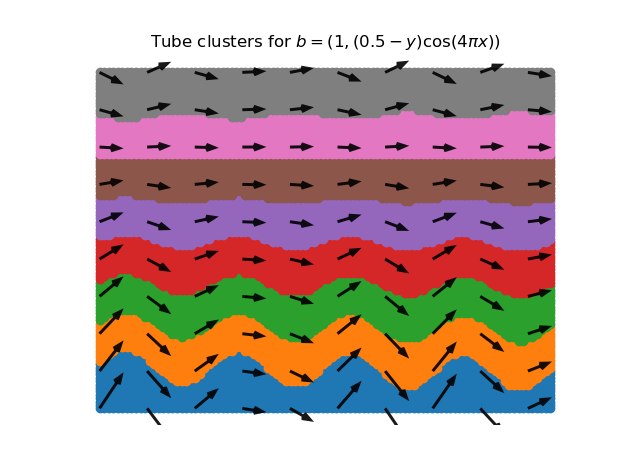}
        \caption{$\bb = (1, y\cos(x))$}
        \label{fig:err(1,ycos(x))}
    \end{subfigure}    
    \begin{subfigure}[b]{0.45\linewidth}
        \centering
        \includegraphics[width=\linewidth,height=0.55\linewidth, trim = 0 0 0 60, clip]{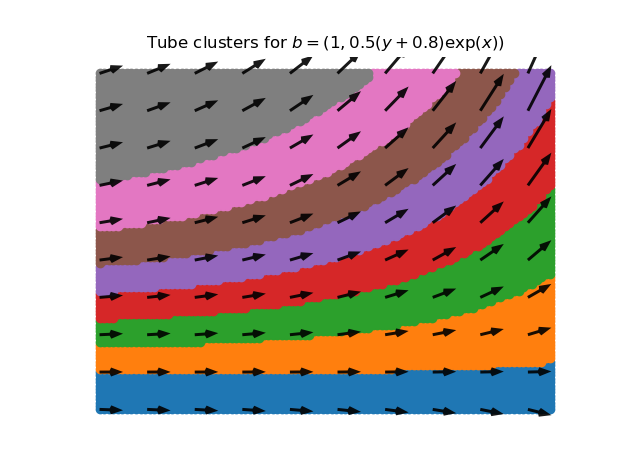}
        \caption{$\bb = (1, y\exp(x))$}
        \label{fig:err(1,yexp(x))}
    \end{subfigure}
    
    \caption{Tube cluster tree at depth $3$ for different convection $\bb$.}
    \label{fig:clusteringtest}
    \end{figure}

Equipped with those cluster tree we can compute the hierarchical $LU$ factorization of the matrices. In Figure \ref{fig:timeLU} we can see the computation time of the factorization, which seems to have a time complexity of $\mathcal{O}(n\log(n)^2)$. In Figure \ref{fig:comprLU} the compression of the factorization $L_\mathcal{H}\backslash U_\mathcal{H}$ and in Figure \ref{fig:errorLU} the error on the resolution of the system $Ax=y$ using this factorization. It appears that the error is bounded by the tolerance of the adaptive cross approximation (here $10^{-6}$) independently of the problem size.
\begin{figure}
    \centering
    \begin{subfigure}[b]{0.45\linewidth}
        \centering
        \includegraphics[width=\linewidth,height=0.55\linewidth, trim = 0 0 0 22, clip]{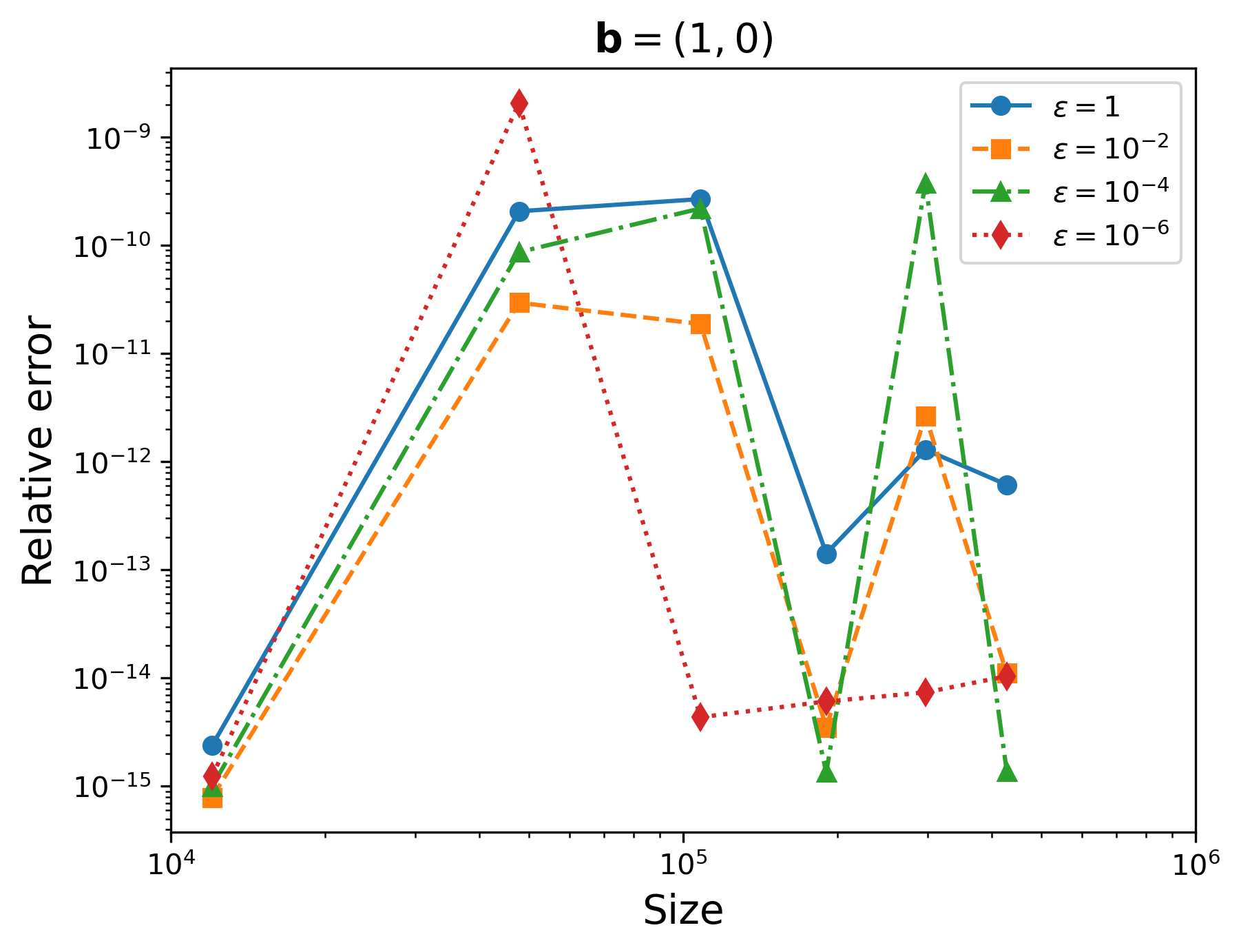}
        \caption{$\bb = (1, 0)$}
        \label{fig:err(1,0)}
    \end{subfigure}
    \begin{subfigure}[b]{0.45\linewidth}
        \centering
        \includegraphics[width=\linewidth,height=0.55\linewidth, trim = 0 0 0 22, clip]{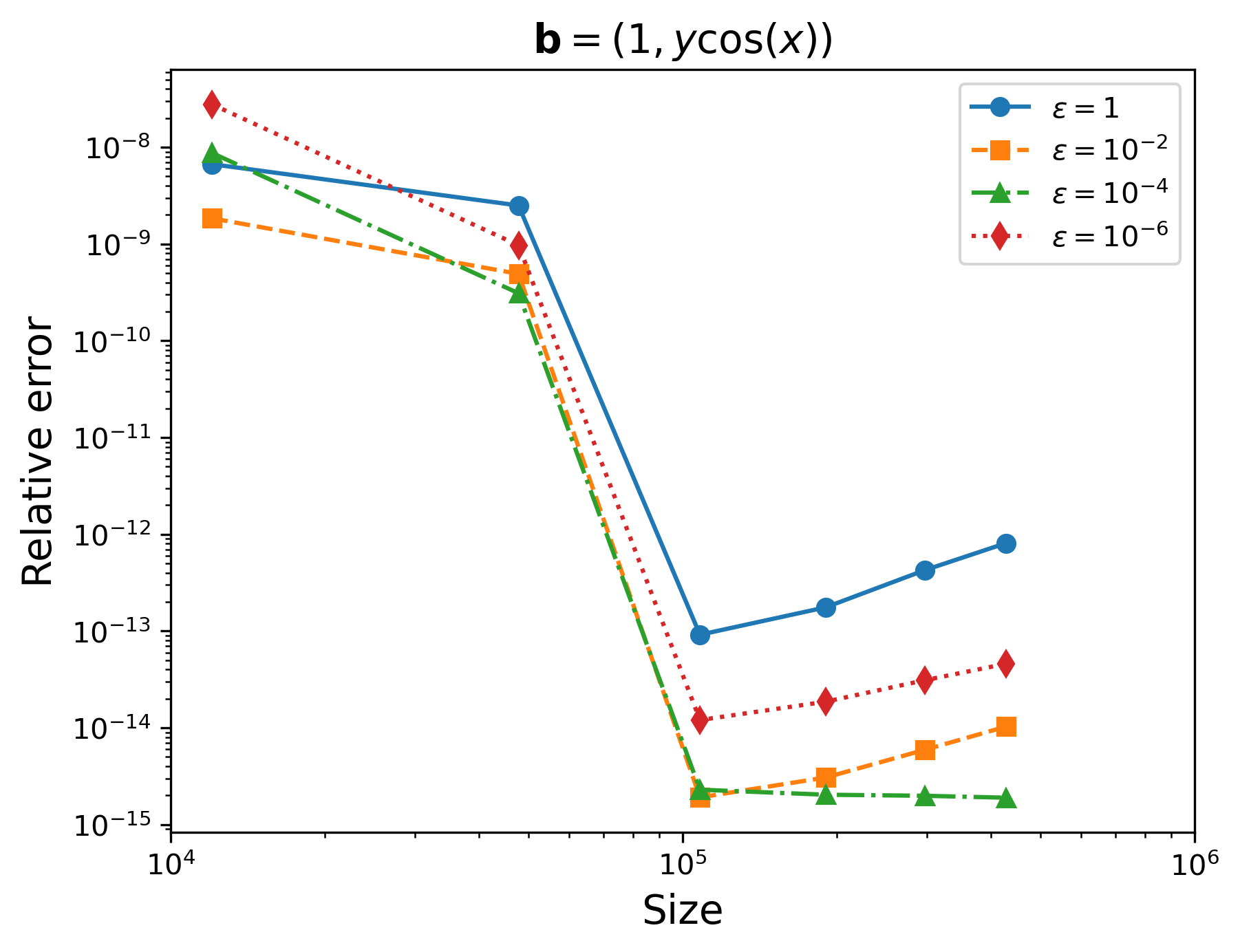}
        \caption{$\bb = (1, y\cos(x))$}
        \label{fig:err(1,ycos(x))}
    \end{subfigure}    
    \begin{subfigure}[b]{0.45\linewidth}
        \centering
        \includegraphics[width=\linewidth,height=0.55\linewidth, trim = 0 0 0 22, clip]{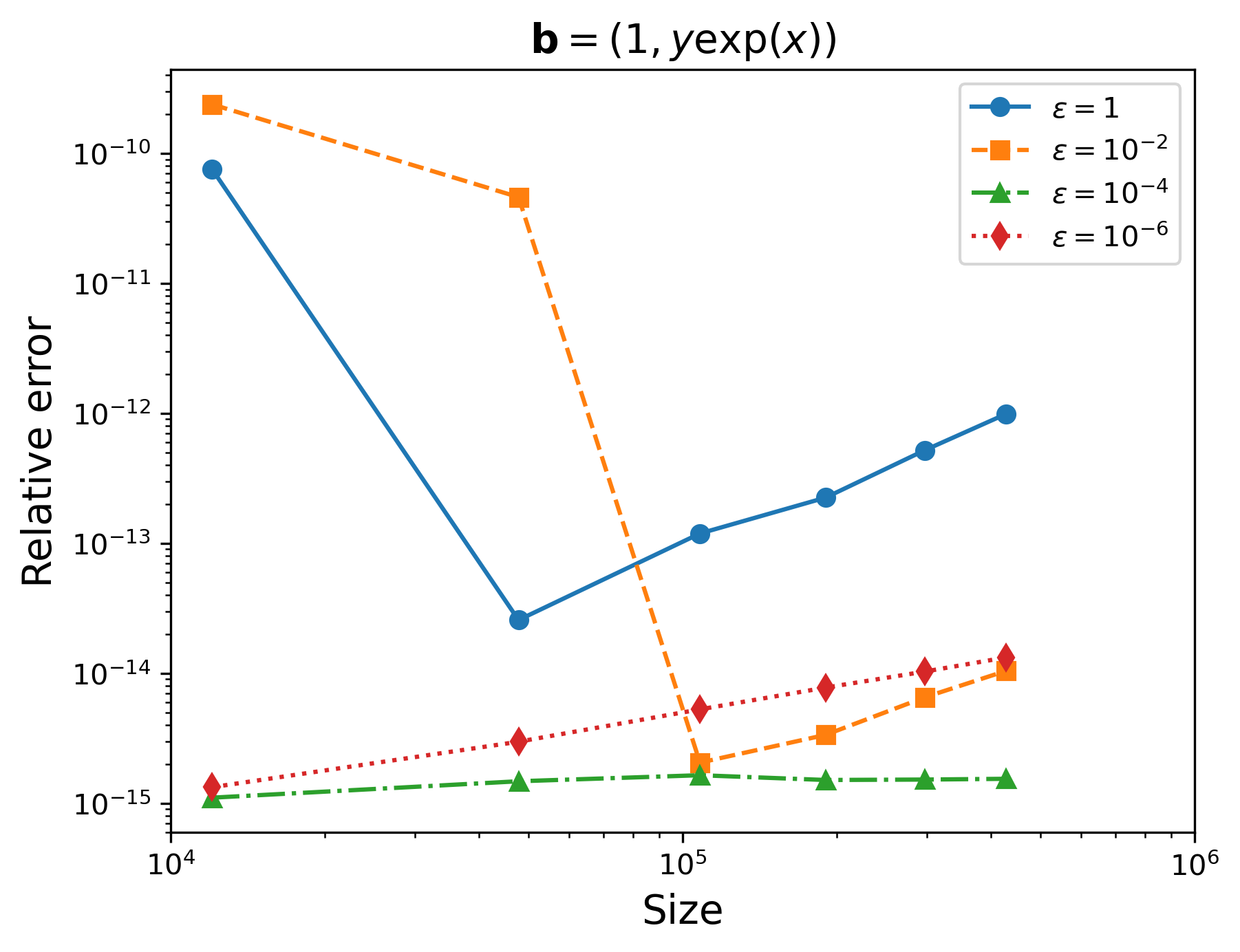}
        \caption{$\bb = (1, y\exp(x))$}
        \label{fig:err(1,yexp(x))}
    \end{subfigure}
    
    \caption{Precision of the $\mathcal{H}$-LU factorization for FreeFem matrices of problem \eqref{Pb:Dirichlet} (TGV = -2) for different convection $\bb$.}
    \label{fig:errorLU}
\end{figure}

\begin{figure}
    \centering
    \begin{subfigure}[b]{0.45\linewidth}
        \centering
        \includegraphics[width=\linewidth,height=0.55\linewidth, trim = 0 0 0 22, clip]{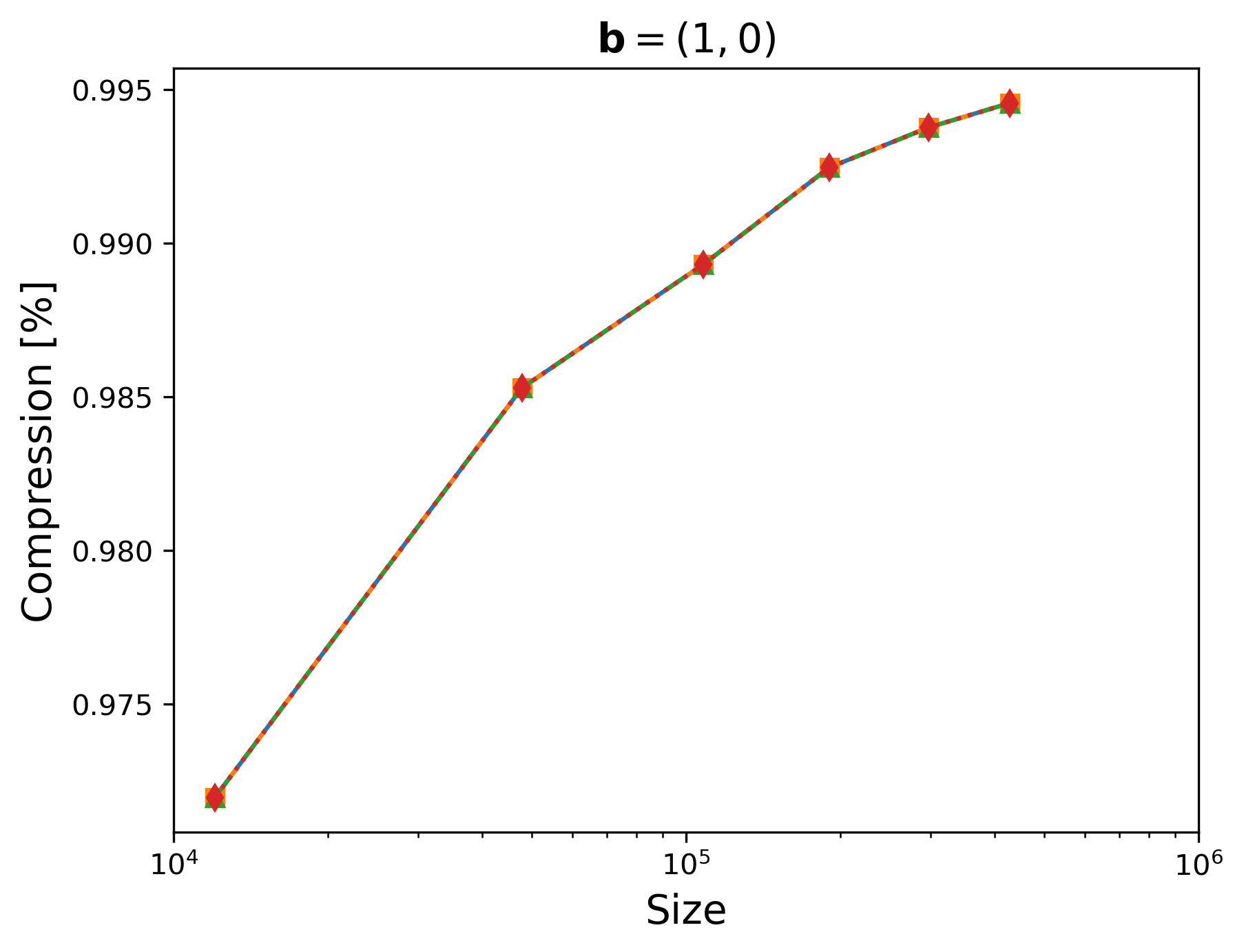}
        \caption{$\bb = (1, 0)$}
        \label{fig:err(1,0)}
    \end{subfigure}
    \begin{subfigure}[b]{0.45\linewidth}
        \centering
        \includegraphics[width=\linewidth,height=0.55\linewidth, trim = 0 0 0 22, clip]{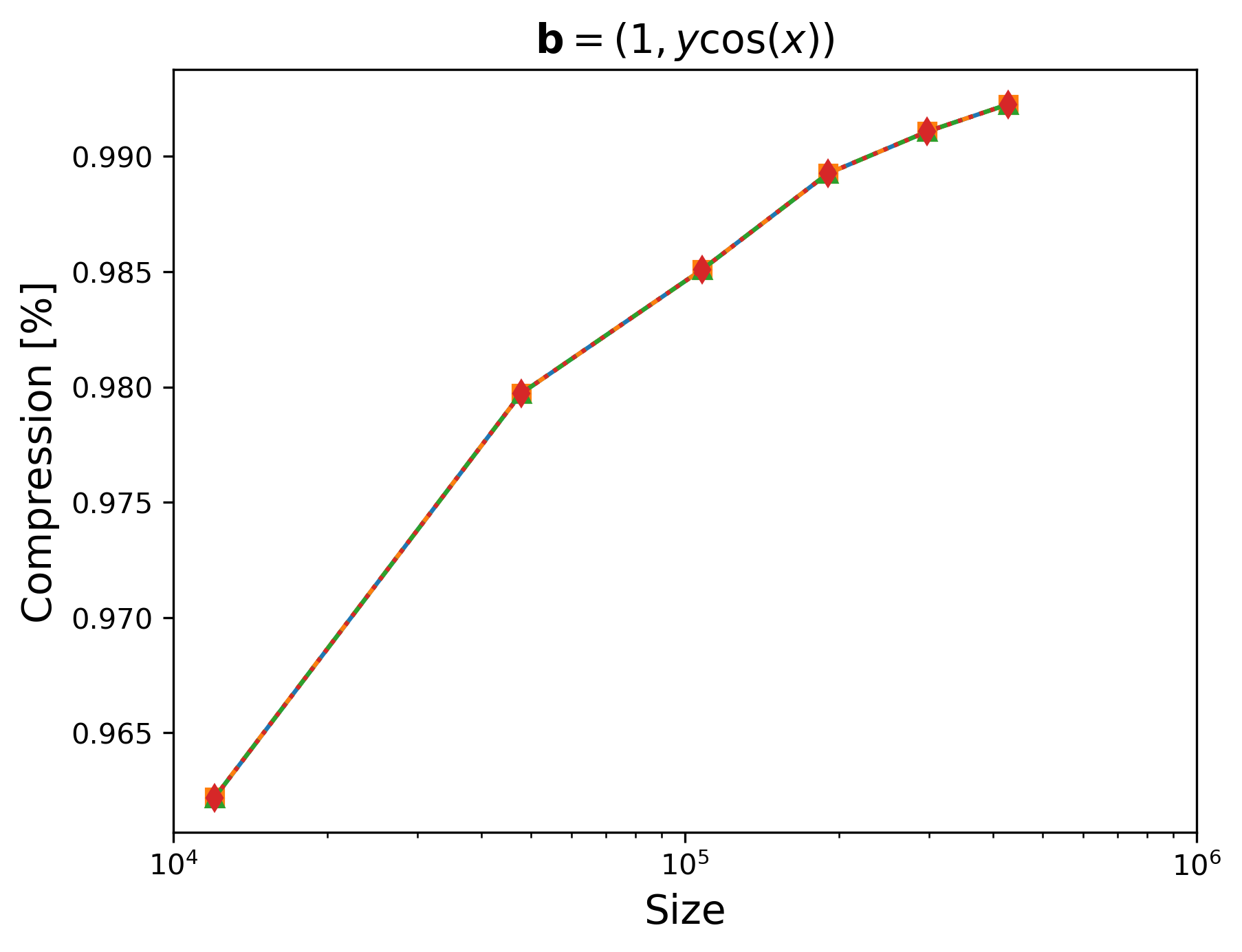}
        \caption{$\bb = (1, y\cos(x))$}
        \label{fig:err(1,ycos(x))}
    \end{subfigure}    
    \begin{subfigure}[b]{0.45\linewidth}
        \centering
        \includegraphics[width=\linewidth,height=0.55\linewidth, trim = 0 0 0 20, clip]{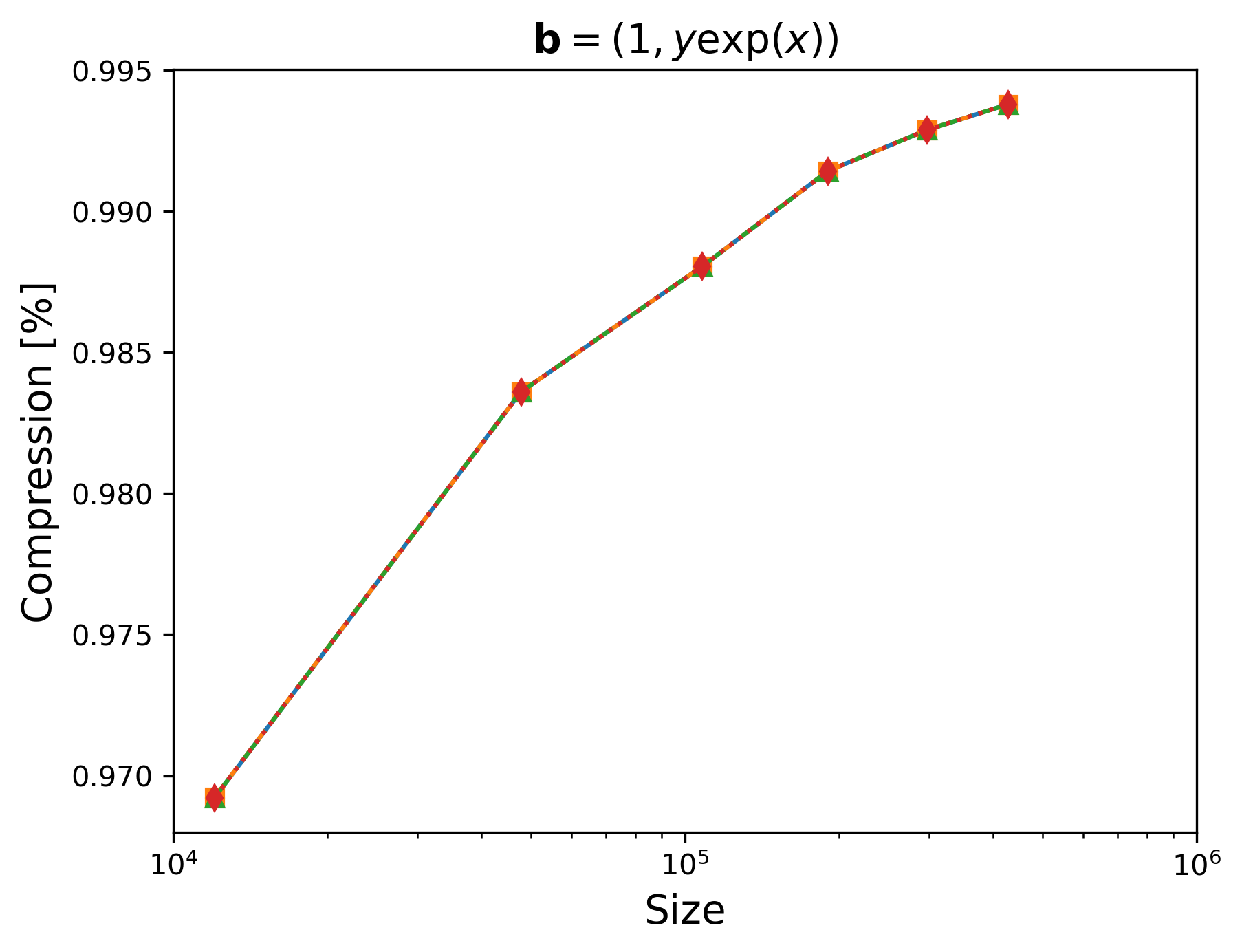}
        \caption{$\bb = (1, y\exp(x))$}
        \label{fig:err(1,yexp(x))}
    \end{subfigure}
    \caption{Compression of $L_\mathcal{H}U_\mathcal{H}$ for FreeFem matrices of problem \eqref{Pb:Dirichlet} (TGV = -2) for different \\ convection~$\bb$.}
    \label{fig:comprLU}
\end{figure}
\begin{figure}
    \centering
    \begin{subfigure}[b]{0.45\linewidth}
        \centering
        \includegraphics[width=\linewidth,height=0.55\linewidth, trim = 0 0 0 0, clip]{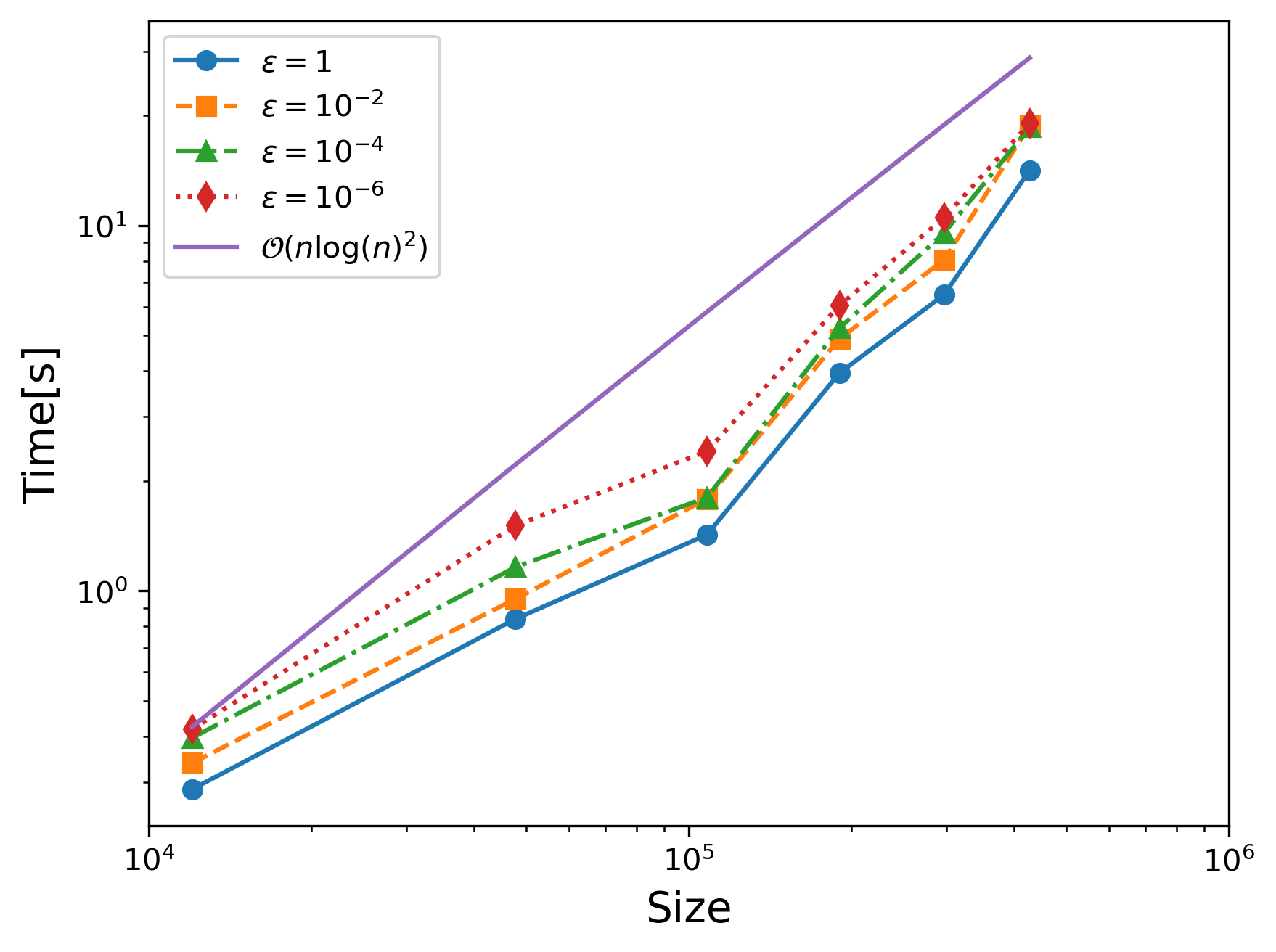}
        \caption{$\bb = (1, 0)$}
        \label{fig:err(1,0)}
    \end{subfigure}
    \begin{subfigure}[b]{0.45\linewidth}
        \centering
        \includegraphics[width=\linewidth,height=0.55\linewidth, trim = 0 0 0 0, clip]{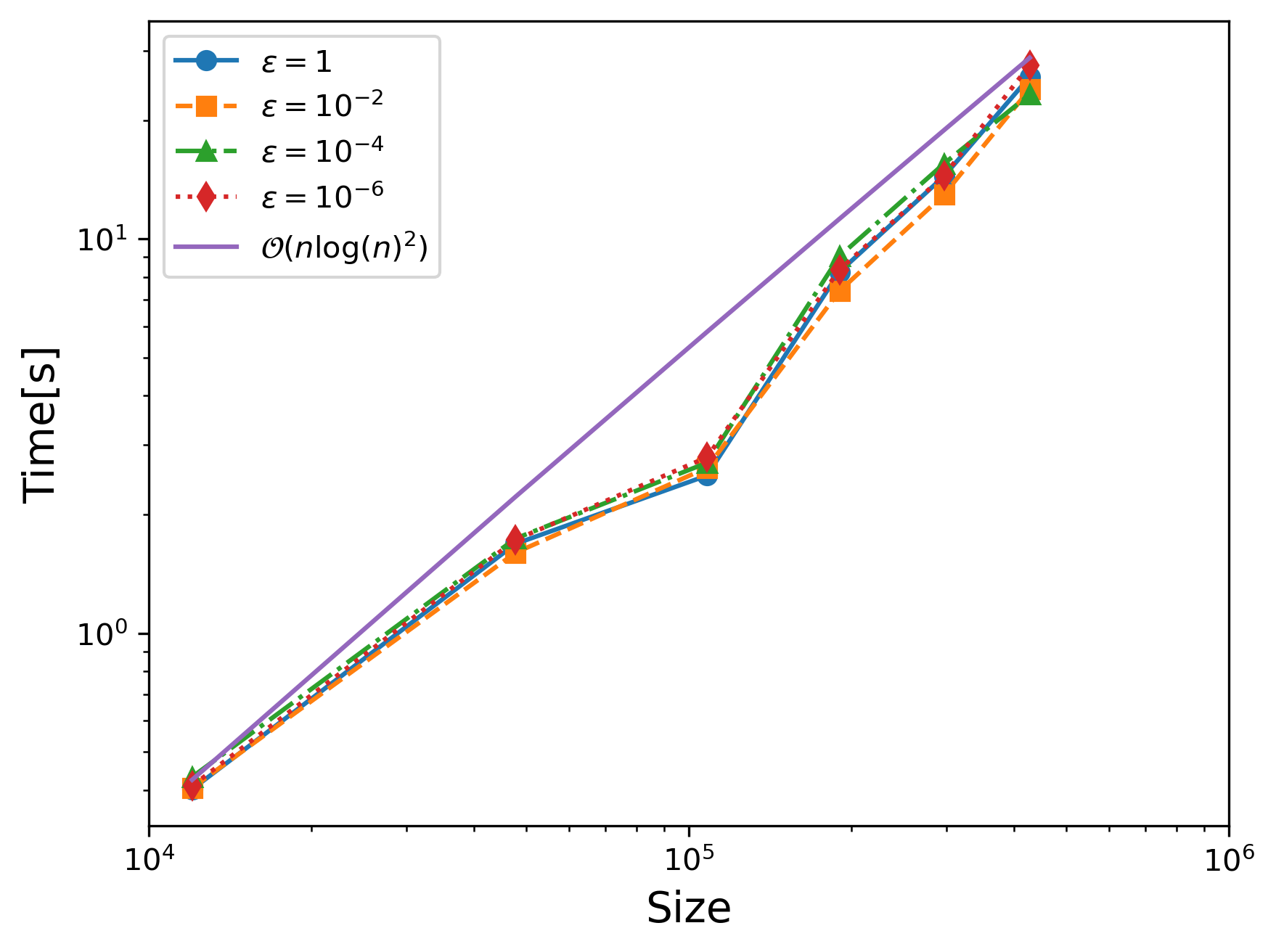}
        \caption{$\bb = (1, y\cos(x))$}
        \label{fig:err(1,ycos(x))}
    \end{subfigure}    
    \begin{subfigure}[b]{0.45\linewidth}
        \centering
        \includegraphics[width=\linewidth,height=0.55\linewidth, trim = 0 0 0 0, clip]{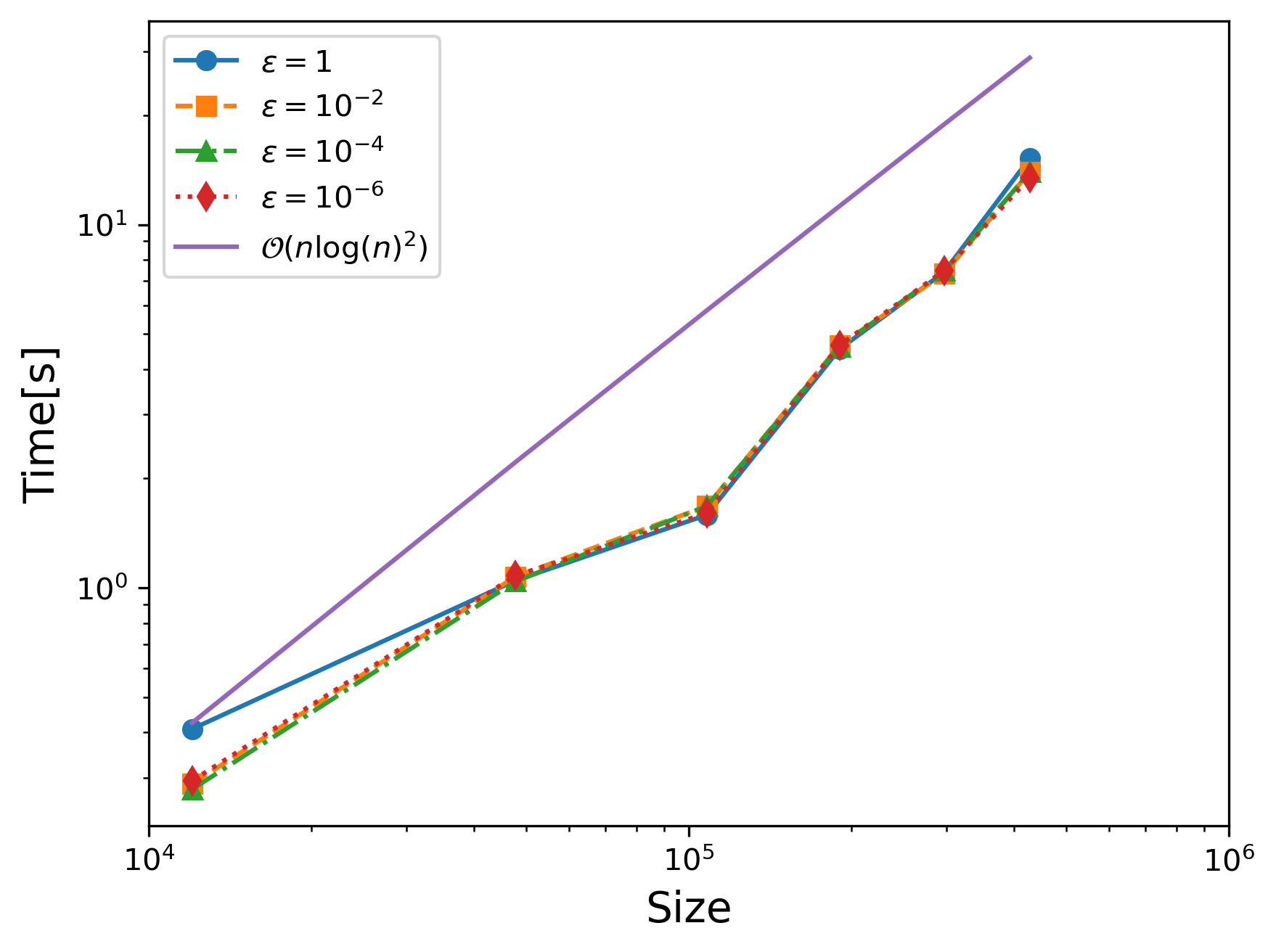}
        \caption{$\bb = (1, y\exp(x))$}
        \label{fig:err(1,yexp(x))}
    \end{subfigure}
    
    \caption{Time of the factorization $L_\mathcal{H}U_\mathcal{H}$ for FreeFem matrices of problem \eqref{Pb:Dirichlet} (TGV = -2) for different convection $\bb$.}
    \label{fig:timeLU}
\end{figure}
 
We did the same experiment with Neumann boundary conditions and the symmetrized formulation of the problem \eqref{Pb:Neumann},
\begin{equation}
\begin{cases}
 -\varepsilon \Delta u + \bb \cdot \nabla u + 2u = (1-|x|)(1-|y|), & \text{in } \Omega, \\
 \dfrac{\partial u}{\partial \mathbf{n}_\Omega} = 0, & \text{on } \partial \Omega.
\end{cases}
\label{Pb:Neumann}
\end{equation} 
In Figure \ref{fig:timeLUneumann} we show the time of the $\mathcal{H}$-LU factorization, in Figure \ref{fig:comprLUneumann} the compression ratio of $L_\mathcal{H}\backslash U_\mathcal{H}$ and in Figure \ref{fig:errorLUneumann} the error of the resolution $Ax=y$ using this factorization.

\begin{figure}
    \centering
    \begin{subfigure}[b]{0.45\linewidth}
        \centering
        \includegraphics[width=\linewidth,height=0.55\linewidth, trim = 0 0 0 22, clip]{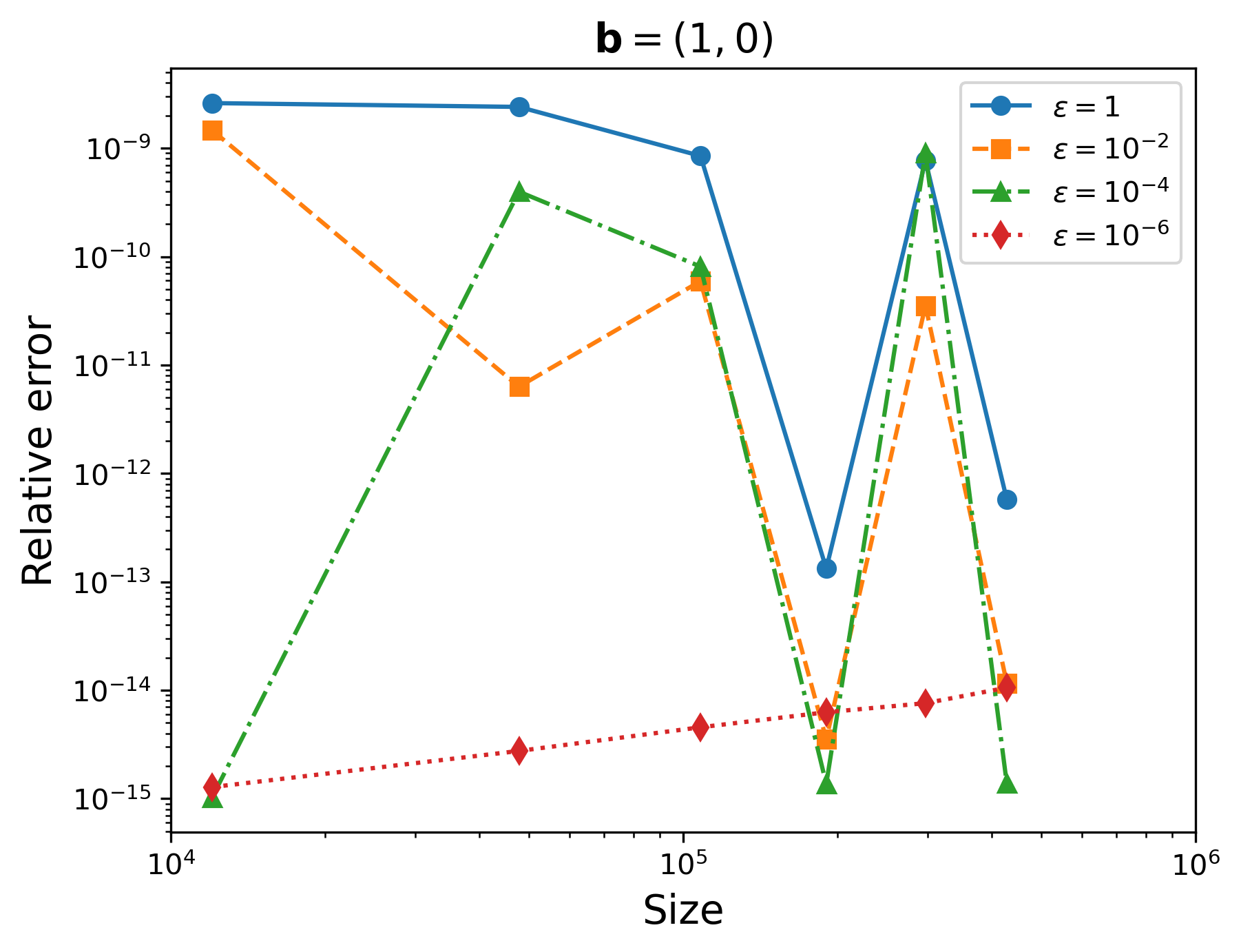}
        \caption{$\bb = (1, 0)$}
        \label{fig:err(1,0)}
    \end{subfigure}
    \begin{subfigure}[b]{0.45\linewidth}
        \centering
        \includegraphics[width=\linewidth,height=0.55\linewidth, trim = 0 0 0 20, clip]{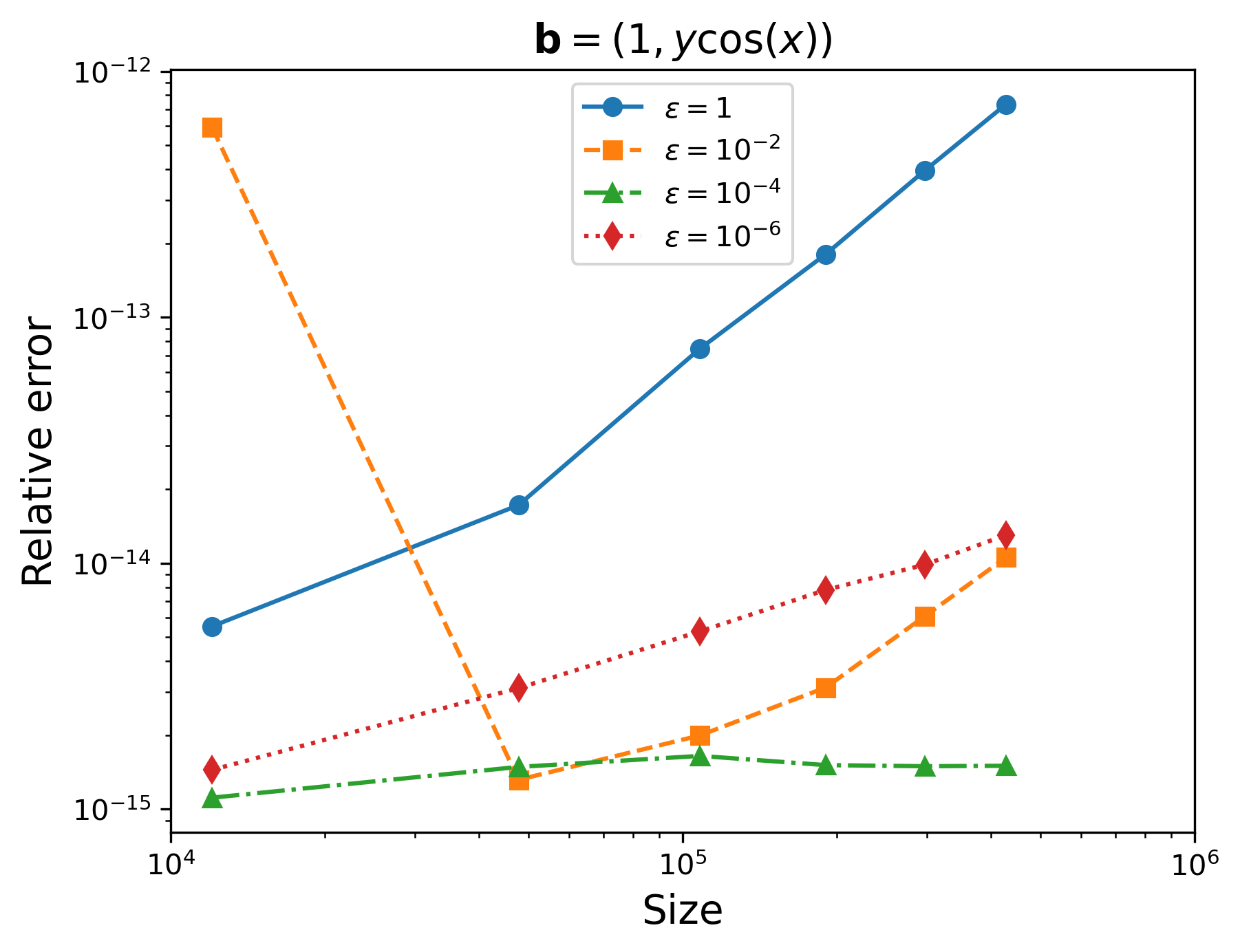}
        \caption{$\bb = (1, y\cos(x))$}
        \label{fig:err(1,ycos(x))}
    \end{subfigure}    
    \begin{subfigure}[b]{0.45\linewidth}
        \centering
        \includegraphics[width=\linewidth,height=0.55\linewidth, trim = 0 0 0 22, clip]{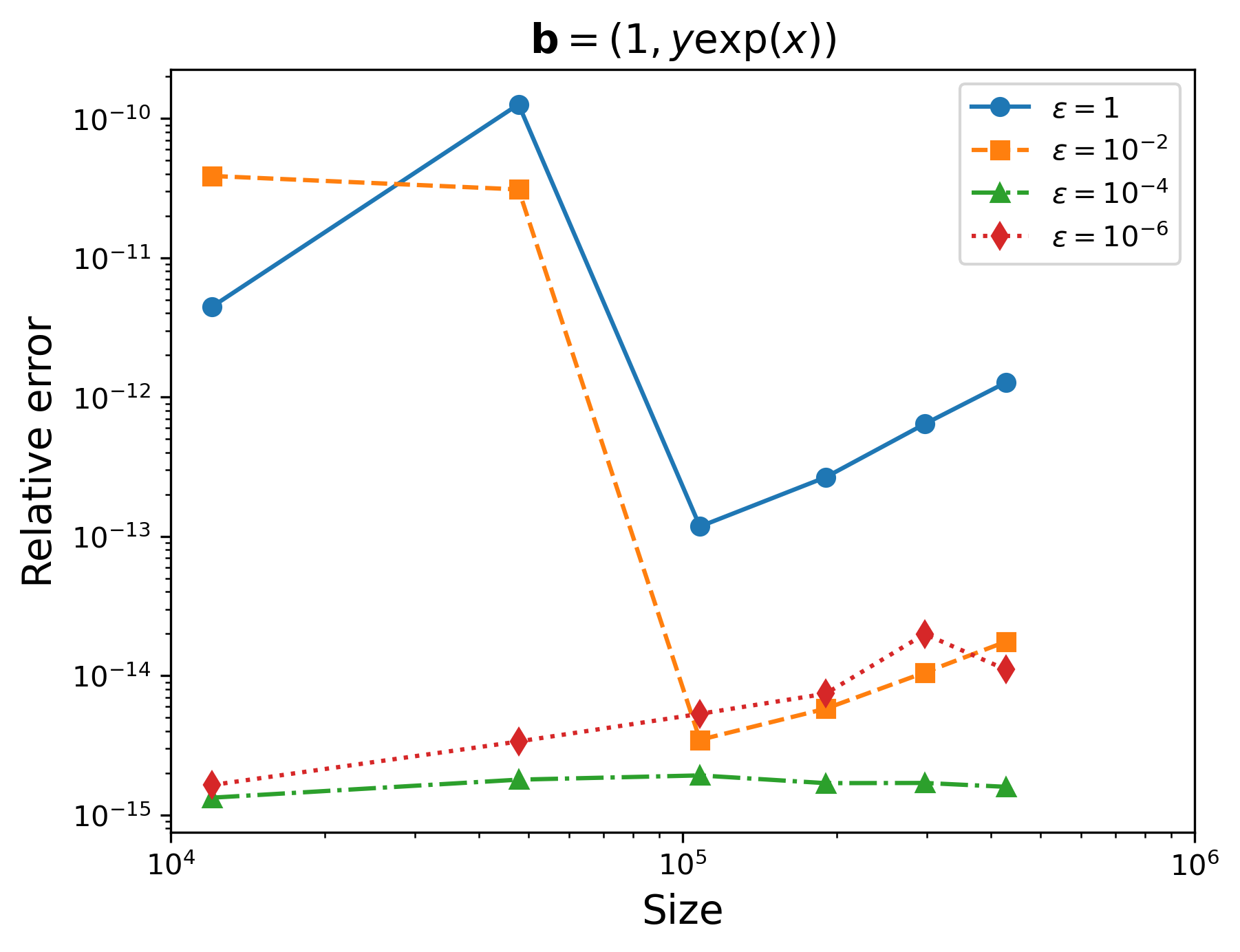}
        \caption{$\bb = (1, y\exp(x))$}
        \label{fig:err(1,yexp(x))}
    \end{subfigure}
    
    \caption{Precision of the $\mathcal{H}$-LU factorization for FreeFem matrices of problem \eqref{Pb:Neumann} for different \\convection~$\bb.$}
    \label{fig:errorLUneumann}
\end{figure}
\begin{figure}
    \centering
    \begin{subfigure}[b]{0.45\linewidth}
        \centering
        \includegraphics[width=\linewidth,height=0.55\linewidth, trim = 0 0 0 22, clip]{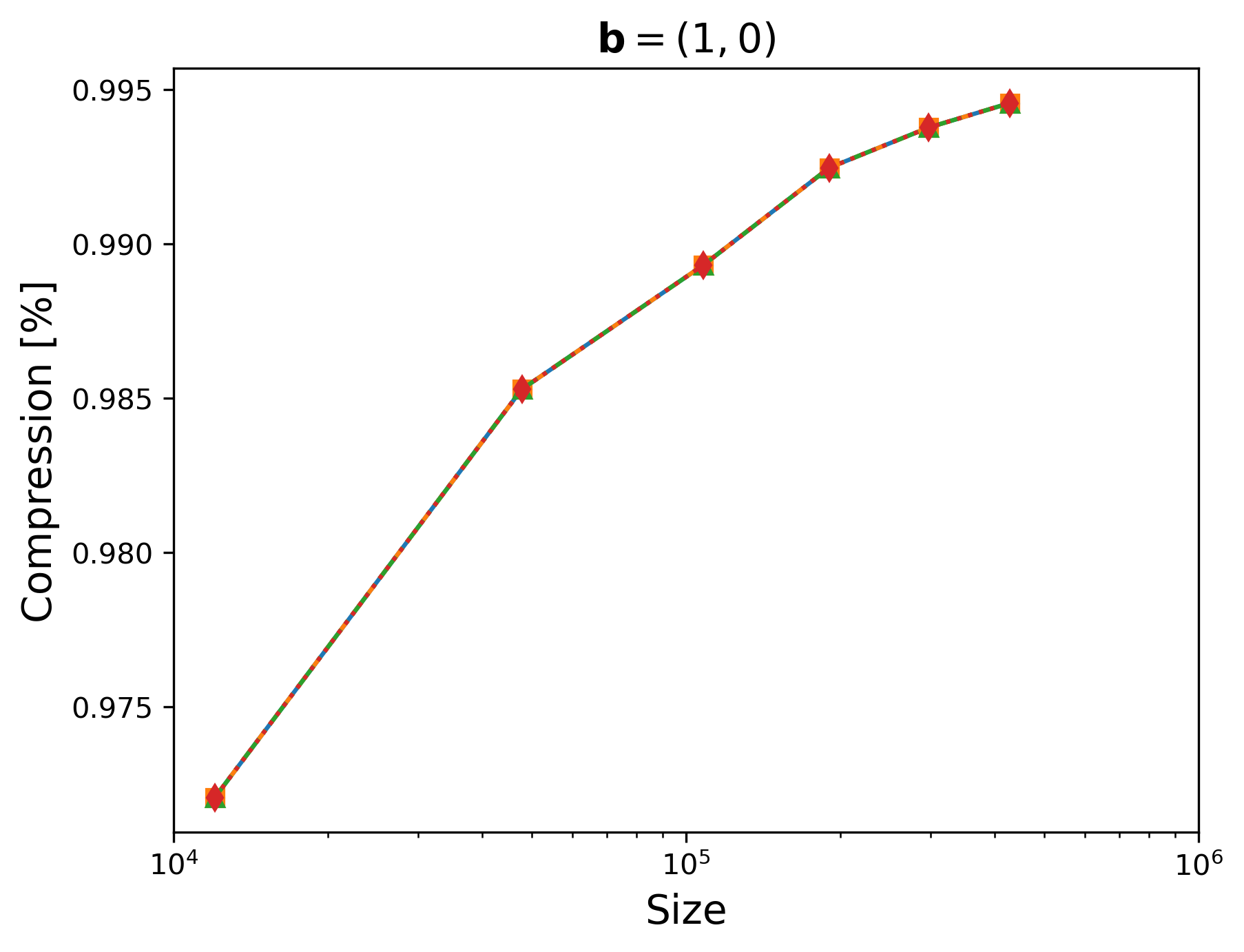}
        \caption{$\bb = (1, 0)$}
        \label{fig:err(1,0)}
    \end{subfigure}
    \begin{subfigure}[b]{0.45\linewidth}
        \centering
        \includegraphics[width=\linewidth,height=0.55\linewidth, trim = 0 0 0 22, clip]{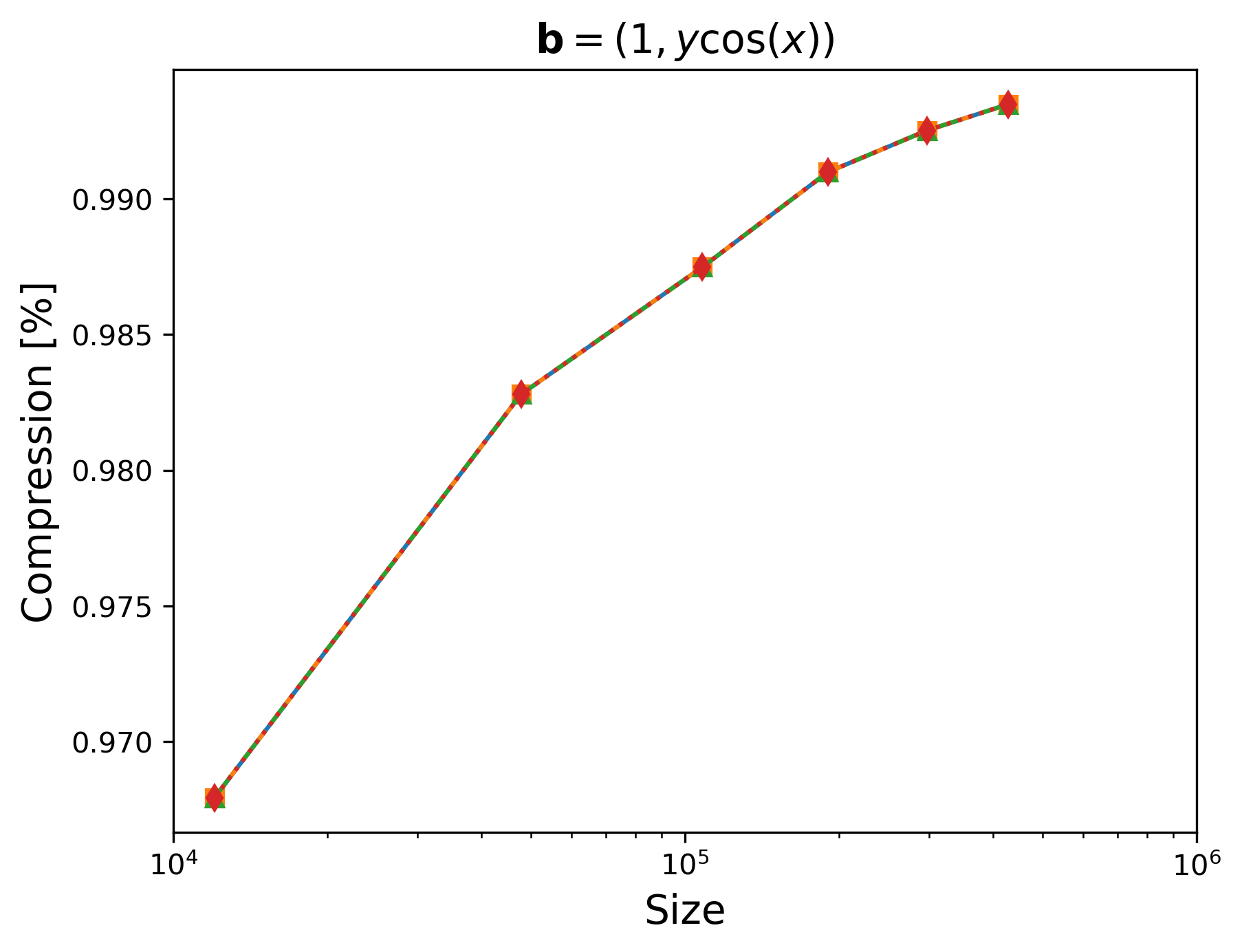}
        \caption{$\bb = (1, y\cos(x))$}
        \label{fig:err(1,ycos(x))}
    \end{subfigure}    
    \begin{subfigure}[b]{0.45\linewidth}
        \centering
        \includegraphics[width=\linewidth,height=0.55\linewidth, trim = 0 0 0 22, clip]{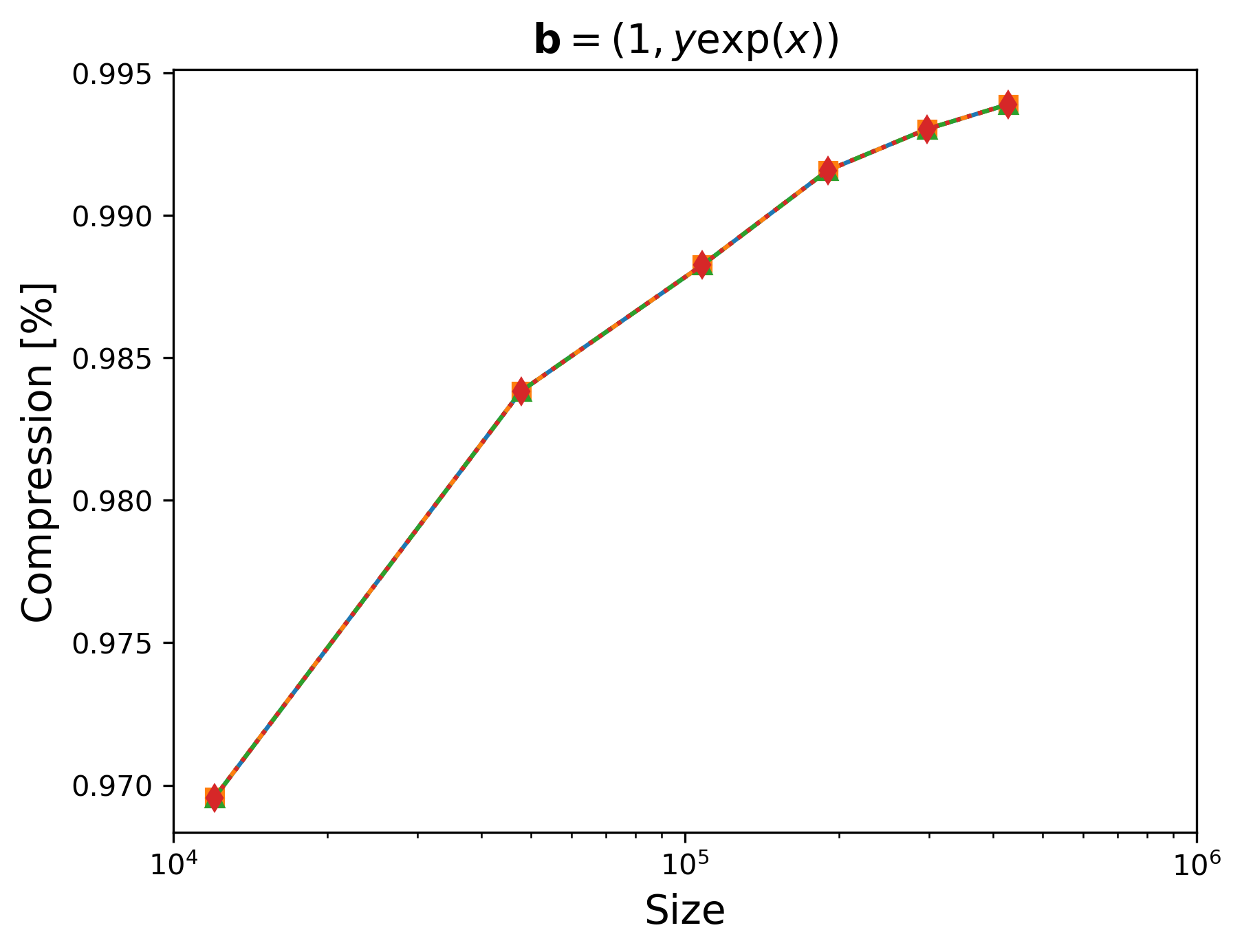}
        \caption{$\bb = (1, y\exp(x))$}
        \label{fig:err(1,yexp(x))}
    \end{subfigure}
\caption{Compression of $L_{\mathcal H} U_{\mathcal H}$ for FreeFem matrices of problem~\eqref{Pb:Neumann} for different convection $\bb$.}
    \label{fig:comprLUneumann}
\end{figure}
\begin{figure}
    \centering
    \begin{subfigure}[b]{0.45\linewidth}
        \centering
        \includegraphics[width=\linewidth,height=0.55\linewidth, trim = 0 0 0 0, clip]{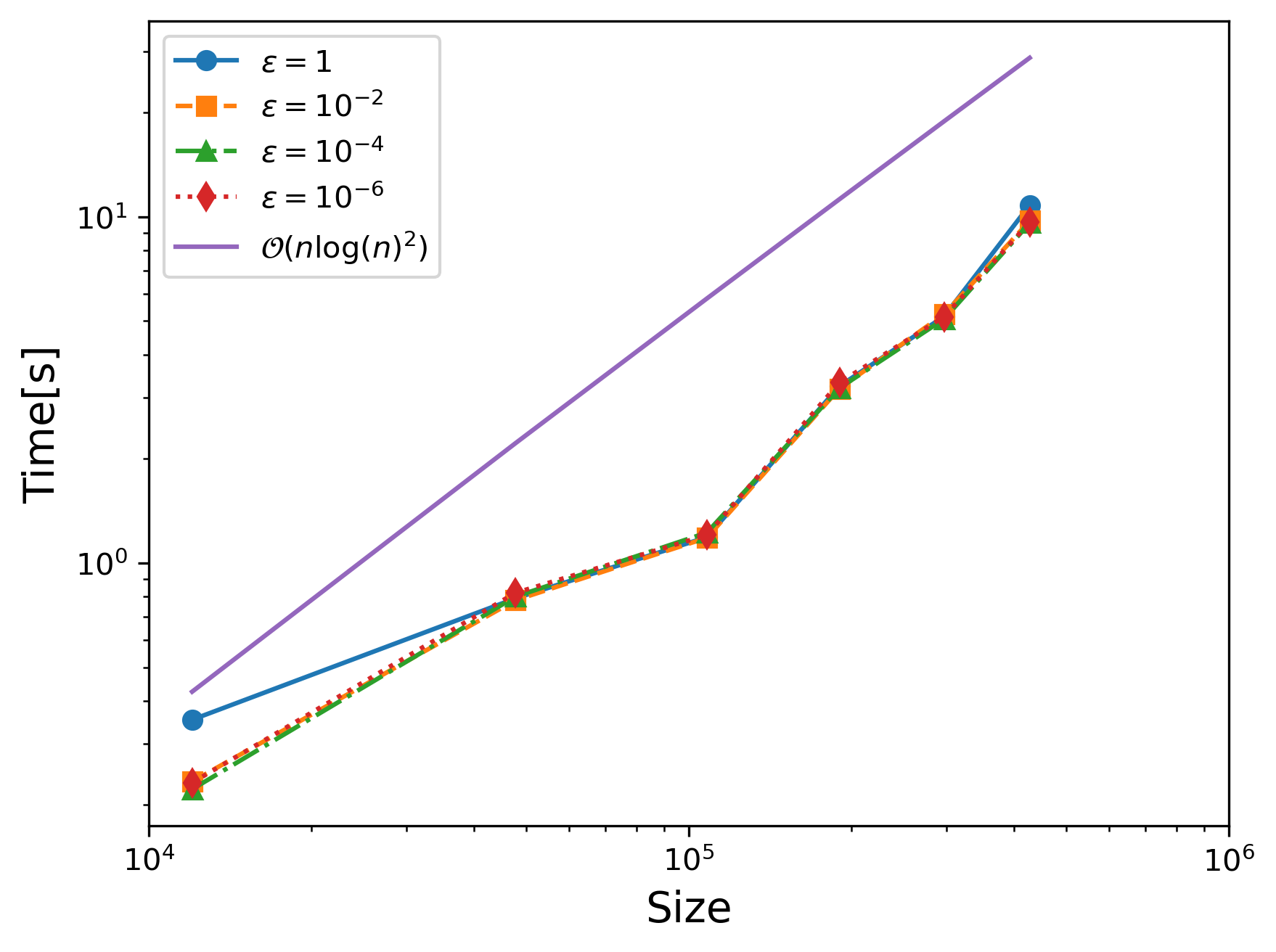}
        \caption{$\bb = (1, 0)$}
        \label{fig:err(1,0)}
    \end{subfigure}
    \begin{subfigure}[b]{0.45\linewidth}
        \centering
        \includegraphics[width=\linewidth,height=0.55\linewidth, trim = 0 0 0 0, clip]{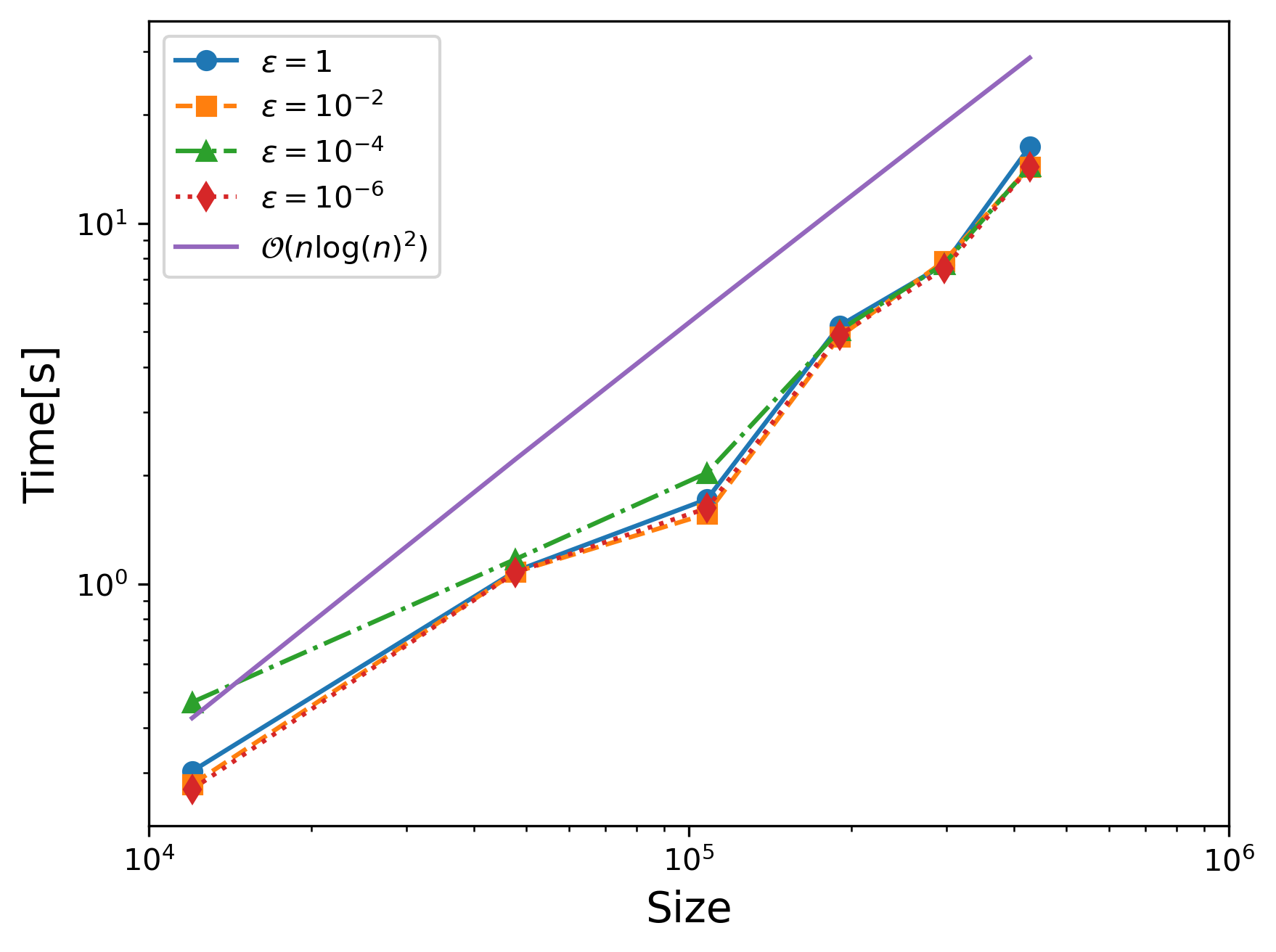}
        \caption{$\bb = (1, y\cos(x))$}
        \label{fig:err(1,ycos(x))}
    \end{subfigure}    
    \begin{subfigure}[b]{0.45\linewidth}
        \centering
        \includegraphics[width=\linewidth,height=0.55\linewidth, trim = 0 0 0 0, clip]{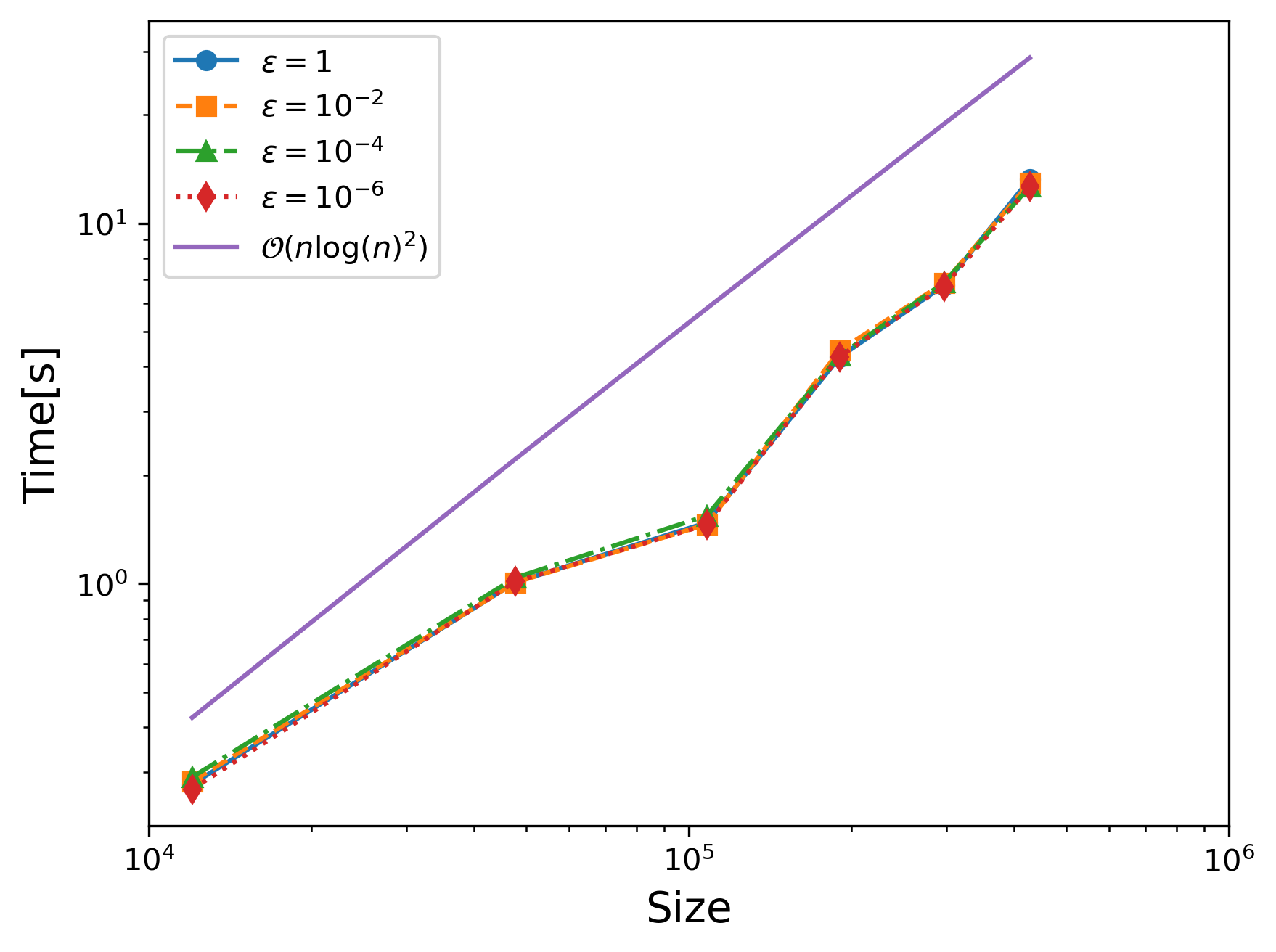}
        \caption{$\bb = (1, y\exp(x))$}
        \label{fig:err(1,yexp(x))}
    \end{subfigure}
    
    \caption{Time of the $\mathcal{H}$-LU factorization for FreeFem matrices of problem \eqref{Pb:Neumann} for different convection~$\bb$.}
    \label{fig:timeLUneumann}
\end{figure}

It appears that with our approach, using tube clusters, we manage to compute a hierarchical $LU$-factorization in a quasi-linear time and that the error induced in the resolution of the system $Ax=y$ is extremely small. More importantly, the diffusion parameter $\varepsilon$ doesn't seem to have a significant impact on the results, as intended for a P\'eclet robust method.

\section{Conclusion}
\label{section5}

In this work, we have investigated the use of hierarchical matrices for advection-diffusion problems in the advection-dominated regime. While the standard hierarchical approaches are well established for elliptic problems, their extension to non-symmetric and convection-dominated settings remains an open challenge. Our contribution lies in identifying a specific class of hierarchical clustering, namely tube clusters, that allow us to extend key theoretical results from the elliptic case to advection-diffusion problems.

Starting with the constant coefficient case, we establish the following key estimate: let $\omega \subsetneq \omega_\delta$ be two nested domains, and let $\eta$ be a cut-off function supported in $\omega_\delta$. Then the solution $u$ to the variational problem \eqref{prob_var} satisfies:
\[ \varepsilon \Vert (\nabla  \eta u )\Vert_{\mL^2(\omega)}^2 \leq   a(u, \eta^2 u ) + \int_{\omega_\delta} \eta u ^2 \bb \cdot \nabla \eta d\bx + \varepsilon \Vert u\nabla \eta \Vert_{\mL^2(\omega_\delta)}^2 .\]

By carefully designing clusters aligned with the advection field, we ensured that the integral term disappears, leading to a Péclet-robust Caccioppoli estimate. This crucial step allowed us to establish a hierarchical approximation theory comparable to that of B\"orm(\cite{MR2606959}) for elliptic problems, but adapted to advection-dominated settings.

We then extended this result to more general advection fields, considering $\bb\in C^1(\mathbb{R}^d)$ with $\bb$ nowhere vanishing. This extension significantly broadens the applicability of our method beyond simple constant-coefficient cases, making it relevant for a wide range of physical and engineering problems. Since the feasibility of our approach depends on the ability to construct tube clusters efficiently, we introduced a deformation technique that transforms the computational domain into a space where the advection streamlines become straight. This reformulation provides a natural way to design hierarchical partitions aligned with the physics of the problem.

To validate our theoretical findings, we conducted numerical experiments on various advection fields, analyzing the impact of the diffusion parameter $\varepsilon$ under mesh refinement, for both Dirichlet and Neumann boundary conditions. The results confirmed the robustness of our approach: the time complexity of the $LU$ factorization and the accuracy of the solution remained independent of $\varepsilon$, demonstrating that our hierarchical method does not suffer from the usual breakdowns associated with high P\'eclet numbers.

Overall, our findings suggest that for advection-dominated problems, it is possible to construct hierarchical approximations of the $LU$ factorization using tube cluster trees, achieving comparable error and computational costs to those obtained in the elliptic setting. By carefully incorporating the underlying physics into the hierarchical framework, we have significantly extended the range of applicability of hierarchical matrices. Further investigations are needed to extend this approach to more complex geometries, such as domains with holes, where the interaction structure may significantly differ. Moreover, in our study, the minimum leaf size of the cluster tree is dictated by the characteristic length of the flow. A natural direction for future work is to explore the implications of allowing deeper refinements in the tree and to assess how this affects the hierarchical representation and compression properties.

\nocite{*}
\bibliography{draft}
\bibliographystyle{elsarticle-num-names}

\end{document}